%% file: lcm.tex
\documentclass[reqno,letterpaper]{amsart}
\newif\ifcref\creftrue
\input{decls}
\usepackage{cmll,cancel,version}
\usepackage{microtype}

\def\aandish{strong\xspace}
\def\Aandish{Strong\xspace}

\def\aprime{$\aor$-prime\xspace}
\def\mprime{$\mor$-prime\xspace}
\def\aint{$\aor$-integral\xspace}
\def\mint{$\mor$-integral\xspace}
\def\amax{$\aor$-maximal\xspace}
\def\mmax{$\mor$-maximal\xspace}
\def\afield{$\aor$-field\xspace}
\def\mfield{$\mor$-field\xspace}

\def\atotal{$\aor$-total\xspace}
\def\mtotal{$\mor$-total\xspace}
\def\atotal{total\xspace}
\def\mtotal{linear\xspace}
\def\mtotality{linearity\xspace}

\def\alocated{$\aor$-located\xspace}
\def\mlocated{$\mor$-located\xspace}

\newcommand{\opl}{{\smash{\mathring{L}}}}
\newcommand{\cll}{{\smash{\overline{L}}}}
\newcommand{\opu}{{\smash{\mathring{U}}}}
\newcommand{\clu}{{\smash{\overline{U}}}}

\newcommand{\lint}{\sqSubset}

\renewcommand{\iint}{\ll}
\newcommand{\niint}{\not\ll}
\newcommand{\apart}{\bowtie}
\newcommand{\intr}{\mathsf{int}}
\newcommand{\cl}{\mathsf{cl}}
\let\close\approx

\newcommand{\umplus}{\mathbin{\raisebox{1pt}{$\scriptstyle\boxbot$}}}
\newcommand{\lmplus}{\mathbin{\raisebox{1pt}{$\scriptstyle\boxtop$}}}

\newcommand{\msup}{{\textstyle\sup^{\smallmor}}}

\def\gg{^{\mathrm{N}}}
\def\gr{^{\mathrm{G}}}

\def\ch{^{\{\}}}
\def\ev{\mathsf{ev}}

\title{Affine logic for constructive mathematics}
\author{Michael Shulman}
\address{Department of Mathematics, University of San Diego, San Diego, CA, 92110, USA}
\email{shulman@sandiego.edu}
\thanks{This material is based on research
    sponsored by The United States Air Force Research Laboratory under
    agreement numbers FA9550-15-1-0053, FA9550-16-1-0292, and FA9550-21-1-0009.  The
    U.S. Government is authorized to reproduce and distribute reprints
    for Governmental purposes notwithstanding any copyright notation
    thereon.  The views and conclusions contained herein are those of
    the author and should not be interpreted as necessarily
    representing the official policies or endorsements, either expressed
    or implied, of the United States Air Force Research Laboratory, the
    U.S. Government, or Carnegie Mellon University.}
\date{\today}
\let\myfrac\frac
\def\frac#1#2{{\textstyle\myfrac{#1}{#2}}}
\def\bSet{\mathbf{Set}}
\def\ASet{\mathfrak{A}\mathbf{Set}}
\def\cCat{\mathcal{C}\mathit{at}}
\def\cInt{\mathcal{I}\mathit{nt}}
\def\cAff{\mathcal{A}\mathit{ff}}

\def\sfeq{\mathsf{eq}}
\let\mand\boxtimes
\let\aand\sqcap
\let\aor\sqcup
\def\mor{\mathbin{\raisebox{-1pt}{\rotatebox{45}{$\scriptstyle\boxtimes$}}}}
\def\smallmor{\mathbin{\raisebox{-1pt}{\rotatebox{45}{$\scriptscriptstyle\boxtimes$}}}}
\def\fa{{\textstyle\bigsqcap}}
\def\ex{{\textstyle\bigsqcup}}
\let\imp\multimap
\let\liff\multimapboth

\def\one{\mathbf{1}}
\def\zero{\mathbf{0}}
\def\nt#1{\smash{{#1}^\perp}}
\def\ntp#1{\smash{{(#1)}^\perp}}
\def\ntP#1{{\Big(#1\Big)}^\perp}
\def\ntpp#1{{\left(#1\right)}^\perp}
\def\ntnt#1{\smash{{#1}^{\perp\perp}}}
\let\ltens\boxtimes
\def\cm#1{{#1}^\perp}
\def\cmcm#1{{#1}^{\perp\perp}}
\def\cmp#1{{(#1)}^\perp}
\def\qeq{\mathrel{\smash{\overset{\raisebox{-1pt}{\tiny ?}}{=}}}}
\def\defeq{\mathrel{\smash{\overset{\scriptscriptstyle\mathsf{def}}{=}}}}
\let\logeq\equiv
\def\nlogeq{\not\logeq}
\def\qlogeq{\mathrel{\smash{\overset{\raisebox{-1pt}{\tiny ?}}{\logeq}}}}
\def\pf#1{{#1}^+}
\def\pfp#1{{(#1)}^+}
\def\rf#1{{#1}^-}
\def\rfp#1{{(#1)}^-}
\def\chu{\mathrm{Chu}}
\def\chuh{\bH_{\pm}}

\def\chuzeroone{\{\zero,\one\}_{\pm}}

\let\P\sP
\let\luk\L
\let\L\fA
\let\I\fI

\let\types\vdash

\def\ochat{\hat{\oc}\hspace{.5pt}}
\def\wnchk{\check{\wn}\hspace{.5pt}}
\def\nle{\not\le}

\def\ieq{=}
\def\ineq{\neq}
\def\leq{\circeq}
\def\lneq{\not\circeq}
\def\lin{\mathrel{\mathrlap{\sqsubset}{\mathord{-}}}}
\def\lnin{\not\mathrel{\mathrlap{\sqsubset}{\mathord{-}}}}
\let\lleq\sqsubseteq
\let\lle\sqsubseteq
\let\llt\sqsubset
\let\lsub\sqsubseteq
\def\lnsub{\mathrel{\not\sqsubseteq}}

\let\lnle\lnleq

\let\iff\leftrightarrow
\let\lcap\aand
\let\lcup\aor
\def\lmcap{\mathbin{\hat{\mand}}}
\def\lmcup{\mathbin{\check{\mor}}}
\let\lbigcap\fa
\let\lbigcup\ex
\def\lempty{\mathord{\not\hspace{-2pt}{\scriptstyle\Box}}}
\def\lmeq{\overset{\scriptscriptstyle\mand}{=}}
\def\imneq{\not\lmeq}
\def\inv{\mathsf{inv}}

\def\ocin{:}

{\catcode`\|=\active
  \xdef\lset{\protect\expandafter\noexpand\csname lset \endcsname}
  \expandafter\gdef\csname lset \endcsname#1{\mathinner
        {\lbrace\,{\mathcode`\|32768\let|\midvert #1}\,\rbrace}}
  \xdef\lsetof{\protect\expandafter\noexpand\csname lsetof \endcsname}
  \expandafter\gdef\csname lsetof \endcsname#1{\left\llbracket%
     \ifx\SavedDoubleVert\relax \let\SavedDoubleVert\|\fi
     \:{\let\|\SetDoubleVert
     \mathcode`\|32768\let|\SetVert
     #1}\:\right\rrbracket}
}

{\catcode`\|=\active
  \xdef\ocset{\protect\expandafter\noexpand\csname ocset \endcsname}
  \expandafter\gdef\csname ocset \endcsname#1{\mathinner
        {\lbrace\,{\mathcode`\|32768\let|\midvert #1}\,\rbrace}}
  \xdef\ocsetof{\protect\expandafter\noexpand\csname ocsetof \endcsname}
  \expandafter\gdef\csname ocsetof \endcsname#1{\oc\!\left\{%
     \ifx\SavedDoubleVert\relax \let\SavedDoubleVert\|\fi
     \:{\let\|\SetDoubleVert
     \mathcode`\|32768\let|\SetVert
     #1}\:\right\}}
}

\begin{document}
\maketitle

\begin{abstract}
  We show that numerous distinctive concepts of constructive mathematics arise automatically from an ``antithesis'' translation of affine logic into intuitionistic logic via a Chu/Dialectica construction.
  This includes apartness relations, complemented subsets, anti-subgroups and anti-ideals,
  strict and non-strict order pairs, cut-valued metrics, and apartness spaces.
  We also explain the constructive bifurcation of some classical concepts using the choice between multiplicative and additive affine connectives.
  Affine logic and the antithesis construction thus systematically ``constructivize'' classical definitions, handling the resulting bookkeeping automatically.
\end{abstract}


\section{Introduction}
\label{sec:introduction}

One of the explicit motivations of Girard's linear logic~\cite{girard:ll} was to recover an involutory ``classical'' negation while retaining ``constructive content'':
\begin{quote}
  \small \dots the linear negation \dots is a constructive and involutive negation; by the way, linear logic works in a classical framework, while being more constructive than intuitionistic logic.  \cite[p3]{girard:ll}
\end{quote}
One might therefore expect that over the past three decades some practicing constructive mathematicians would have adopted linear logic instead of intuitionistic logic;%
\footnote{I will use ``intuitionistic'' to refer to the formal logic codified by Heyting, and ``constructive'' for the programme of doing mathematics with ``constructive content''.
This is unfaithful to the original philosophical meaning of ``intuitionistic'', but for better or for worse the phrase ``intuitionistic logic'' has come to refer to Heyting's logic, and I have been unable to think of a satisfactory alternative.
I will use ``classical'' for classical mathematics and classical nonlinear logic, and ``linear'' (resp.\ ``affine'') for the ``classical'' form of linear (resp.\ affine) logic that has an involutive negation.}
but this does not seem to be the case.
One might conjecture many reasons for this.
However, I will instead argue that linear logic has nevertheless been present \emph{implicitly} in  constructive mathematics, going back all the way to Brouwer.

Specifically, I will show that there are aspects of constructive mathematical practice that are better explained by linear logic than by intuitionistic logic.
The non-involutory intuitionistic negation often leads constructive mathematicians to study both a classical concept and its formal De Morgan dual, such as equality and apartness, subgroups and antisubgroups, topological spaces and apartness spaces, and so on.
We will show that such ``dual pairs of propositions'' can be regarded as single propositions in a model of \emph{linear} logic, or more specifically \emph{affine} logic, which we call the \emph{antithesis model}.
Notions such as apartness relations then arise by writing a classical definition in affine logic and interpreting it in this model.

The antithesis model is a special case of a Chu or Dialectica construction~\cite{chu:construction,paiva:dialectica-chu,shulman:dialectica} (the two constructions coincide in this special case), applied to the algebra of intuitionistic propositions.
It is well-known that Chu and Dialectica constructions yield models of linear and affine logic (see for instance~\cite{barr:staraut-ll,depaiva:dialectica,oliva:dialectica-ll}); the novelty here is in how the logic of this particular special case relates to constructive mathematics.

Since it constructs a model of affine logic from any model of intuitionistic logic in a functorial way, the antithesis model can also be used as a purely syntactic translation, transforming any definition, theorem, or proof in affine logic into one in intuitionistic logic; we call this the \emph{antithesis translation}.
This is analogous to other translations such as the G\"odel--Gentzen double-negation translation, which constructs a model of classical logic from a model of intuitionistic logic, and therefore transforms classical theorems into intuitionistic ones; or the Girard translation, which constructs a model of intuitionistic logic from a model of linear logic, and therefore transforms intuitionistic theorems into linear ones.

Importantly, however, unlike the G\"odel--Gentzen and Girard translations, the antithesis translation is not conservative.
Indeed, to a linear logician, the antithesis model looks quite degenerate, particularly in the behavior of its exponentials.
Thus, we should not view the antithesis translation as an ``explanation'' or ``embedding'' of affine logic into intuitionistic logic.
Rather, we view it as giving a way to treat affine logic as a ``high-level'' or ``domain-specific'' language that can be ``compiled'' into intuitionistic logic.
(Because the antithesis translation is a one-sided inverse of the Girard translation, we can in fact view affine logic as a strict \emph{extension} of intuitionistic logic: the high-level language includes an ``escape to assembler''.)
In other words, the antithesis translation is a tool primarily for the intuitionistic mathematician, not the linear one.

This tool has several possible uses.
Firstly, it formalizes a technique for ``constructivizing'' classical definitions: write them in affine logic and apply the antithesis translation.
This method often yields a better result than the usual one of simply regarding the classical connectives as having their intuitionistic meanings; the latter frequently requires manual ``tweaking'' to become intuitionistically sensible.

We also obtain a uniform explanation for some instances of the fact that ``constructivizing'' is multi-valued.
Namely, in affine logic the connectives ``and'' and ``or'' bifurcate into ``additive'' and ``multiplicative'' versions.
Thus, many classical definitions can be written in affine logic in more than one way, by making different choices about whether to interpret the classical connectives as additive or multiplicative affine ones.
Under the antithesis translation, this then leads to different intuitionistic versions of a classical definition, many of which occur naturally in examples and have been already written down ``manually'' by constructive mathematicians.
(There are, of course, also other reasons for the constructive multifurcation of concepts.)

Roughly speaking, the additive disjunction ``$P$ or $Q$'' corresponds, under the antithesis translation, to the intuitionistic disjunction; while the multiplicative disjunction ``$P$ \emph{par} $Q$'' corresponds to the intuitionistic pattern ``if not $P$, then $Q$; and if not $Q$, then $P$'' that is often used constructively when the intuitionistic disjunction is too strong.
For instance, in intuitionistic logic the rational numbers are a field in the strong ``geometric'' sense that every element is either zero or invertible.
The real numbers are not a field in this sense, but they are a field in the weaker ``Heyting'' sense that every nonzero\footnote{Here ``nonzero'' means ``apart from zero''.} element is invertible and every noninvertible element is zero.
This condition is the image under the antithesis translation of the affine statement that every element is either zero \emph{par} invertible.
Similarly, for real numbers $x\le y$ is not equivalent to ``$x=y$ or $x<y$'', but it \emph{is} equivalent to the antithesis translation of ``$x=y$ \emph{par} $x<y$''.
In \crefrange{sec:reals}{sec:topology} we will see that a systematic use of \emph{par} can solve a few tricky problems in intuitionistic constructive mathematics, such as defining a notion of ``metric space'' that includes the Hausdorff metric, or a union axiom for a ``closure space'' that is not unreasonably strong.

Secondly, we can also apply the antithesis translation to proofs.
Many classical proofs are also affinely valid with little change; surprisingly often, the lack of contraction is not a problem when definitions are formulated appropriately.
Hence, the antithesis translation turns such proofs into intuitionistic proofs of theorems involving apartness relations, antisubgroups, and so on: the process of ``turning everything around'' to deal with such concepts can be automated.
This tends to work for classical proofs that may use \emph{proof by contradiction} (or equivalently the \emph{law of double negation}) as long as they avoid the \emph{law of excluded middle}.
In intuitionistic logic, the laws of double negation and excluded middle are equivalent, but linear and affine logic disentangle (some versions of) them.

Having an automatic way to produce intuitionistic definitions and proofs is more than just a convenience: it can prevent or correct mistakes.
Working explicitly with apartness relations and their ilk is tedious and error-prone: it's easy to omit one of the contrapositive conditions, or forget to check that a function is strongly continuous or that a subset is strongly extensional.
Moreover, it's not always obvious exactly what the axioms on an apartness structure should be; but the antithesis translation always seems to give the right answer (or at least \emph{a} right answer).

A third possible use of the antithesis translation is more speculative.
Rather than viewing affine logic and the antithesis translation as tools for doing intuitionistic constructive mathematics, one might imagine instead a constructive mathematics (in the informal sense of ``mathematics with constructive content'') that uses \emph{only} affine logic.
The antithesis interpretation would then be a \emph{guide} to the correct way to formulate concepts in affine constructive mathematics.
It is not yet clear how feasible this idea is,\footnote{For instance, can one really formulate mathematics sufficiently carefully that it can \emph{all} be done in affine logic, without contraction?
  Is there a form of affine \emph{dependent} type theory that would be appropriate for such a mathematics?}
but we will make some remarks about it in \cref{sec:lcm}.


In this paper, we will mainly focus on the first use: translating definitions.
We include a few proofs, but for the most part we leave the development of ``affine constructive mathematics'' for future work.

\subsection*{Outline}
\label{sec:outline}

In \cref{sec:meaning} we describe our viewpoint on affine logic informally, analogously to the BHK interpretation of intuitionistic logic; no prior familiarity with affine or linear logic is required.
Then in \cref{sec:chu,sec:types} we formalize this interpretation both semantically, as a Chu or Dialectica construction, and syntactically, as a translation between propositional, first-order, and higher-order logics.
Many introductions to linear logic present it as a logic of ``resources'' or ``games''; we view it as a logic of \emph{mathematics}, like intuitionistic logic and classical logic, which is designed to be ``constructive'' in a different way than intuitionistic logic.

The rest of the paper consists of ``case studies'', showing that rewriting classical definitions directly in affine logic and passing across this antithesis translation yields well-known notions in intuitionistic constructive mathematics.
In \cref{sec:sets,sec:linear-sets} we treat sets and functions, then algebra in \cref{sec:algebra}, order relations in \cref{sec:posets}, real numbers in \cref{sec:reals}, and topology in \cref{sec:topology}.
Finally, in \cref{sec:lcm} we speculate a bit about how one might motivate and explain an ``affine constructive mathematics'' on purely philosophical grounds. 

\subsection*{Acknowledgments}
\label{sec:acknowledgments}

My understanding of constructive mathematics and linear logic has been greatly influenced by Toby Bartels, Mart\'{i}n Escard\'{o}, Andrej Bauer, Dan Licata, Valeria de Paiva, Todd Trimble, and Peter LeFanu Lumsdaine.
I am also grateful to the referees for stimulating comments and suggestions, and for suggesting the word ``antithesis''.

\section{A meaning explanation}
\label{sec:meaning}

Intuitionistic logic is often explained informally (e.g.\ in~\cite{tvd:constructivism-i}) by the so-called Brouwer--Heyting--Kolmogorov (BHK) interpretation, which explains the meaning of the logical connectives and quantifiers ``pragmatically'' in terms of what counts as a proof of them.
For instance, a few of the rules are:
\begin{itemize}
\item A proof of $P\to Q$ is a method converting any proof of $P$ into a proof of $Q$.
\item A proof of $P\lor Q$ is either a proof of $P$ or a proof of $Q$.
\item A proof of $\neg P$ is a method converting any proof of $P$ into a proof of an absurdity.
\end{itemize}
This leads to the rules of intuitionistic logic; for instance, we cannot prove $P\lor \neg P$ in general since we cannot decide whether to give a proof of $P$ or a proof of $\neg P$.

Practicing constructive mathematicians, however, have found that it is often not sufficient to know what counts as a proof of a statement: it is often just as important, if not more so, to know what counts as a \emph{refutation} of a statement.
For instance, while it is of course essential to know that two real numbers are equal if they agree to any desired degree of approximation, it is also essential to know that they are \emph{unequal} if there is some finite degree of approximation at which they disagree.
If real numbers are defined using Cauchy sequences $x,y:\dN\to\dQ$ with specified rate of convergence $|x_n-x_m|< \frac1n+\frac1m$, then we want to separately define
\begin{align*}
  (x=y) &\defeq \forall n. |x_n-y_n| \le \textstyle\frac2n\\
  (x\neq y) &\defeq \exists n. |x_n-y_n| >\textstyle \frac2n.
\end{align*}

In classical logic, a refutation of $P$ means a proof of $\neg P$, and these two definitions are each other's negations.
But intuitionistic negation is not involutive, and this ``$x\neq y$'' is not the logical negation of $x=y$.
Thus, when constructive mathematics is done in intuitionistic logic --- as it usually is --- we must define inequality as a new \emph{apartness relation} with which the set of reals is equipped.
In Bishop's words:
\begin{quote}
  \small It is natural to want to replace this negativistic definition [the logical negation of equality] by something more affirmative\dots Brouwer himself does just this for the real number system, introducing an affirmative and stronger relation of inequality\dots Experience shows that it is not necessary to define inequality in terms of negation.  For those cases in which an inequality relation is needed, it is better to introduce it affirmatively\dots
  \cite[p10]{bb:constr-analysis}
\end{quote}

Similar things happen all throughout constructive mathematics.
In addition to knowing when an element is in a subgroup, we need to know when an element is not in a subgroup; thus we introduce \emph{antisubgroups} (and similarly anti-ideals, etc.).
In addition to knowing when a point is in the interior of a set, we need to know when it is in the exterior of a set; thus we introduce \emph{apartness spaces}~\cite{bv:apartness}.

To repeat, the problem is that the BHK interpretation and resulting intuitionistic logic privilege \emph{proofs} over \emph{refutations}.
In the words of Patterson~\cite{patterson:thesis}:
\begin{quotation}
  \small Once we take on the Brouwerian view that proofs should be constructions, both negation and ``falsity'' disappear because absurdity is not the same thing as demonstratively false.  This is because a construction leading to a contradiction does not mean that we can provide a counterexample.\dots

  In intuitionistic logic we have taken ``true'' to be primitive as well as ``absurdity''.\dots Thus, a ``proof leading to absurdity'' is a derived notion of falsity and the only one afforded to us in intuitionistic logic.

  Negative information can be just as constructive as positive information.\dots the correct way to use negative information in a constructive setting would be to do the ``opposite'' or ``backward'' construction in some way.~\cite[p8--9]{patterson:thesis}
\end{quotation}

This suggests a BHK-like explanation of logical connectives in terms of \emph{both} what counts as a proof and what counts as a refutation.
We now explore what such an explanation might look like.
The only requirement we impose is that no formula should be both provable and refutable
(but see \cref{rmk:00,rmk:constr-duality}).

We start with the following explanations of conjunction and disjunction, which we denote $\aand$ and $\aor$ rather than $\land$ and $\lor$ as a warning that they will not behave quite like the usual intuitionistic or classical connectives.
\begin{itemize}
\item A proof of $P\aand Q$ is a proof of $P$ together with a proof of $Q$.
\item A refutation of $P\aand Q$ is either a refutation of $P$ or a refutation of $Q$.
\item A proof of $P\aor Q$ is either a proof of $P$ or a proof of $Q$.
\item A refutation of $P\aor Q$ is a refutation of $P$ together with a refutation of $Q$.
\end{itemize}
These ``proof'' clauses are the usual BHK ones, while the ``refutation'' clauses are natural-seeming De Morgan duals.
The most natural clauses for negation are:
\begin{itemize}
\item A proof of $\nt P$ is a refutation of $P$
\item A refutation of $\nt P$ is a proof of $P$.
\end{itemize}
This negation is \emph{involutive}, $\ntnt P \logeq P$, with strict De Morgan duality: $\ntp{P\aand Q} \logeq \nt P \aor \nt Q$ and $\ntp{P\aor Q} \logeq \nt P \aand \nt Q$.
($\logeq$ denotes inter-derivability.)

A little thought suggests that one natural explanation of implication (which we indulge in some foreshadowing by writing as $\imp$) is:
\begin{itemize}
\item A proof of $P\imp Q$ is a method converting any proof of $P$ into a proof of $Q$, \emph{together with} a method converting any refutation of $Q$ into a refutation of $P$.
\item A refutation of $P\imp Q$ is a proof of $P$ together with a refutation of $Q$.
\end{itemize}

\begin{rmk}
  This does not mean that when we prove an implication we must also prove its contrapositive explicitly.
  The ``proofs'' in these explanations, like those in the BHK interpretation, are not the ``proofs'' that a mathematician writes in a paper (or even formalizes on a computer).
  Rather they are ``verifications'', ``fully normalized proofs'', or ``data that must be extractable from a proof''.
  Not every intuitionistic proof (in the ordinary sense) of $P\lor Q$ begins by deciding whether to prove $P$ or $Q$, but intuitionistic logic satisfies the ``disjunction property'' that from any proof of $P\lor Q$ in the empty context we can \emph{extract} either a proof of $P$ or a proof of $Q$.
  Similarly, any proof of $P\imp Q$ in the empty context must contain enough \emph{information} to transform refutations of $Q$ into refutations of $P$ as well as proofs of $P$ into proofs of $Q$.
\end{rmk}

Building contraposition into the definition of implication makes it unsurprising that we get $(P \imp Q) \logeq (\nt Q \imp \nt P)$.
So it might seem that we are going to fall into \emph{classical} logic, but this is not the case.
For instance, classically we have $\neg (P \to Q) \logeq P \land \neg Q$; but despite the apparent presence of this law in the ``refutation'' clause for $\imp$, we do \emph{not} have $\ntp{P\imp Q} \qlogeq P \aand \nt Q$.
Instead we have $\ntp{P\imp Q} \logeq P \mand \nt Q$, where $\mand$ is a different kind of conjunction:
\begin{itemize}
\item A proof of $P\mand Q$ is a proof of $P$ together with a proof of $Q$.
\item A refutation of $P\mand Q$ is a method converting any proof of $P$ into a refutation of $Q$, together with a method converting any proof of $Q$ into a refutation of $P$.
\end{itemize}
Note that $P\aand Q$ and $P\mand Q$ have the same \emph{proofs}, but different \emph{refutations}.
Both refutation clauses are based on the idea that $P$ and $Q$ cannot both be true, but to refute $P\aand Q$ we must specify \emph{which} of them fails to be true, whereas to refute $P\mand Q$ we simply have to show that if one of them is true then the other cannot be.

This suggests that $P\mand Q$ is stronger than $P\aand Q$, and in fact we can justify $(P\mand Q) \imp (P\aand Q)$ on the basis of our informal explanations.
Since $P\mand Q$ and $P\aand Q$ have the same proofs, it suffices to transform any refutation of $P\aand Q$ into a refutation of $P\mand Q$.
The former is either a refutation of $P$ or of $Q$; without loss of generality assume the latter.
Then we can certainly produce a refutation of $Q$ that doesn't even need to use a proof of $P$.
On the other hand, given a refutation of $Q$ it is impossible that we could also have a proof of $Q$; so by \emph{ex contradictione quodlibet} from any proof of $Q$ we can vacuously produce a refutation of $P$.

It follows that the De Morgan dual $P \mor Q \defeq \ntp{\nt P \mand \nt Q}$ of $\mand$ is a different kind of disjunction that is \emph{weaker} than $\aor$.
(For a discussion of notation, see \cref{rmk:lines}.)
Its explanation is:
\begin{itemize}
\item A proof of $P\mor Q$ is a method converting any refutation of $P$ into a proof of $Q$, together with a method converting any refutation of $Q$ into a proof of $P$.
\item A refutation of $P\mor Q$ is a refutation of $P$ together with a refutation of $Q$.
\end{itemize}
Thus $P\mor Q$ has the same \emph{refutations} as $P\aor Q$, but more proofs: while $P\aor Q$ supports proof by cases, $P\mor Q$ supports only the disjunctive syllogism.
As noted in \cref{sec:introduction}, $P\mor Q$ encapsulates a common constructive pattern for weakening definitions when the intuitionistic ``or'' is too strong: rather than asserting that one of two conditions holds, we assert that if either one of two conditions fails then the other must hold.
We have $(P\imp Q) \logeq (\nt P \mor Q)$, a version of the classical law $(P \to Q) \logeq (\neg P \lor Q)$.

A reader familiar with linear logic may recognize $\mand$ and $\mor$ as its \emph{multiplicatives}, while $\aand$ and $\aor$ are its \emph{additives},\footnote{The origin of the terminology is apparently the fact that the distributive law in linear logic is $P \mand (Q\aor R) \logeq (P\mand Q) \aor (P\mand R)$, i.e.\ ``multiplication distributes over addition''.} with $\imp$ as its \emph{linear implication}.
Indeed, this explanation is similar to the ``game semantics'' of linear logic, in which a ``proposition'' is regarded as a game or interaction between a ``prover'' and a ``refuter''.

Actually, our explanation justifies not fully general linear logic but \emph{affine logic}, because in the nullary case (``true'' and ``false'') the distinctions collapse:
\begin{itemize}
\item There is exactly one proof of $\top$.
\item There is no refutation of $\top$.
\item There is no proof of $\bot$.
\item There is exactly one refutation of $\bot$.
\end{itemize}
These are units for both additive and multiplicative connectives: $P \aand \top \logeq P \mand \top \logeq P$ and $P \aor \bot \logeq P \mor \bot \logeq P$.
The most nontrivial part of this is the refutations of $P\mand \top$ (or dually the proofs of $\nt P\mor \bot$), which by definition consist of a method transforming any proof of $\top$ into a refutation of $P$, together with a method transforming any proof of $P$ into a refutation of $\top$.
The former is essentially just a refutation of $P$; but given this, there can be no proof of $P$, so the latter method is vacuous.

The quantifiers are essentially additive; we write them as $\ex/\fa$ instead of $\exists/\forall$.
\begin{itemize}
\item A proof of $\ex x. P(x)$ is a value $a$ together with a proof of $P(a)$.
\item A refutation of $\ex x. P(x)$ consists of a refutation of $P(a)$ for an arbitrary $a$.
\item A proof of $\fa x. P(x)$ consists of a proof of $P(a)$ for an arbitrary $a$.
\item A refutation of $\fa x. P(x)$ is a value $a$ together with a refutation of $P(a)$.
\end{itemize}
The most novel of these clauses is the one for refutations of $\fa x. P(x)$: just as a constructive \emph{proof} of an existence statement should supply a witness, we stipulate that a constructive \emph{disproof} of a universal statement should supply a counterexample.
This yields De Morgan dualities
\begin{mathpar}
  \ntP{\ex x.P(x)} \logeq \fa x. \nt{P(x)}\and
  \ntP{\fa x.P(x)} \logeq \ex x. \nt{P(x)}
\end{mathpar}
and also ``Frobenius'' laws involving the \emph{multiplicative} connectives:
\begin{mathpar}
  P \mand \ex x.Q(x) \logeq \ex x. (P \mand Q(x)) \and
  P \mor \fa x.Q(x) \logeq \fa x. (P \mor Q(x)).
\end{mathpar}
But $P \aand \ex x.Q(x) \qlogeq \ex x. (P \aand Q(x))$ fails: a refutation of $\ex x. (P \aand Q(x))$ consists of, for every $a$, either a refutation of $P$ or a refutation of $Q(a)$; while a refutation of $P \aand \ex x.Q(x)$ must decide at the outset whether to refute $P$ or to refute all $Q(a)$'s.

This explanation of the connectives and quantifiers solves the problem mentioned above with equality and inequality of real numbers.
If we define
\[ (x=y) \defeq \fa n. |x_n-y_n| \le \textstyle\frac2n \]
then we find that (assuming that $\ntp{p\le q} \logeq (p > q)$ for $p,q\in\dQ$)
\[ \ntp{x=y} \logeq \ex n. |x_n-y_n| > \textstyle\frac2n. \]
Thus, the correct notions of equality and inequality for real numbers are each other's negations, relieving us of the need for a separate ``apartness relation''.

The names \emph{linear} and \emph{affine} logic refer to the fact that $\mand$ and $\mor$ are not idempotent: $P\mand P \nlogeq P$ and $P\mor P \nlogeq P$.
(More precisely, what fails are $P \imp (P\mand P)$ and $(P\mor P) \imp P$.)
A proof of $P\mor P$ consists of a method (well, technically two methods) for converting any refutation of $P$ into a proof of $P$.
Since $P$ cannot be both provable and refutable, this is equivalently a method showing that $P$ cannot be refuted --- which is, of course, different from saying that it can be proven.
Note that linear logic always satisfies the ``multiplicative law of excluded middle'' $P\mor \nt P$ and the ``multiplicative law of non-contradiction'' $\ntp{P\mand \nt P}$ (in fact, they are essentially the same statement).
We call a proposition \textbf{decidable} if it satisfies the \emph{additive} law of excluded middle $P\aor \nt P$, or equivalently the additive law of non-contradiction $\ntp{P\aand \nt P}$.
According to the above informal explanations, decidability means that we have either a proof of $P$ or a refutation of $P$.

It may help to understand $\mand$ if we write out the meaning of $(P\mand Q)\imp R$, which the reader can verify is equivalent to $P\imp (Q\imp R)$:
\begin{itemize}
\item A proof of $(P\mand Q)\imp R$ consists of methods for:
  \begin{itemize}
  \item converting any proofs of $P$ and $Q$ into a proof of $R$,
  \item converting any proof of $P$ and refutation of $R$ into a refutation of $Q$, and
  \item converting any proof of $P$ and refutation of $Q$ into a refutation of $R$.
  \end{itemize}
\item A refutation of $(P\mand Q)\imp R$ consists of a proof of $P$, a proof of $Q$, and a refutation of $R$.
\end{itemize}
More generally, a proof of $(P_1\mand P_2\mand \cdots\mand P_n)\imp R$ consists of ``all possible direct or by-contrapositive proofs'' that contradict one of the hypotheses.
By contrast:
\begin{itemize}
\item A proof of $(P \aand Q) \imp R$ consists of:
  \begin{itemize}
  \item A method converting any proofs of $P$ and $Q$ into a proof of $R$, and
  \item A method converting any refutation of $R$ into either a refutation of $P$ or a refutation of $Q$.
  \end{itemize}
\item A refutation of $(P \aand Q) \imp R$ consists of a proof of $P$, a proof of $Q$, and a refutation of $R$.
\end{itemize}
That is, when proving the (stronger) statement $(P \aand Q) \imp R$, the by-contrapositive direction must use $R$ to determine which of $P$ or $Q$ fails, whereas when proving $(P \mand Q) \imp R$ we are allowed to assume one of $P$ and $Q$ and contradict the other.

We define $(P\liff Q) \defeq (P\imp Q) \aand (Q\imp P)$, with the following meaning:
\begin{itemize}
\item A proof of $P\liff Q$ consists of methods for converting:
  \begin{itemize}
  \item any proof of $P$ into a proof of $Q$, and vice versa, plus
  \item any refutation of $P$ into a refutation of $Q$, and vice versa.
  \end{itemize}
\item A refutation of $P\liff Q$ consists of either:
  \begin{itemize}
  \item a proof of $P$ and a refutation of $Q$, or
  \item a refutation of $P$ and a proof of $Q$.
  \end{itemize}
\end{itemize}
Often in linear logic one defines $P\liff Q$ to be $(P\imp Q) \mand (Q\imp P)$ instead.
However, for us $(P\imp Q) \aand (Q\imp P)$ is preferable, due in part to its more informative disjunctive notion of refutation; see also \cref{eg:lin-set-omega,eg:linear-prel,sec:posets}.

Finally, following Girard~\cite{girard:ll} we introduce two unary connectives $\oc$ and $\wn$ called \emph{exponential modalities}, with the following meanings:
\begin{itemize}
\item A proof of $\oc P$ is a proof of $P$.
\item A refutation of $\oc P$ is a method converting any proof of $P$ into an absurdity.
\item A refutation of $\wn P$ is a refutation of $P$.
\item A proof of $\wn P$ is a method converting any refutation of $P$ into an absurdity.
\end{itemize}
These exponentials deal with the potential objection that not \emph{every} constructive proposition has a ``strong dual''.
For instance, not every set has an apartness relation.
But there is always the Heyting negation $\neg P \defeq (P\to\zero)$, and the propositions $\oc P$ are those whose refutations are the ``tautological'' ones of this form.
We will call a proposition in affine logic $P$ \textbf{affirmative} if $P\logeq \oc P$. 
For instance, a set in the antithesis model whose affine equality is affirmative will correspond to a set in intuitionistic logic equipped with the \emph{denial inequality}, $(x\ineq y) \defeq \neg (x\ieq y)$.

It is also common to encounter propositions that are the Heyting negation of their strong dual.
For instance, while real numbers do not satisfy $(x\ineq y) \qlogeq \neg (x\ieq y)$, they do satisfy $(x\ieq y) \logeq \neg (x\ineq y)$ (the inequality is \emph{tight}).
In the antithesis model these are the propositions with $P \logeq \wn P$, 
which we call \textbf{refutative}.

We can also understand $\oc$ by considering $\oc P \imp Q$.
Since $Q$ cannot be both provable and refutable, if we can transform proofs of $P$ into proofs of $Q$, then any refutation of $Q$ already entails the impossibility of a proof of $P$.
Thus, in proving $\oc P \imp Q$ the contrapositive direction is subsumed by the forwards direction, giving:
\begin{itemize}
\item A proof of $\oc P \imp Q$ is a method converting any proof of $P$ into a proof of $Q$.
\item A refutation of $\oc P \imp Q$ is a proof of $P$ together with a refutation of $Q$.
\end{itemize}
Unlike $\ntp{P\imp Q} \logeq (\nt Q \imp \nt P)$, we only have $\ntp{\oc P \imp Q} \logeq (\nt Q \imp \wn(\nt P))$.
Thus $\oc P$ is ``usable multiple times'' (since $\oc P \logeq \oc P \mand \oc P$) but ``not contraposable''.


Note the proofs of $\oc P \imp Q$ are just the ordinary BHK interpretation of $P\to Q$.
This foreshadows the Girard translation of intuitionistic logic into linear logic.

\begin{notn}\label{rmk:lines}
  Since we will be passing back and forth between intuitionistic and affine logic frequently, to minimize confusion I have tried not to duplicate any notations between the two contexts.
  As a mnemonic, our notations for affine connectives generally involve \emph{perpendicular} lines; thus we have $\mand,\aand,\aor,\mor,\fa,\ex,\top,\bot$\footnote{There is no uniformity in notation for linear logic.
    The most common notation for $\mand$ is $\otimes$, but $\owedge,\oslash,\circ,\&$ are also used; whereas $\mor$ has been denoted by $\parr,\oplus,\parallel,\bullet,\odot,\ovee,\Box,\sharp,*$.
    Notations for $\aand/\aor$ include $\&/\oplus$, $\land/\lor$, and $\times/+$.
    Our $\mand/\mor$ and $\aand/\aor$ visually represent De Morgan duality, do not clash with other standard notations that I know of, and are easily distinguishable.} in place of the intuitionistic $\land,\lor,\forall,\exists,\one,\zero$.
  We carry this principle over to non-logical symbols as well, writing $\lleq,\llt,\lin$ and so on in place of the intuitionistic $\le,<,\in$.

  The main exceptions are the affine implication $\imp$, the exponentials $\oc,\wn$, and equality.
  The symbols $\imp,\oc,\wn$ are associated strongly with linear logic, and sufficiently visually distinctive to need no mnemonic.
  And the intuitionistic $\ieq$ can't be made any more perpendicular, so we instead write $\leq$ for affine equality to evoke $\imp$.

  In the intuitionistic context, will always use ``slashed'' symbols such as $\ineq,\notin,\nle$ and so on to denote strong ``affirmative'' negations, rather than the weak logical negations such as $\neg(x\ieq y)$.
  In the affine context, the corresponding slashed symbols $\lneq,\lnin,\lnle$ will refer to the involutive affine negation: $(x\lneq y) \defeq \ntp{x\leq y}$.
\end{notn}

\section{The antithesis translation for propositional logic}
\label{sec:chu}

Like the BHK interpretation of intuitionistic logic, the explanation of the affine connectives and quantifiers in \cref{sec:meaning} is informal, and nonspecific about what constitutes a ``method''.
However, the \emph{relationship} between the two interpretations can be made precise, in the form of a ``translation'' of affine logic into intuitionistic logic.
This is analogous to the G\"odel--Gentzen double-negation translation of classical logic into intuitionistic logic and the Girard translations of intuitionistic logic into linear logic, and like them it has both a semantic and a syntactic side.

Consider the G\"odel--Gentzen translation, restricted to propositional logic for simplicity.
On the semantic side, this constructs a Boolean algebra from a Heyting algebra.
More generally, let $\bH$ be any \emph{bicartesian closed category}, meaning a cartesian closed category with finite coproducts; a Heyting algebra is the special case of a bicartesian closed poset.
We regard the objects of $\bH$ as intuitionistic propositions, and hence use logical notations for its structure: $\land$ for cartesian products, $\lor$ for coproducts, $\one$ and $\zero$ for terminal and initial object, and $\to$ for exponentials.

\begin{rmk}
  Constructive logics must always deal with the question of \emph{proof relevance}: whether we interpret a proposition to belong to a \emph{poset}, such as a Heyting algebra (the proof-irrelevant version), or a more general \emph{category} (the proof-relevant version).
  Of course, a proof-relevant interpretation retains more information, including the algorithms implicitly defined by a constructive proof.
  But the natural proof-irrelevance of some models can be important, such as when defining the Dedekind real numbers in a topos.
  Similarly, some axioms cannot be stated consistently in the naturally proof-relevant logic of dependent type theory, such as Brouwer's continuity principle~\cite{escardo-xu:brouwer-ch}, or the combination of excluded middle and univalence~\cite{hottbook}.
  (A referee has pointed out that this doesn't necessarily mandate full proof-irrelevance either; one might be able to use an intermediate modality instead.)
  Fortunately, as we will see in this section and the next, the antithesis translation is insensitive to this question: it works just as well for categories as for posets.%
  \footnote{Univalent type theory~\cite{hottbook} combines proof-relevance and proof-irrelevance in one framework, relating them by propositional truncation; thus one can choose between the two case-by-case rather than globally.
  But it is not clear to me how to incorporate this flexibility into the antithesis model.}
\end{rmk}

Returning to the G\"odel--Gentzen translation, any bicartesian closed category $\bH$ is distributive~\cite{clw:ext-dist}, so its initial object is strict, meaning any morphism with codomain $\zero$ is an isomorphism.
In particular, $\zero$ is subterminal, and thus so is the Heyting negation $\neg P \defeq (P\to \zero)$ of any object.
Thus the full subcategory
\[ \bH_{\neg\neg} \defeq \setof{ P \in \bH | P \cong \neg\neg P } \]
consists of subterminal objects, and hence is a preorder.
It is closed in $\bH$ under $\land$ and $\to$, and contains $\one$ and $\zero$; it is not closed under $\lor$ but it does have binary joins defined in \bH by $\neg\neg (P\lor Q)$.
With these operations it (or more precisely its skeleton) is a Boolean algebra.
Moreover, this construction defines a left adjoint, in a suitable sense, to the forgetful functor from Boolean algebras to bicartesian closed categories.
This is the semantic side of the (propositional) G\"odel--Gentzen translation: it makes a \emph{model} of \emph{intuitionistic} logic into a \emph{model} of \emph{classical} logic.

The syntactic side goes in reverse, translating any \emph{formula} in \emph{classical} logic into a \emph{formula} in \emph{intuitionistic} logic.
It can be obtained automatically from the semantic side, by considering the free bicartesian closed category $\dH[\Sigma]$ generated by some signature $\Sigma$, whose objects and morphisms are formulas and proofs in intuitionistic logic, and its resulting Boolean algebra $\dH[\Sigma]_{\neg\neg}$.
Then if $\dB[\Sigma]$ is the free Boolean algebra generated by the same signature, whose elements and inequalities are formulas and entailments in classical logic, its universality means there is a unique Boolean algebra homomorphisms $(-)\gg : \dB[\Sigma] \to \dH[\Sigma]_{\neg\neg}$.

This is the syntactic side of the G\"odel--Gentzen translation, which maps formulas and entailments in classical logic into formulas and proofs in intuitionistic logic.
Its usual explicit definition can be read off from the Boolean algebra structure of $\dH[\Sigma]_{\neg\neg}$ above, e.g.\ $(P\land Q)\gg = P\gg \land Q\gg$ and $(P\lor Q)\gg = \neg\neg(P\gg \lor Q\gg)$.
In particular, deriving these formulas semantically in this way means that the translation is automatically sound, i.e.\ maps proofs to proofs.

The situation with the Girard translation is similar.
We recall the relevant category-theoretic models of (propositional) linear logic.

\begin{defn}
  A \textbf{$\ast$-autonomous category}~\cite{barr:staraut} is a closed symmetric monoidal category $(\bL,\mand,\imp)$ equipped with an object $\bot$ such that for any $P$ the double-dualization map from $P$ to $(P\imp \bot)\imp \bot$ is an isomorphism.
  If \bL has finite products, a \textbf{Seely comonad}~\cite{seely:ll-staraut,mellies:catsem-ll} on it is a comonad $\oc$ such that the Kleisli category $\bL_\oc$ has finite products and the forgetful functor $\bL_\oc \to \bL$ is strong symmetric monoidal (where $\bL_\oc$ is regarded as cartesian monoidal).
\end{defn}

In a $\ast$-autonomous category we write $\nt{P} = (P \imp \bot)$ and $P \mor Q = \ntp{\nt{P} \mand \nt {Q}}$, and if it has a Seely comonad we write $\wn P = \ntp{\oc\ntp{P}}$.
Since $\ntp{-}$ is a self-duality, if a $\ast$-autonomous category has products then it also has coproducts; as in \cref{sec:meaning} we write its products as $\aand$ and its coproducts as $\aor$.

If we identify the objects of $\bL_\oc$ with objects of \bL as usual, with $\bL_\oc(P,Q) = \bL(\oc P,Q)$ and the forgetful functor $\bL_\oc \to \bL$ acting on objects by $\oc$, then the cartesian product in $\bL_\oc$ must take objects $P$ and $Q$ to $P\aand Q$, hence we have $\oc(P\aand Q) \cong \oc P \mand \oc Q$ in \bL.\footnote{The definition of Seely comonad is usually expanded out more explicity in terms of coherent isomorphisms such as these, but we will not need that.
  This is also apparently the origin of the sobriquet ``exponential'' for the modalities such as $\oc$, since ``exponentials turn additives into multiplicatives'' is akin to $\exp(a+b) = \exp(a)\cdot \exp(b)$.}
It follows that $\bL_\oc$ is a cartesian closed category, with exponential $(P\to Q) = (\oc P \imp Q)$.

This is the semantic side of the Girard translation.
Its syntactic side can be deduced as before, as a map $(-)\gr : \dH'[\Sigma] \to \dL[\Sigma]_\oc$ from the free cartesian closed category $\dH'[\Sigma]$ on some signature to the cartesian closed category $\dL[\Sigma]_\oc$ underlying the free $\ast$-autonomous category with finite products and Seely comonad on the same signature.
This yields a map from formulas and proofs in intuitionistic logic\footnote{We do not actually get a translation of all of intuitionistic logic, because $\dL[\Sigma]_\oc$ may not have coproducts.
  Thus, there is no obvious way to interpret $(P\lor Q)$; as noted already by Seely~\cite{seely:ll-staraut} it is hard to semantically justify Girard's formula $(P\lor Q)\gr = \oc P\gr \aor \oc Q\gr$.
  However, as we will see, in our case of interest these coproducts exist automatically.} to those in linear logic, whose syntactic rules can be read off of the cartesian closed structure of $\dL[\Sigma]_\oc$, e.g.\ $(P\land Q)\gr = P\gr \aand Q\gr$ and $(P\to Q)\gr = (\oc P\gr \imp Q\gr)$.

The semantic side of the antithesis translation, therefore, will be a construction of a model of affine logic from a model of intuitionistic logic.
By a model of affine logic we mean a $\ast$-autonomous category with finite products and Seely comonad that is \textbf{semicartesian}, meaning that its monoidal unit is also its terminal object (and hence the dualizing object $\bot$ is also the initial object).

The semantic antithesis translation is actually an instance of both the \emph{Chu construction}~\cite{chu:construction,chu:constr-app} and the \emph{Dialectica construction} in the form of~\cite{paiva:dialectica-chu}.
We will not describe these constructions in general, but only the specific case of interest to us, in which they coincide.
On the side of {subsets} rather than {predicates}, a similar notion was already introduced by~\cite[Chapter 3, \S2]{bb:constr-analysis} under the name \emph{complemented subset}; see \cref{thm:rel}.

\begin{defn}
  For a bicartesian closed category $\bH$, let $\chuh$ be the full subcategory of $\bH\times \bH\op$ determined by the pairs $P=(\pf P, \rf P)$ such that $\pf P \land \rf P$ is initial (equivalently, such that there is a morphism $\pf P \land \rf P\to \zero$; such a morphism is unique when it exists because $\zero$ is subterminal).
\end{defn}

Thus, a morphism $f : P\to Q$ in $\chuh$ consists of maps $\pf f : \pf P \to \pf Q$ and $\rf f : \rf Q \to \rf P$ in $\bH$.
In general, Chu and Dialectica constructions impose additional constraints on such morphisms --- the difference between the two being in the constraints --- but since $\zero$ is subterminal, those constraints are vacuous for us.

\begin{lem}\label{thm:bicart}
  $\chuh$ has finite products and coproducts.
\end{lem}
\begin{proof}
  They are inherited from $\bH\times \bH\op$, with the following definitions:
\begin{alignat*}{2}
  \top &= (\one,\zero) & \qquad
  P \aand Q &= (\pf P \land \pf Q, \rf P \lor \rf Q) 
  \\
  \bot &= (\zero,\one) &
  P \aor Q &= (\pf P \lor \pf Q, \rf P \land \rf Q).
\end{alignat*}
Note that we use \cref{rmk:lines} for these.
\end{proof}

\begin{lem}\label{thm:staraut}
  $\chuh$ is a semicartesian $\ast$-autonomous category.
\end{lem}
\begin{proof}
  The monoidal structure is defined by
  \[ P \mand Q \defeq (\pf P \land \pf Q, (\pf P \to \rf Q) \land (\pf Q \to \rf P)) \]
  A general Chu construction requires a pullback rather than a product in the refutations of $P\mand Q$, but subterminality of $\zero$ makes that unnecessary here.
  We leave associativity and symmetry to the reader (or standard references on the Chu and Dialectica constructions).
  For the unit, since $\top = (\one,\zero)$ we have
  \begin{align*}
    P \mand \top &= (\pf P \land \one, (\pf P \to \zero) \land (\one \to \rf P))\\
    &\cong (\pf P, (\pf P \to \zero) \land \rf P)\\
    &\cong (\pf P, \rf P).
  \end{align*}
  Here the final isomorphism is because $\pf P \to \zero$ is subterminal and there is a morphism $\rf P \to (\pf P \to \zero)$.
  The closed structure is defined by
  \[   P \imp Q \defeq ((\pf P \to \pf Q) \land (\rf Q \to \rf P), \pf P \land \rf Q). \]
  We leave the verification of adjointness to the reader.
  Now since $\bot = (\zero,\one)$, we have
  \begin{align*}
    P \imp \bot &\cong
    ((\pf P \to \zero) \land (\one \to \rf P), \pf P \land \one) \\
    &\cong
    (\rf P, \pf P)
  \end{align*}
  from which it follows immediately that $P \cong (P \imp \bot) \imp \bot$.
\end{proof}



\begin{lem}\label{thm:lnl}
  $\chuh$ has a Seely comonad.
  Moreover:
  \begin{itemize}
  \item The Seely comonad $\oc$ is idempotent.
  \item Its Kleisli category $(\chuh)_\oc$ is equivalent to \bH, and in particular has coproducts.
  \item The right adjoint $\chuh \to (\chuh)_\oc$ is strong monoidal; thus in addition to the Seely conditions we have $\oc(P\mand Q) \cong \oc P \mand \oc Q$ and $\wn (P \mor Q) \cong \wn P \mor \wn Q$.
  \end{itemize}
\end{lem}
\begin{proof}
  The forgetful map $\pfp{-} : \chuh \to \bH$ has a fully faithful left adjoint sending $P$ to $(P,\neg P)$, where $\neg P \defeq (P\to\zero)$ is the Heyting negation.
  The induced comonad and monad are
  \begin{equation*}
    \oc P \defeq (\pf P, \neg \pf P) \hspace{2cm}
    \wn P \defeq (\neg \rf P, \rf P).
  \end{equation*}
  Since the left adjoint is fully faithful, this comonad and monad are idempotent, and their Kleisli (and also Eilenberg--Moore) adjunctions coincide with the adjunction we started with.
  The definition of $\mand$ makes it clear that the right adjoint $\pfp{-}$ is strong symmetric monoidal, while for the left adjoint $F$ we have
  \begin{align*}
    F P \mand F Q &= ( P \land  Q, ( P \to \neg  Q) \land ( Q \to \neg P))\\
    &\cong ( P \land  Q, ( P \to Q \to \zero) \land ( Q \to P \to \zero))\\
    &\cong ( P \land  Q, (P \land Q \to \zero) \land (P \land Q \to \zero))\\
    &\cong ( P \land  Q, \neg (P \land Q))\\
    &= F(P\land Q)
  \end{align*}
  using that since $P \land Q \to \zero$ is subterminal, it is its own cartesian square.
\end{proof}


\cref{thm:bicart,thm:staraut,thm:lnl} define the \textbf{semantic antithesis translation}, which in fact is a right adjoint to a suitable forgetful functor.
As before, we obtain the \textbf{syntactic antithesis translation} as the unique morphism $(-)^\pm : \dA[\Sigma] \to \dH[\Sigma]_{\pm}$, where $\dA[\Sigma]$ is the free semicartesian $\ast$-autonomous category with products and a Seely comonad on the signature $\Sigma$.
This is a translation that maps formulas and proofs in affine logic to \emph{pairs} of formulas and proofs in intuitionistic logic.
It is given by explicit formulas that can be read off of the structure of $\chuh$, and which are shown in \cref{fig:antithesis}.
As with the G\"odel--Gentzen and Girard translations, the semantic derivation of these formulas automatically ensures soundness: any proof in affine logic automatically translates to a pair of proofs in intuitionistic logic.

\begin{figure}
  \centering
  \begin{alignat*}{2}
  \pfp{P\aand Q} &= \pf{P} \land \pf{Q} &\hspace{1cm} \rfp{P\aand Q} &= \rf{P} \lor \rf{Q}\\
  \pfp{P\aor Q} &= \pf{P} \lor \pf{Q} & \rfp{P\aor Q} &= \rf{P} \land \rf{Q}\\
  \pfp{\nt P} &= \rf P & \rfp{\nt P} &= \pf P\\
  \pfp{P\imp Q} &= (\pf P \to \pf Q) \land (\rf Q \to \rf P) & \rfp{P\imp Q} &= \pf P \land \rf Q\\
  \pfp{P\mand Q} &= \pf{P} \land \pf{Q} & \rfp{P\mand Q} &= (\pf P \to \rf Q) \land (\pf Q \to \rf P)\\
  \pfp{P\mor Q} &= (\rf P \to \pf Q) \land (\rf Q \to \pf P) & \rfp{P\mor Q} &= \rf P \land \rf Q\\
  \pf{\top} &= \one & \rf{\top} &= \zero \\
  \pf{\bot} &= \zero & \rf{\bot} &= \one \\
  \pfp{\oc P} &= \pf P & \rfp{\oc P} &= \neg(\pf P)\\
  \pfp{\wn P} &= \neg(\rf P) & \rfp{\wn P} &= \rf P.
\end{alignat*}
\caption{The syntactic antithesis translation for propositional logic}
\label{fig:antithesis}
\end{figure}

Of course, all the definitions in \cref{fig:antithesis} match our informal explanations of the connectives in \cref{sec:meaning}.
Thus any rigorous version of the BHK interpretation, yielding a bicartesian closed category that models propositional intuitionistic logic, can be enhanced to a model of propositional affine logic matching our meaning explanation.

Moreover, since the Kleisli category of $\chuh$ recovers $\bH$ again, the Girard translation undoes the antithesis translation:
\begin{center}
  \small
  \begin{tikzpicture}
  \node (a) at (0,0) {\begin{tabular}{c}
    intuitionistic logic\\
    (bicartesian\\ closed cat.)\\
    $\bH$ 
  \end{tabular}};
  \node (b) at (5,0) {\begin{tabular}{c}
    linear/affine logic\\
    ($\ast$-autonomous cat.\\ w/ Seely comonad)\\
    $\chuh$ 
  \end{tabular}};
  \node (c) at (9.5,0) {\begin{tabular}{c}
    intuitionistic logic\\
    (bicartesian\\ closed cat.)\\
    $(\chuh)_\oc \simeq \bH$
  \end{tabular}};
\draw[->] (a) -- node[auto] {antithesis} (b);
\draw[->] (b) -- node[auto] {Girard} (c);
\end{tikzpicture}
\end{center}
In other words, we can regard affine logic as an \emph{extension} of intuitionistic logic.

A very important point, however, is that unlike the G\"odel--Gentzen and Girard translation, the antithesis translation is not \emph{conservative} in the logical sense.
That is, there are statements in affine logic that always hold under the antithesis translation (i.e.\ in categories of the form $\chuh$), but are not provable in general affine logic.
Some such statements include:
\begin{mathpar}
  P \aand (Q\aor R) \logeq (P \aand Q) \aor (P \aand R)\and
  P\mand P\mand P \logeq P\mand P\and
  \oc\oc P \logeq \oc P \and
  \oc P \logeq P\mand P\and
  \wn (P\imp \oc P)\and
  \wn \oc P \types \oc \wn P\and
  \oc (P \aor Q) \logeq \oc P \aor \oc Q\and
  \oc (\ex x. P(x)) \logeq \ex x. \oc P(x).
\end{mathpar}
To a linear logician, this makes the logic look quite degenerate: much of the potential richness of the exponentials is invisible to the antithesis translation.
Therefore, unlike the G\"odel--Gentzen and Girard translations, we should not view the antithesis translation as a way to study affine logic by ``embedding'' it into intuitionistic logic. 
Instead, we view it as a way to \emph{use} affine logic as a tool for stating and proving definitions and theorems in intuitionistic logic.
(But we will return to this in \cref{sec:lcm}.)

\begin{rmk}\label{rmk:proofs}
  As noted above, the antithesis translation is sound for proofs.
  We will not make very extensive use of proofs in affine logic, but
  it is worth briefly summarizing the relevant rules (see e.g.~\cite{girard:ll} for more detail).

  Informally, affine logic looks like classical logic except that each hypothesis may only be used at most once (except for those with a $\oc$ on them).
  Put differently, the hypotheses of a theorem are implicitly combined with $\mand$, and since $P \nlogeq P\mand P$ they cannot be ``duplicated''.
  If we have $P\mand Q$ we can use both $P$ and $Q$ (at most once each), whereas if we have $P\aand Q$ we can choose to use $P$ or to use $Q$, but not both.
  Similarly, $\fa x.P(x)$ can only be instantiated at \emph{one} value of $x$.
  And dually, to prove $P\mand Q$ we prove $P$ and $Q$ with each hypothesis used only in one sub-proof, while to prove $P\aand Q$ we can use each hypothesis in both sub-proofs (once in each).

  A hypothesis of $P\aor Q$ can be case-split, while a hypothesis of $P\mor Q$ is used by disjunctive syllogism (e.g.\ proving $\nt P$ to conclude $Q$).
  To prove $P\aor Q$ we prove $P$ or prove $Q$, while to prove $P\mor Q$ we can assume $\nt P$ to prove $Q$ or vice versa.
  Implication $P\imp Q$ behaves as classically, including contraposition; proof by contradiction is universally valid.
  (Intuitionistically, proof by contradiction implies excluded middle $P\lor \neg P$ since $\neg(P\lor \neg P) \logeq (\neg P\land P)$ is a contradiction; but affinely $\ntp{P\aor \nt P} \logeq (\nt P \aand P)$ is no contradiction since we can't use \emph{both} $\nt P$ and $P$.)
\end{rmk}

\begin{rmk}\label{rmk:words}
  Notations such as $\aand/\aor$ and $\mand/\mor$ are fine for writing logical formulas explicitly, but for talking about mathematics it is useful to also represent each connective by an English word.
  Girard suggested to pronounce $\aand$ as ``with'', $\aor$ as ``plus'', $\mand$ as ``tensor'', and $\mor$ as ``par''; but most of these words have other meanings in mathematics and everyday English, leading to potential confusion.

  When it is understood that the ambient logic is linear or affine (so that there is no danger of confusion with $\land$ and $\lor$), I prefer to prounonce $\mand$ as simply ``and'', since this conjunction implicitly combines multiple hypotheses, is left adjoint to implication, and is very often used where both intuitionistic and classical mathematics use $\land$.
  Similarly, I prefer to pronounce $\aor$ as simply ``or'', since this is the disjunction that supports proof by cases, is almost always\footnote{We will see in the rest of the paper that that $\aand$ and $\mor$ do often appear in affine representations of \emph{concepts} from intuitionistic constructive mathematics.
    But the mathematician using intutionistic logic has to write out the corresponding more complicated statement using $\land$, $\lor$, and $\to$, hence is not used to using the words ``and'' and ``or'' for $\aand$ and $\mor$ respectively.} what an intuitionistic constructive mathematician means by ``or'', and about half the time is what a classical mathematician means by ``or'' as well.
  The other half of the time the classical mathematician means $\mor$, for which Girard's word ``par'' is at least unlikely to lead to confusion; but two less awkward-sounding possibilities are ``unless'' and ``or else'', since $P\mor Q$ is equivalent to both $\nt{Q}\imp P$ and $\nt{P} \imp Q$.
  Pronouncing $\aand$ is trickier, but Noah Snyder has suggested ``exclusive and'' (``xand'' for short) --- there is no formal relationship to the ``exclusive or'', but the word ``exclusive'' conveys the intuition of ``exactly one of the two'', which is how a hypothesis of $P\aand Q$ can be used in a linear proof: as $P$ or as $Q$, but not both.
\end{rmk}

\begin{rmk}\label{rmk:00}
  One might argue that $\chuh$ is too large, as it contains propositions like $(\zero,\zero)$ which are very far from being either provable \emph{or} refutable.
  We cannot constructively expect every proposition to be either provable or refutable,
  but we might try some weaker restriction like $\neg (\neg \pf P \land \neg \rf P)$.
  However, while propositions satisfying $\neg (\neg \pf P \land \neg \rf P)$ are closed under finitary \emph{connectives},
  their closure under \emph{quantifiers} is equivalent to the non-constructive law of ``double-negation shift'' $(\forall x. \neg\neg P(x)) \to (\neg\neg\forall x. P(x))$. 
  For a dramatic counterexample, let $\bH = \cO(\dR)$ be the open-set lattice of the real numbers, with $x:\dR$ and $\pf {P(x)} \defeq \dR \setminus \{x\}$ and $\rf {P(x)} \defeq \zero$; then $\neg (\neg \pf {P(x)} \land \neg \rf {P(x)})$ for all $x$, but $\fa x. (\pf {P(x)}, \rf {P(x)}) = (\zero,\zero)$.
\end{rmk}

\begin{rmk}\label{rmk:vickers}
  In fact, already Vickers~\cite{vickers:topology-via-logic} explicitly suggested considering separately for each proposition its \emph{affirmations} and \emph{refutations}:
  \begin{quote}\small
    Given an assertion, we can therefore ask --
    \begin{itemize}
    \item Under what circumstances could it be affirmed?
    \item Under what circumstances could it be refuted?~\cite[p6]{vickers:topology-via-logic}
    \end{itemize}
  \end{quote}
  It is thus natural to imagine propositions that can never be affirmed (i.e.\ proven) and also never refuted.
  The antithesis construction can thus be viewed as an intensional theory of the proofs and refutations of propositions without regard to ``truth''.
  Ignoring truth is constructively sensible since we can never directly observe it (we can only affirm or refute propositions), and intensionality is sensible since two propositions that happen to have the same extension (truth circumstances) might have different affirmations or refutations depending on how they are phrased.

  Vickers defines a proposition to be \emph{affirmative} if it is true exactly when it can be affirmed (i.e.\ proven), and \emph{refutative} if it is false exactly when it can be refuted.
  Our definition is a bit stronger, and more intensional: roughly speaking, we call a proposition affirmative if we \emph{know}, by virtue of its definition, that whenever it is true it can be affirmed, so that we can refute by showing that it cannot be affirmed.
  Similarly, we call a proposition refutative if we know that whenever it is false it can be refuted, so that we can affirm it by showing that it cannot be refuted.

  If intuitionistic logic is the logic of affirmative propositions, and co-intuitionistic logic~\cite{trafford:co-constructive,shramko:dil} is the logic of refutative propositions, then we can view affine logic as a logic of propositions that are subject to either affirmation or refutation.
\end{rmk}

\begin{rmk}
  Even if $\bH$ is a Boolean algebra, $\chuh$ is larger than $\bH$.
  For instance, $\chuzeroone = \{ (\zero,\one) \le (\zero,\zero) \le (\one,\zero) \}$ coincides with
  three-valued \luk{}ukasiewicz logic, where $(\zero,\zero)$ is called ``unknown'' or ``undefined''.
\end{rmk}

\begin{rmk}
  Dan Licata has pointed out that the antithesis translation has certain parallels with the natural deduction of~\cite{lc:clnatded} for classical (linear) logic that uses two judgments $P \;\mathbf{true}$ and $P\;\mathbf{false}$.
\end{rmk}

\begin{rmk}\label{rmk:constr-duality}
  There are other ways to add ``constructive negation'' to intuitionistic logic.
  We have already noted that the antithesis construction is both a Chu construction and a Dialectica construction, and both of these constructions have more general versions that also model linear logic.
  For instance, it is shown in~\cite{patterson:thesis} that the ``constructible falsity'' logic of~\cite{nelson:constr-falsity} is modeled by the Chu construction $\chu(\bH,\one)$.
  Compared to $\chuh$ (which is $\chu(\bH,\zero)$), this drops even the requirement $\neg (\pf P \land \rf P)$, allowing propositions like $(\one,\one)$ that are both provable and refutable.
  The lattice $\chu(\bH,\one)$ is $\ast$-autonomous but not semicartesian, so the units of $\mand$ and $\mor$ no longer coincide with those of $\aand$ and $\aor$.
  Instead we have the MIX rule~\cite{cs:pfth-bill}, i.e.\ the units of $\mand$ and $\mor$ coincide with each other; indeed they are precisely $(\one,\one)$.
\end{rmk}

\section{The antithesis translation for predicate logic}
\label{sec:types}

\renumberthms{subsection}

To do substantial mathematics we require not just propositional logic, but at least first-order logic, and often higher-order logic or even dependent types.
While attempting not to get bogged down by detail, in this section we describe antithesis translations for these richer theories.
I encourage a reader who is not already an afficionado of categorical semantics to skim this section on a first reading.

\subsection{First-order logic}
\label{sec:fol}

This corresponds semantically to the following notion, due essentially to Lawvere~\cite{lawvere:adjointness}.

\begin{defn}
  Let \cK be a 2-category with a forgetful functor $U:\cK \to \cCat$.
  A \textbf{\cK-valued hyperdoctrine} consists of:
  \begin{itemize}
  \item A category \bT with finite products.
  \item A pseudofunctor $\cP : \bT\op \to \cK$.
  \item For any product projection $\pi:A\times B \to A$ in \bT, the functor
    \[ U\pi^* : U\cP(B) \to U\cP(A\times B) \]
    has both a left adjoint $\Sigma_B$ and a right adjoint $\Pi_B$.
  \item The Beck-Chevalley condition holds, meaning that for any $f:A'\to A$ in \bT the induced maps are isomorphisms:
    \[ \Sigma_B \circ (f\times B)^* \toiso f^* \circ \Sigma_B \qquad
      f^* \circ \Pi_B \toiso \Pi_B \circ (f\times B)^*.
    \]
  \end{itemize}
\end{defn}

We think of the objects of \bT as representing \emph{types} and its morphisms as \emph{terms}, with the objects of $\cP(A)$ being \emph{predicates} on $A$.
The morphism $f^* : \cP(A) \to \cP(A')$ of \cK induced by $f:A'\to A$ represents \emph{substitution} into a predicate, while the adjoints $\Sigma_B$ and $\Pi_B$ act like existential and universal \emph{quantification}.
Note that $\Sigma_B$ and $\Pi_B$ are not in general morphisms of \cK.

If $\cK=\cInt$ is the 2-category of bicartesian closed categories, with functors preserving finite products, coproducts, and exponentials, and natural isomorphisms between them, we speak of an \textbf{intuitionistic hyperdoctrine}, and write $\Sigma_B = \exists_B$ and $\Pi_B = \forall_B$.
Similarly, if $\cK=\cAff$ is the 2-category of semicartesian $\ast$-autonomous categories with finite products and a Seely comonad, with functors that preserve all this structure up to isomorphism, and natural isomorphisms between them, we speak of an \textbf{affine hyperdoctrine}, and write $\Sigma_B = \ex_B$ and $\Pi_B = \fa_B$.

\begin{eg}\label{eg:cplt-hyd}
  If \bH is a complete and cocomplete cartesian closed category, then there is an intuitionistic hyperdoctrine with $\bT = \bSet$ and $\cP(A) = \bH^A$.
  The adjoints $\Pi_B$ and $\Sigma_B$ are given by products and coproducts.
  Similarly, if \bL is a complete and cocomplete semicartesian $\ast$-autonomous category with a Seely comonad, then $\bT = \bSet$ and $\cP(A) = \bL^A$ defines an affine hyperdoctrine.
\end{eg}

\begin{egs}\label{eg:flim-hd}
  Suppose \bT is a category with finite limits.
  Then there is a pseudo\-functor $\cP : \bT\op \to \cCat$ sending $A$ to the poset of subobjects of $A$, which is an intuitionistic hyperdoctrine if and only if \bT is a Heyting category.
  There is also such a pseudofunctor sending $A$ to the slice category $\bT/A$, which is an intuitionistic hyperdoctrine if and only if \bT is locally cartesian closed with finite coproducts.

  More generally, any full comprehension category having $\Sigma$-types, $\Pi$-types (with function extensionality), and finite sum types is an intuitionistic hyperdoctrine.
  If it also has propositional truncations, in the sense of~\cite{hottbook}, then its ``h-propositions'' (types with at most one element) also form an intuitionistic hyperdoctrine.
  If it has universe objects closed under the relevant type formers, the elements of any particular universe also form an intuitionistic hyperdoctrine.
\end{egs}

\begin{rmk}\label{rmk:aff-quant}
  Even in an affine hyperdoctrine, the base category \bT is still \emph{cartesian} monoidal.
  We could allow \bT to be \emph{semicartesian} monoidal, as then it would still have ``projections'' whose adjoints would supply quantifiers; something similar appears in first-order Bunched Implication~\cite{op:bi}.
  But since the antithesis translation leaves the base category \bT unchanged, we have no need for this generality, although we will mention it again in \cref{rmk:set-hyd}.
\end{rmk}

We now extend the antithesis translation to hyperdoctrines.

\begin{lem}
  The semantic antithesis translation from \cref{sec:chu} defines a 2-functor
  \[ (-)_\pm : \cInt \to \cAff. \]
\end{lem}
\begin{proof}
  Immediate.
  Note that since $\chuh$, like its substrate $\bH\times \bH\op$, is partly covariant and partly contravariant, it can only be 2-functorial on natural \emph{isomorphisms}; this is why we defined $\cInt$ and $\cAff$ to contain only these.
\end{proof}

\begin{thm}
  If $\cP : \bT\op \to \cInt$ is an intuitionistic hyperdoctrine, the composite
  \[ \bT\op \xto{\cP}\cInt \xto{(-)_{\pm}} \cAff \]
  is an affine hyperdoctrine $\cP_{\pm}$.
\end{thm}
\begin{proof}
  It remains to show that if $\pi^* : \cP(A) \to \cP(A\times B)$ is a bicartesian closed functor with left and right adjoints $\exists_B$ and $\forall_B$, then $(\pi^*)_{\pm} : \cP(A)_{\pm} \to \cP(A\times B)_\pm$ also has left and right adjoints satisfying the Beck-Chevalley condition.
  We define these by the expected formulas:
  \begin{mathpar}
    \ex_B (\pf P,\rf P) \defeq (\exists_B \pf P, \forall_B \rf P)\and
    \fa_B (\pf P,\rf P) \defeq (\forall_B \pf P, \exists_B \rf P).
  \end{mathpar}
  We leave it to the reader to verify that this works.
\end{proof}

This defines the \textbf{semantic antithesis interpretation} for first-order logic.

Moving now to syntax, we consider a formal system of first-order logic with \emph{types}, each containing \emph{terms} or \emph{elements} $t:A$ that may involve variables belonging to other types, and a class of \emph{propositions} that may also involve variables belonging to types.
We assume finite \emph{product types} $A\times B$, whose elements are ordered pairs, and a \emph{unit type} $1$ that has one element.
Propositions are related by \emph{entailments}
\[ P, Q \types_{x:A, y:B} R \]
where $P,Q,R$ are propositions involving only the variables $x$ (of type $A$) and $y$ (of type $B$).
In \emph{intuitionistic} first-order logic, we equip the propositions with the usual intuitionistic logical operations:
\[ \land, \lor, \one, \zero, \to, \neg, \forall, \exists \]
and the usual intuitionistic rules of deduction.
Similarly, in \emph{affine} first-order logic, we equip the propositions with the affine logical operations:
\[ \mand, \mor, \aand, \aor, \top, \bot, \imp, \ntp{-}, \fa, \ex, \oc, \wn \]
together with the usual affine rules of deduction (that is, the rules of~\cite{girard:ll} for linear logic, plus weakening).

\begin{notn}
  For clarity, we sometimes annotate a quantified variable by the type to which it belongs, e.g.\ $\exists x^A. P(x)$ if $x:A$.
\end{notn}

By standard arguments (see e.g.~\cite{jacobs:cltt}), the syntax of either kind of first-order logic, starting from some signature of base types, terms, and propositions, presents a free hyperdoctrine of the appropriate sort.
(We gloss over coherence questions here, which can be resolved as in~\cite{hofmann:ttinlccc,lw:localuniv}, and are automatic in the proof-irrelevant case when the categories $\cP(A)$ are posets.)
Thus, as in the propositional case, by applying the semantic antithesis translation to the syntactic intuitionistic hyperdoctrine, we obtain a \textbf{syntactic antithesis translation} of affine first-order logic into intuitionistic first-order logic.
This translation leaves the types unchanged, acts on the propositional connectives as in \cref{fig:antithesis}, and acts on the quantifiers as shown in \cref{fig:antithesis-fo}.
As before, the derivation of this translation from the semantic version means that it is automatically sound for proofs (though not complete).

\begin{figure}
  \centering
  \begin{alignat*}{2}
    \pfp{\ex x. P(x)} &= \exists x. \pf P(x) &\qquad
    \rfp{\ex x. P(x)} &= \forall x. \rf P(x)\\
    \pfp{\fa x. P(x)} &= \forall x. \pf P(x) &\qquad
    \rfp{\fa x. P(x)} &= \exists x. \rf P(x)
  \end{alignat*}
  \caption{The syntactic antithesis translation for first-order logic}
  \label{fig:antithesis-fo}
\end{figure}

We now consider various additional structure that can be added to a hyperdoctrine, and their corresponding operations in syntax.

\subsection{Comprehension}
\label{sec:comprehension}

Let $\cCat_t$ be the 2-category of categories with a terminal object.\footnote{If we wanted to define comprehension for \emph{linear} hyperdoctrines in addition to affine ones, we would need to replace the terminal object $1$ by the monoidal unit.}
We denote such terminal objects generically by $1$.

\begin{defn}[{\cite{lawvere:comprehension}}]\label{defn:comp}
  Suppose $U:\cK \to \cCat$ factors through $\cCat_t$.
  A \cK-valued hyperdoctrine $\cP : \bT\op\to\cK$ has \textbf{comprehension} if for all $A\in \bT$ and $P\in \cP(A)$, the following functor is representable:
  \begin{align*}
    (\bT/A)\op &\to \bSet\\
    (f:B\to A) &\mapsto \cP(B)(1,f^*(P))
  \end{align*}
  We denote a representing object by $i_P : \{P\} \to P$: it can be thought of as the subtype of $A$ consisting of those elements that satisfy $P$.
\end{defn}

\begin{eg}
  The intuitionistic hyperdoctrine $\cP(A) = \bH^A$ has comprehension, with $\{P\}$ the set of pairs $(a,p)$ where $a\in A$ and $p\in \bH(\one,P_a)$.
  The same is true for the affine hyperdoctrine $\cP(A) = \bL^A$.
\end{eg}

\begin{eg}
  \cref{eg:flim-hd} all have comprehension.
  The comprehension of a subobject or object of a slice category is itself, while a comprehension category includes as data a comprehension operation.
\end{eg}

\begin{prop}\label{thm:comp}
  If \cP is an intuitionistic hyperdoctrine with comprehension, then $\cP_{\pm}$ is an affine hyperdoctrine with comprehension.
\end{prop}
\begin{proof}
  Since $\top = (\one,\zero)$ in $\cP(A)$, a morphism $\top \to P$ in $\cP(A)$ consists of morphisms $\one \to \pf{P}$ and $\rf{P}\to \zero$.
  But the latter is unique it if exists, which it does if there is a morphism $\one \to \pf{P}$ since $\pf{P}\land \rf{P} \to \zero$.
  Thus, we can define $\{P\} \defeq \{\pf{P}\}$.
\end{proof}

Note that comprehension in the antithesis model discards all information about refutations; hence in particular $\{P\} = \{\oc P\}$.
More generally, we have:

\begin{prop}\label{thm:comp-aff}
  For $P\in\cP(A)$ in any affine hyperdoctrine with comprehension, there is a morphism from $\top$ to $i_P^*( \oc P)$ in $\cP(\{P\})$.
\end{prop}
\begin{proof}
  By definition, there is a morphism from $\top$ to $i_P^*(P)$ in $\cP(\{P\})$.
  Now we apply the functor $\oc$ and use the fact that $\oc \top\cong \top$.
\end{proof}

Since $\oc$ is not in general idempotent (though it is in the antithesis model), this does not imply $\{P\} \cong \{\oc P\}$.
But it does make $P$ arbitrarily duplicable over $\{P\}$, i.e.\ we have $P \imp P\mand P$ over $\{P\}$.
Thus, we have to be careful to avoid comprehension whenever we want to retain ``refutational'' information.
This leads in particular to a wider gap betwen ``subsets of $A$'' and ``sets that inject into $A$''.
We will return to this point in \cref{rmk:affirm-axioms} and \cref{sec:sets}.

Nevertheless, we cannot really do mathematics without comprehension.
Fortunately, it is a fairly harmless assumption: even if we start from a hyperdoctrine without comprehension, we can add comprehensions ``freely'', replacing the types by ``formal comprehensions'' or ``pre-sets'' (types with an ``existence predicate'').

\begin{prop}\label{thm:int-comp}
  For any intuitionistic hyperdoctrine $\cP : \bT\op \to \cInt$, there is an intuitionistic hyperdoctrine with comprehension $\cP\ch : (\bT\ch)\op \to \cInt$ in which:
  \begin{itemize}
  \item The objects of $\bT\ch$ are pairs $(A,P)$ with $A\in\bT$ and $P \in \cP(A)$.
  \item The morphisms $(A,P) \to (B,Q)$ are pairs of $f:A\to B$ and $g:P \to f^*Q$.
  \item The objects of $\cP\ch(A,P)$ are those of $\cP(A)$.
  \item The morphisms $Q\to R$ in $\cP\ch(A,P)$ are morphisms $P\land Q \to R$ in $\cP(A)$.
  \end{itemize}
\end{prop}
\begin{proof}
  The product in $\bT\ch$ is $(A,P) \times (B,Q) \defeq (A\times B, \pi_1^*P \land \pi_2^* Q)$, and the terminal object is $(1,\one)$.
  It is straightforward to show that $\cP\ch(A,P)$ is bicartesian closed and that $\cP\ch$ is a functor.
  The quantifiers are $\exists_{(B,Q)}(R) \defeq \exists_B(Q\land R)$ and $\forall_{(B,Q)}(R) \defeq \forall_B(Q\to R)$.
  The comprehension of $Q \in \cP\ch(A,P) = \cP(A)$ is $(A, P\land Q)$.
\end{proof}

\begin{rmk}
  The construction of \cref{thm:int-comp} appears in many places with many names.
  Categorically, $\bT\ch$ is the ``Grothendieck construction'' of the composite $\bT\op\xto{\cP}\cInt\xto{U}\cCat$; for the Calculus of Constructions it is the ``first-order deliverables'' of~\cite{mckinna:thesis}.
  The general construction has a universal property of adding comprehensions ``freely''; see e.g.~\cite{trotta:cplt-elem-doc} for the posetal version.
\end{rmk}

\begin{prop}\label{thm:aff-comp}
  For any affine hyperdoctrine $\cP : \bT\op \to \cAff$, there is an affine hyperdoctrine with comprehension $\cP\ch : (\bT\ch)\op \to \cAff$ in which:
  \begin{itemize}
  \item The objects of $\bT\ch$ are pairs $(A,P)$ with $A\in\bT$ and $P \in \cP(A)$.
  \item The morphisms $(A,P) \to (B,Q)$ are pairs of $f:A\to B$ and $g:\oc P \to f^*Q$.
  \item The objects of $\cP\ch(A,P)$ are those of $\cP(A)$.
  \item The morphisms $Q\to R$ in $\cP\ch(A,P)$ are morphisms $\oc P\mand Q \to R$ in $\cP(A)$.
  \end{itemize}
\end{prop}
\begin{proof}
  The product in $\bT\ch$ is $(A,P) \times (B,Q) \defeq (A\times B, \pi_1^*P \aand \pi_2^* Q)$, and the terminal object is $(1,\top)$.
  Note that $\oc(\pi_1^*P \aand \pi_2^* Q) \cong \pi_1^* \oc P \mand \pi_2^* \oc Q$.
  The same operations as in $\cP(A)$ lift to make $\cP\ch(A,P)$ semicartesian $\ast$-autonomous with products and a Seely comonad.
  The quantifiers are $\ex_{(B,Q)}(R) \defeq \ex_B(\oc Q\mand R)$ and $\fa_{(B,Q)}(R) \defeq \fa_B(\oc Q\imp R)$.
\end{proof}

Syntactically, comprehension corresponds to an operation taking a proposition $P$ in the context of a variable $x:A$ to a type $\setof{x:A | P(x)}$.
In the intuitionistic case, rules for this operation can be found in~\cite[\S4.6]{jacobs:cltt};
note that $P$ cannot contain any variables other than $x$.
(It is possible to formulate a more general kind of comprehension without this restriction, at the expense of introducing dependent types.)
The affine case is essentially identical, but due to \cref{thm:comp-aff} we will emphasize the essentially affirmative nature of affine comprehension by writing it as
\[\ocsetof{x:A | P(x)}.\]

\subsection{Leibniz--Lawvere equality}
\label{sec:equality}

This operation will not be very useful for us, but we sketch it briefly to explain why.

\begin{defn}[{\cite{lawvere:comprehension}}, {\cite[\S3.4]{jacobs:cltt}}]\label{defn:hyd-eq}
  A \cK-valued hyperdoctrine $\cP : \bT\op\to\cK$ has \textbf{Leibniz--Lawvere equality} if for any diagonal $\triangle_A : A\to A\times A$ and object $B$, the functor $(1_B\times \triangle_A)^*$ has a partial left adjoint defined at the terminal object and satisfying the Beck-Chevalley condition.
\end{defn}

We denote the value of this left adjoint by $\sfeq_A \in \cP(B\times A\times A)$.


\begin{prop}\label{thm:ant-eq}
  If \cP is an intuitionistic hyperdoctrine with Leibniz--Lawvere equality, then $\cP_{\pm}$ also has Leibniz--Lawvere equality with $\sfeq_A^\pm \defeq (\sfeq_A, \neg \sfeq_A)$.\qed
\end{prop}

Note that $(\sfeq_A, \neg \sfeq_A)$ is always affirmative.
In fact, more generally we have:

\begin{prop}\label{thm:eq-aff}
  In any affine hyperdoctrine with Leibniz--Lawvere equality, the predicate $\sfeq_A$ is affirmative, i.e.\ we have a map $\sfeq_A \to \oc \sfeq_A$ in $\cP(B\times A\times A)$.
\end{prop}
\begin{proof}
  By the universal property of $\sfeq_A$, such a morphism is determined by a map $\top \to \triangle^* \oc \sfeq_A \cong \oc \triangle^* \sfeq_A$ in $\cP(B\times A)$.
  But $\oc\top \cong \top$, so it suffices to give a morphism $\top \to \triangle^* \sfeq_A$, and this is just the unit of the partial adjunction.
\end{proof}

See~\cite{grishin:linear} for a more syntactic argument.
Unlike the analogous \cref{thm:comp-aff} for comprehension, this result makes Leibniz--Lawvere equality unsuitable for us.
Indeed, equality was our primary example in \cref{sec:introduction} of a \emph{non-affirmative} proposition (with nontrivial refutations).
Thus, instead of using Leibniz--Lawvere equality, we will follow~\cite{bishop:fca,hjp:tripos} in equipping types with equality relations (see \cref{sec:sets,sec:linear-sets}).

\subsection{Higher-order structures}
\label{sec:hol}

Many higher-order structures are properties of the base category \bT that don't affect the hyperdoctrine over it; these are automatically preserved by the antithesis construction.
For instance, we can ask that \bT be cartesian closed; this corresponds syntactically to enhancing the base type theory of our first-order logic to a simply-typed $\lambda$-calculus, with \emph{operation types} $B^A$ whose canonical elements are abstractions $\lambda x.t$, satisfying $\beta$ and $\eta$ conversion rules.\footnote{These rules must hold up to a \emph{judgmental} equality of terms, the syntactic counterpart of equality of morphisms in \bT.
  This is distinct from the equality \emph{propositions} of \cref{sec:equality}.}
(These are usually called \emph{function types}, but for us ``functions'' will be defined to be operations that respect a given equality relation; see \cref{sec:sets,sec:linear-sets}.)
We observe:

\begin{prop}\label{thm:comp-cc}
  If \bT is cartesian closed, then so is the $\bT\ch$ defined in \cref{thm:int-comp,thm:aff-comp}.
\end{prop}
\begin{proof}
  The exponentials in the two cases are
  \begin{align*}
    (B,Q)^{(A,P)} &\defeq (B^A, \forall_A(\pi_2^*P \to \ev^*Q))\\
    (B,Q)^{(A,P)} &\defeq (B^A, \fa_A(\oc \pi_2^*P \imp \ev^*Q))
  \end{align*}
  where $\ev : B^A \times A \to B$ is the evaluation in \bT.
\end{proof}

Similarly, we can ask that \bT be equipped with a comprehension category or category with families, unrelatedly to the hyperdoctrine.
This corresponds syntactically (again, modulo coherence issues that can be addressed as in~\cite{hofmann:ttinlccc,lw:localuniv}) to enhancing the base type theory with dependent types, possibly with any desired type formers such as $\Sigma$-types, $\Pi$-types, identity types, etc.\ (which are, at least \emph{a priori}, unrelated to the hyperdoctrine and its quantifiers).
Put differently, this results in a \emph{logic-enriched dependent type theory} in the sense of~\cite{ag:colldtt} in which the logic is that of the hyperdoctrine.
Since this structure is likewise undisturbed by the antithesis construction on the hyperdoctrine, we have an antithesis translation from affine-logic-enriched type theory into intuitionistic-logic-enriched type theory.
We leave it to the reader to extend \cref{thm:comp-cc} to such cases.

\begin{eg}
In particular, as in \cref{eg:flim-hd}, we can regard a comprehension category with $\Sigma$- and $\Pi$-types as \emph{itself} an intuitionistic hyperdoctrine with comprehension.
That is, any type theory admits an intuitionistic-logic enrichment given by propositions-as-types.
Applying the antithesis construction, we obtain a translation from affine-logic-enriched dependent type theory into ordinary intuitionistic dependent type theory.
Similarly, we can apply the antithesis construction to the hyperdoctrine of h-propositions, or the elements of some fixed universe.
\end{eg}

\begin{rmk}\label{rmk:dtt}
This way of applying the antithesis construction to dependent type theory acts only on the ``top level'' of type dependency, and we will not attempt to extend it further in this paper (although see \cref{rmk:set-hyd}).
In particular, there are by now many different approaches to ``linear dependent type theory'', and it is unclear which, if any, of them would be appropriate for such an extended translation.
The lack of a definite answer to this question is one obstacle to a native ``affine constructive mathematics'', since dependent type theory has definite advantages over higher-order logic as a foundational system for all of mathematics.
But we can still use affine logic, by way of the antithesis translation, to say useful things about the top-level logic of intuitionistic dependent type theory.
\end{rmk}

\subsection{Generic predicates}
\label{sec:generic-predicates}

This is the primary higher-order structure that \emph{does} interact with a hyperdoctrine.

\begin{defn}
  A \textbf{generic predicate}\footnote{Called a ``weak generic object'' in~\cite[\S5.2]{jacobs:cltt}.} in a hyperdoctrine $\cP : \bT\op\to\cK$ is an object $\Om\in \bT$ with an element $\done \in \cP(\Om)$ such that for any $A\in \bT$ and $P\in \cP(A)$, there exists a (not necessarily unique) $f:A\to\Om$ and isomorphism $P \cong f^*(\done)$.
\end{defn}

\begin{eg}
  If \bH is a small complete Heyting algebra, the intuitionistic hyperdoctrine $\cP(A) = \bH^A$ has a generic predicate with $\Om$ the underlying set of \bH and $\done \in \bH^{\Omega}$ the identity function.
  A similar argument applies to the affine hyperdoctrine $\cP(A) = \bL^A$, if \dL is a small $\ast$-autonomous complete lattice with a Seely comonad.
\end{eg}

\begin{eg}
  The subobject classifier of an elementary topos is a generic predicate for the hyperdoctrine of subobjects.
  More generally, a preorder-valued intuitionistic hyperdoctrine with cartesian closed base and a generic predicate is a \emph{tripos}~\cite{hjp:tripos}.
\end{eg}

\begin{eg}
  A comprehension category, regarded as an intuitionistic hyperdoctrine, does not generally have a generic predicate.
  However, its restricted hyperdoctrine of elements of some universe does have one, namely the universe.
\end{eg}

\begin{prop}
  If $\Omega$ is a generic predicate for $\cP$, then $(\Omega,1)$ is a generic predicate for the $\cP\ch$ defined in \cref{thm:int-comp,thm:aff-comp}, where $1$ is the terminal object of $\cP(\Omega)$.\qed
\end{prop}

\begin{prop}\label{thm:comp-gp}
  If \cP is an intuitionistic hyperdoctrine with comprehension and a generic predicate, then $\cP_{\pm}$ also has a generic predicate.
\end{prop}
\begin{proof}
  Over $\Om\times \Om$ we have two canonical predicates $\pi_1^*\done$ and $\pi_2^*\done$.
  Let $\Om_{\pm}$ be the comprehension of $\neg (\pi_1^*\done \land \pi_2^*\done)$.
  Then to give a morphism $f:A \to \Om_{\pm}$ is the same as to give two morphisms $\pf f : A\to \Om$ and $\rf f : A \to \Om$, corresponding to predicates $(\pf f)^*(\done)$ and $(\rf f)^*(\done)$ over $A$, such that $(\pf f)^*(\done)\land (\rf f)^*(\done)$ is initial in $\cP(A)$.
  Thus, $\Om_{\pm}$ is a generic predicate for $\cP_{\pm}$.
\end{proof}

\begin{eg}
  In the hyperdoctrine of subobjects in a topos, $\Om_{\pm}$ is the subobject of $\Om\times \Om$ consisting internally of pairs of incompatible propositions.
\end{eg}

Syntactically, a generic predicate corresponds to having an (impredicative) type $\Omega$ of all propositions.
The usual way of presenting a higher-order type theory of this sort is to \emph{define} the propositions to be the terms of type $\Omega$, or at least bijective to them.
In a hyperdoctrine with a generic predicate, we have only an essentially surjective function $\bT(A,\Omega) \to \cP(A)$, but as in~\cite{hjp:tripos} we can use this to replace $\cP(A)$ by an equivalent category whose set of objects is precisely $\bT(A,\Omega)$.

Thus, with \cref{thm:comp-gp} we can translate affine higher-order logic into intuitionistic higher-order logic, where both have comprehensions and a type of propositions.
The syntactic expression of the type of propositions derived from \cref{thm:comp-gp} is
\[ \Om_{\pm} \defeq \setof{ (\pf p,\rf p) : \Om\times \Om | \neg (\pf p \land \rf p) }.
\]
Note that unlike the syntactic antithesis translations for propositional and first-order logic in \cref{fig:antithesis,fig:antithesis-fo}, this is a definition of a \emph{type}, not a proposition or predicate.
The antithesis translation does not modify the collection of types or most of the operations on them, but it does change the type of propositions.

\subsection{Infinity}
\label{sec:infinity}

Finally, as a starting point for concrete mathematics, we require at least a type of natural numbers that permits definitions by recursion and proofs by induction.
The intuitionistic version of this is straightforward.

\begin{defn}\label{thm:inno}
  In an intuitionistic hyperdoctrine, a \textbf{natural numbers type} is an object $N\in \bT$ together with morphisms $o:1\to N$ and $s:N\to N$ such that:
  \begin{enumerate}
  \item For any objects $A,B\in \bT$ with $f:A\to B$ and $g:A\times N\times B\to B$, there exists a morphism $h:A\times N\to B$ making the following diagrams commute:\label{item:nno-rec}
    \[
      \begin{tikzcd}
        A \ar[r,"{(1_A,o)}"] \ar[dr,"f"'] & A\times N  \ar[d,"h"]\\
        & B
      \end{tikzcd}
      \qquad
      \begin{tikzcd}
        A\times N \ar[d,"{(1_{A\times N},h)}"'] \ar[r,"{1_A \times s}"] & A\times N \ar[d,"h"]\\
        A\times N\times B \ar[r,"g"'] & B.
      \end{tikzcd}
      \]
  \item For any predicate $P\in \cP(A\times N)$, the following entailment holds:\label{item:nno-ind}
    \begin{equation}
      P_a(0)\,\land\, \forall k. (P_a(k) \to P_a(k+1)) \;\types_{a:A}\; \forall n. P_a(n).\label{eq:iind}
    \end{equation}
  \end{enumerate}
\end{defn}

Syntactically, the diagrams in~\ref{item:nno-rec} say that $h(a,0) = f(a)$ and $h(a,n+1) = g(a,n,h(a,n))$, as we expect for an operation defined recursively.
We have expressed the induction rule~\eqref{eq:iind} in syntax already; the reader is free to re-express it in more semantic language.
Note that $h$ in~\ref{item:nno-rec} is not required to be unique; thus $N$ is only a ``weak natural numbers object'' in \bT.
Such uniqueness is irrelevant for us, as with a defined equality on the codomain (see \cref{sec:sets,sec:linear-sets}) operations defined by recursion will always be unique as functions.

The affine version of induction is somewhat less obvious, but the following will be appropriate for us.

\begin{defn}
  In an affine hyperdoctrine, a \textbf{natural numbers type} is $(N,o,s)$ satisfying~\ref{item:nno-rec} of \cref{thm:inno} and such that for any predicate $P\in \cP(A\times N)$, the following entailment holds:
\begin{equation}
  P(0)\;\mand\; \oc\, \fa k. (P(k) \imp P(k+1)) \;\types_{P:\Omega^\dN}\; \fa n. P(n).\label{eq:lind}
\end{equation}
\end{defn}

Note that the induction step is marked with the modality $\oc$.
This is natural if we think of $\oc$ as denoting a hypothesis that can be used more than once, as the induction step must certainly be ``used'' $n$ times in order to conclude $P(n)$.
But it is also mandated by the antithesis translation.

\begin{lem}
  If an intuitionistic hyperdoctrine \cP contains a natural numbers type, so does its antithesis translation $\cP_{\pm}$.
\end{lem}
\begin{proof}
  The antithesis translation of~\eqref{eq:lind} consists of the following two entailments:
  \begin{align*}
    \pf P(0)\,\land\, \forall k. ((\pf P(k) \to \pf P(k+1)) \land (\rf P(k+1) \to \rf P(k))) \;&\types_{P:\Omega^\dN}\; \forall n. \pf P(n) \\
    \exists n. \rf P(n)\,\land\, \forall k. ((\pf P(k) \to \pf P(k+1)) \land (\rf P(k+1) \to \rf P(k))) \;&\types_{P:\Omega^\dN}\; \rf P(0)
\end{align*}
  Both can be proven easily from~\eqref{eq:iind}.
\end{proof}

On the other hand, if we drop the $\oc$ in~\eqref{eq:lind}, then its antithesis translation would also include a third entailment
\begin{equation}
  \pf P(0)\,\land\, \exists n. \rf P(n) \;\types_{P:\Omega^\dN} \exists k. (\pf P(k) \land \rf P(k+1))\label{eq:non-ind}
\end{equation}
which is equivalent to excluded middle.
Specifically, let $P(0) = (\top,\bot)$ and $P(n) = (\bot,\top)$ for $n\ge 2$, while $P(1) = (Q,\neg Q)$ for some arbitrary statement $Q$; then by~\eqref{eq:non-ind} we have either $\neg Q$ (if $k=0$) or $Q$ (if $k=1$).
Thus, we are forced to formulate affine induction as in~\eqref{eq:lind}.

\begin{prop}
  If \cP has a natural numbers type, so does the $\cP\ch$ defined in \cref{thm:int-comp,thm:aff-comp}.\qed
\end{prop}

\subsection{Conclusions}
\label{sec:conclusions}

In the rest of the paper, we will apply the antithesis translation to recover well-known intuitionistic definitions from naturally-defined affine ones.
On both intuitionistic and affine sides we will use higher-order logic with comprehension, operation types, a generic predicate, and a natural numbers type.
We have seen that this combination is preserved by the antithesis translation, and that nearly all naturally occurring models of intuitionistic logic satisfy it.
(One exception is that a tripos need not have comprehension, but we can add comprehensions to it as in \cref{thm:int-comp}.)
If necessary to disambiguate between affine and intuitionistic notions, we will use the annotations $\L$ for affine and $\I$ for intuitionistic; e.g.\ ``$\L$-predicate'' and ``$\I$-predicate'', or $\Omega^\L$ and $\Omega^\I$.

\begin{rmk}\label{rmk:affirm-axioms}
  We will frequently be discussing \emph{structured types} such as groups, rings, posets, topological spaces, and even sets (types with an equality predicate).
  Since ordinary first-order and higher-order logic do not allow quantification over types (i.e.\ ``for all types $A$'' internally to the logic), theorems relating to structured types are technically metatheorems.
  In particular, the \emph{axioms} of such structures are assumed \emph{entailments} $P\types Q$, or equivalently $\types (P\imp Q)$, and hence imply $\types \oc (P\imp Q)$.
  In other words, \emph{axioms are affirmative}.

  Another take on this is possible if we use a base theory with dependent types and type universes.
  In this case, we can quantify over all small types (those belonging to some universe \cU), and so it would be possible to assume non-affirmative axioms about a structured small type.
  However, if we also have comprehension, we can define \emph{types of} small structured types (e.g.\ the type of small groups), and in this case by \cref{thm:comp,thm:comp-aff} the axioms will again be affirmative, or at least arbitrarily duplicable.
\end{rmk}

\renumberthms{section}

\section{Intuitionistic sets and functions}
\label{sec:sets}

In most of the rest of the paper, we will first state definitions in affine logic and then translate them into intuitionistic logic.
But for sets and equality, we begin with the intuitionistic context to fix conventions.

As mentioned after \cref{thm:eq-aff}, we follow Bishop's dictum:
\begin{quote}\small
  The totality of all mathematical objects constructed in accordance with certain requirements is called a \emph{set}.
  The requirements of the construction, which vary with the set under consideration, determine the set.\dots
  Each set will be endowed with a binary relation $=$ of \emph{equality}.
  This relation is a matter of convention, except that it must be an \emph{equivalence relation}\dots
  \cite[\S2.1]{bb:constr-analysis}
\end{quote}
Thus a ``Bishop set'' has two ingredients: the ``requirements'', which we regard as the specification of a type, and the equality, which is an equivalence relation.

\begin{defn}\label{defn:iset}
  A \textbf{set} is a type $A$ with a predicate $\ieq$ on $A\times A$ such that
  \begin{equation*}
    \begin{array}{rll}
      &\types_{x: A}& x\ieq x\\
      x\ieq y &\types_{x,y: A}& y\ieq x\\
      (x\ieq y) \land (y\ieq z) &\types_{x,y,z: A}& x\ieq z.
    \end{array}
  \end{equation*}
\end{defn}

\begin{rmk}
  Suppose we start with a hyperdoctrine without comprehension, such as a tripos, apply \cref{thm:int-comp} to obtain comprehension, and then interpret \cref{defn:iset}.
  In terms of the \emph{original} hyperdoctrine, the resulting notion of ``set'' is essentially a \emph{partial} equivalence relation.
  It is common in tripos theory and realizability to work directly with partial equivalence relations.
  Instead, we divorce existence from equality, incorporating the former into a comprehension operation on types.
  This matches Bishop's two-stage conception better, as well as common mathematical practice (the construction of subsets is distinct from quotient sets), and generalizes better to the affine context.
\end{rmk}

\begin{eg}\label{eg:triv-set}
  If our base theory has Leibniz--Lawvere equality types $\sfeq_A$, then every type $A$ has a ``minimal'' structure of a set, with equality $\sfeq_A$.
\end{eg}

\begin{notn}
  If $A$ is a set and $P$ is a predicate on its underlying type, we implicitly give the comprehension $\setof{x: A | P(x)}$ the same equality predicate as $A$, making it again a set.
\end{notn}


\begin{eg}\label{eg:set-prod}
  If $A$ and $B$ are sets, their \textbf{cartesian product set} is the product type $A\times B$ with
  \(((x_1,y_1)\ieq (x_2,y_2))\defeq (x_1\ieq x_2)\land (y_1\ieq y_2)\).
\end{eg}

\begin{eg}\label{eg:set-omega}
  The type of propositions $\Omega$ is a set with 
  \( (P \ieq  Q) \defeq (P \iff Q). \)
\end{eg}


\begin{defn}\label{defn:rel}
  A \textbf{relation} on a set $A$ is a predicate $P$ on its underlying type such that
  \[ (x\ieq y) \land P(x) \types_{x,y: A} P(y). \]
  A relation is also called a \textbf{subset} of $A$, with $x\in P$ meaning $P(x)$.
  We overload notation by writing $P$ as $\setof{x: A | P(x)}$,
  though a subset is not itself a set.
\end{defn}

The relations on a given set are closed under all the logical operations.
Put differently, the subsets of a set are a sub-Heyting-algebra of the predicates on its underlying type.
We write $U\cap V \defeq \setof{x:A | (x\in U) \land (x\in V)}$ and so on.

\begin{defn}\label{defn:int-func}
  A \textbf{function} between two sets is an operation $f:B^A$ such that
  \[
  \begin{array}{rll}
    &\types_{x: A} & f(x) \in B\\
    (x_1\ieq x_2) & \types_{x_1,x_2: A}& (f(x_1)\ieq f(x_2)).
  \end{array}
  \]
  The \textbf{function set} is defined by
  \begin{align*}
    (A\to B) &\defeq \setof{f:B^A | \forall x_1^A x_2^A. ((x_1\ieq x_2) \to (f(x_1)\ieq f(x_2))) }\\
    (f \ieq  g) &\defeq \forall x^A. (f(x)\ieq g(x)).
  \end{align*}
\end{defn}

\noindent
Note the notation: the operation type is $B^A$, the function set is $A\to B$.

\begin{eg}\label{eg:powerset}
  We can regard a predicate on $A$ as an operation $P:\Omega^A$, and we have
  $(x\ieq y) \land P(x) \types_{x,y: A} P(y)$ if and only if
  $
  (x\ieq y) \types_{x,y: A} (P(x) \iff P(y))$.
  Thus, a relation on $A$ is the same as a function from $A$ to the set $\Omega$ (\cref{eg:set-omega}), so we can define the \textbf{power set} of $A$ as $\P A \defeq (A\to\Omega)$.
  Its induced equality relation is
\(
    (U \ieq  V) \defeq \forall x^A. (x\in U \iff x\in V).
\)
\end{eg}

One final remark concerns the following alternative definition of ``function''.

\begin{defn}\label{defn:anafun}
  For sets $A,B$, an \textbf{anafunction}\footnote{This term is inspired by the ``anafunctors'' of~\cite{makkai:avoiding-choice}.} is a relation $F$ on $A\times B$ that is total and functional, i.e.\ such that
  \[\begin{array}{rll}
    &\types_{x: A}& \exists y^B. F(x,y)\\
    F(x,y_1) \land F(x,y_2) &\types_{x: A, y_1:B,y_2: B} & (y_1\ieq y_2).
  \end{array}\]
\end{defn}

If $f:A\to B$ is a function, then $(f(x)\ieq y)$ is an anafunction; the principle of \textbf{function comprehension} (a.k.a.\ \textbf{unique choice}) says that every anafunction is of this form.
Function comprehension is not provable in first-order logic, higher-order logic, or logic-enriched type theory, and indeed fails in many triposes.
Nevertheless, constructivists of Bishop's school often assume it implicitly (one can argue for it by positing a closer relationship between ``operations'' and the existential quantifier than is implied by first-order or higher-order logic).

In the absence of function comprehension, it is often preferable to use anafunctions rather than functions.
For instance, this is how one builds the topos represented by a tripos (such as a realizability topos), and in particular how one recovers the correct internal logic of a topos from its tripos of subobjects.

\section{Affine sets and functions}
\label{sec:linear-sets}

We now switch to the affine context, for this section and the rest of the paper, except when discussing the antithesis translation.
In the definition of $\L$-sets we find our first additive/multiplicative bifurcation.

\begin{defn}\label{defn:set}
  A \textbf{set} is a type with a predicate $\leq$ on $A\times A$ such that
  \begin{alignat*}{2}
    &\;\types_{x\ocin A}\;&\;& x\leq x\\
    x\leq y &\;\types_{x,y\ocin A}&\;& y\leq x\\
    (x\leq y) \mand (y\leq z) &\;\types_{x,y,z\ocin A}&\;& x\leq z.\\
    \intertext{A set is \textbf{\aandish} if it satisfies the stronger transitivity axiom}
    (x\leq y) \aand (y\leq z) &\;\types_{x,y,z\ocin A}&\;& x\leq z.
  \end{alignat*}
\end{defn}

\begin{eg}\label{eg:triv-Lset}
  As in the intuitionistic case, if our first-order affine logic has Leibniz--Lawvere equality types $\sfeq_A$ (\cref{defn:hyd-eq}), then every type $A$ has a ``minimal'' structure of an $\L$-set, with equality $\sfeq_A$.
  This is less useful than in the intuitionistic case (\cref{eg:triv-set}), however, since by \cref{thm:eq-aff} any such $\L$-set has affirmative equality, while we are often interested in $\L$-sets with non-affirmative equality.
\end{eg}

\begin{notn}\label{notn:ocin}
  Recall that if $P$ is an $\L$-predicate on an $\L$-type $A$, we write $\ocsetof{x:A | P(x)}$ for the comprehension type.
  If $A$ is given as a set, we implicitly give $\ocsetof{x: A | P(x)}$ the same equality predicate. 
\end{notn}

Under the antithesis translation, an $\L$-set is an $\I$-type with \emph{two} binary predicates $(\ieq,\ineq)$ such that
\[
\begin{array}{rll}
  &\types_{x,y: A}& \neg ((x\ieq y)\land (x\ineq y))\\
  &\types_{x: A}& x\ieq x\\
  x\ieq y &\types_{x,y: A} & y\ieq x\\
  x\ineq y &\types_{x,y: A} & y\ineq x\\
  (x\ieq y) \land (y\ieq z) &\types_{x,y,z: A}& x\ieq z\\
  (x\ineq z) \land (y\ieq z) &\types_{x,y,z: A}& x\ineq y\\
  (x\ineq z) \land (x\ieq y) &\types_{x,y,z: A}& y\ineq z.
\end{array}
\]
The axioms involving only $\ieq $ say that $(A,\ieq )$ is an $\I$-set,
and the last two axioms say that $\ineq$ is an $\I$-relation (\cref{defn:rel}) on $A\times A$.
Given this, the first axiom is equivalent to $\types_{x: A}\neg (x\ineq x)$.
Thus we have:

\begin{thm}\label{thm:eq}
  Under the antithesis translation:
  \begin{enumerate}
  \item An $\L$-set is an $\I$-set equipped with an \textbf{inequality relation}: a relation $\ineq$ such that $\neg (x\ineq x)$ and $(x\ineq y) \to (y\ineq x)$ (i.e.\ it is irreflexive and symmetric).\label{item:eq1}
  \item It is \aandish if and only if $\ineq$ is an \textbf{apartness}, i.e.\ $(x\ineq z) \to (x\ineq y) \lor (y\ineq z)$.\label{item:eq2}
  \item Its equality is affirmative if and only if $\ineq$ is \textbf{denial}, $(x\ineq y) \logeq \neg (x\ieq y)$.\label{item:eq3}
  \item Its equality is refutative if and only if $\ineq$ is \textbf{tight}: $\neg (x\ineq y) \logeq (x\ieq y)$.\label{item:eq4}\qed
  \end{enumerate}
\end{thm}

\begin{eg}\label{eg:linear-set-prod}\label{defn:prod-apart}
  If $A$ and $B$ are sets, their \textbf{cartesian product set} is the cartesian product type $A\times B$ with
  \[((x_1,y_1)\leq (x_2,y_2))\defeq (x_1\leq x_2)\aand (y_1\leq y_2).\]
  Under the antithesis translation, this yields the cartesian product of $\I$-sets with the disjunctive \textbf{product inequality} (or \textbf{product apartness}):
  \[ ((x_1,y_1)\ineq (x_2,y_2)) \defeq (x_1\ineq x_2) \lor (y_1\ineq y_2).\]
\end{eg}

\begin{eg}\label{thm:linear-set-tens}
  The \textbf{tensor product set} $A\ltens B$ has the same underlying type, but with equalities combined multiplicatively:
  \[((x_1,y_1)\lmeq (x_2,y_2))\defeq (x_1\leq x_2)\mand (y_1\leq y_2).\]
  In the antithesis translation, thus yields the weaker inequality
  \[ ((x_1,y_1)\imneq (x_2,y_2)) \defeq ((x_1\leq x_2)\to (y_1\ineq y_2)) \land ((y_1\leq y_2) \to (x_1\ineq x_2)). \]
\end{eg}

If $A$ and $B$ have affirmative equality, so does $A\ltens B$, but $A\times B$ need not.
If $A$ and $B$ have \aandish or refutative equality, so does $A\times B$, but $A\ltens B$ need not.

\begin{eg}\label{eg:lin-set-omega}
  The type $\Omega$ is a set with
  \begin{equation}
    (P \leq  Q) \defeq (P \liff Q) \defeq (P \imp Q) \aand (Q\imp P).\label{eq:omega-eq}
  \end{equation}
  In the antithesis translation, this yields
  \begin{align*}
    (P\ieq Q) &\logeq (\pf P \iff \pf Q) \land (\rf P \iff \rf Q).\\
    (P\ineq Q) &\logeq (\pf P \land \rf Q) \lor (\rf P \land \pf Q).
  \end{align*}
  We could also use $\mand$ in~\eqref{eq:omega-eq}, but using $\aand$ yields a more useful $\ineq$ and has better formal properties (see \cref{eg:linear-prel,sec:posets}).
  In neither case is the equality \aandish, nor is it affirmative nor refutative even if $P$ and $Q$ are both one or the other.
\end{eg}

\begin{rmk}\label{rmk:nat-strong}
  The notion of ``strong set'' is quite natural under the antithesis translation, since apartness relations are well-studied in intuitionistic constructive mathematics.
  However, to a reader familiar with linear logic (and particularly with linear proof theory), the $\aand$-transitivity of a strong set may seem unreasonably strong.
  An assumption of $(x\leq y)\aand (y\leq z)$ means that we can choose to use either $x\leq y$ \emph{or} $y\leq z$ but not both, so how could we ever hope to prove $x\leq z$?

  In fact, however, there are many sets that \emph{can} be proven to be strong inside affine logic.
  The key is that we don't have to \emph{start} by deciding which of $x\leq y$ and $y\leq z$ to use: we can decompose $x$, $y$, and $z$ and use the definition of $\leq$ to make case distinctions, and then make different choices of $x\leq y$ and $y\leq z$ in different cases.

  A paradigmatic example is the natural numbers $\dN$, for which we define equality recursively in the usual way:
  \begin{alignat*}{2}
    (0\leq 0)&\defeq \top &\qquad
    (0\leq y+1) &\defeq \bot\\
    (x+1\leq y+1) &\defeq (x=y) &\qquad
    (x+1\leq 0) &\defeq \bot.
  \end{alignat*}
  We prove $(x\leq y)\aand (y\leq z) \types_{x,y,z:\dN} (x\leq z)$ by induction on $x,y,z$.
  The case when $x$ and $z$ are both $0$ is trivial.
  If $x$ is $0$ but $z$ is a successor, then either $y$ is $0$, in which case we can use $y\leq z$ to get a contradiction, or $y$ is a successor, in which we can use $x\leq y$ to get a contradiction.
  The case when $x$ is a successor and $z$ is $0$ is symmetric.
  Finally, if $x$ is $x'+1$ and $z$ is $z'+1$, then if $y$ is $0$ we can use either $x\leq y$ or $y\leq z$ to get a contradiction, while if $y$ is a successor $y'+1$ then our goal reduces to the inductive hypothesis $(x'\leq y')\aand (y'\leq z') \types_{x',y',z':\dN} (x'\leq z')$.
\end{rmk}

We now move on to discuss $\L$-relations and subsets.

\begin{defn}\label{defn:lin-rel}
  A \textbf{relation} on a set $A$ is a predicate $P$ such that
  \[ (x\leq y) \mand P(x) \types_{x,y\ocin A} P(y). \]
  A relation is \textbf{\aandish} if
  \[ (x\leq y) \aand P(x) \types_{x,y\ocin A} P(y). \]
  We also refer to a relation as a \textbf{subset}, writing $x\lin P$ instead of $P(x)$, and $\lsetof{x: A|P(x)}$ for $P$ itself.
  (Unlike in the intuitionistic case, we distinguish this notationally from a comprehension $\ocsetof{x:A|P(x)}$, since the latter discards refutational information.)
\end{defn}

\begin{thm}\label{thm:rel}
  Let $U$ be an $\L$-subset of an $\L$-set $A$.
  In the antithesis translation:
  \begin{enumerate}
  \item $U$ is a \textbf{complemented subset} as in~\cite[Chapter 3, Definition (2.2)]{bb:constr-analysis}: a pair of $\I$-subsets $(U,\cancel U)$ of $A$ such that\label{item:rel1}
    \[ (x\in U) \land (y\in \cancel U) \types_{x,y: A} (x\ineq y) .\]
  \item It is \aandish if and only if $\cancel U$ is \textbf{strongly extensional} (also called \textbf{$\ineq$-open}):\label{item:rel2}
    \[
    \begin{array}{rll}
      (y\in \cancel U) &\types_{x,y: A}& (x\ineq y) \lor (x\in\cancel U).
    \end{array}
    \]
  \end{enumerate}
\end{thm}
\begin{proof}
  The subset condition $(x\leq y) \mand (x\lin U) \types_{x,y\ocin A} (y\lin U)$ becomes
  \[
  \begin{array}{rll}
    (x\ieq y) \land (x\in U) &\types_{x,y: A}& (y\in U)\\
    (x\ieq y) \land (y\in \cancel U) &\types_{x,y: A}& (x\in \cancel U)\\
    (x\in U) \land (y\in \cancel U) &\types_{x,y: A}& (x\ineq y).
  \end{array}
  \]
  The first two say that $U$ and $\cancel U$ are $\I$-subsets, and the last is the ``strong disjointness'' condition in~\ref{item:rel1}.
  The ``strong extensionality'' condition in~\ref{item:rel2} is exactly the contrapositive information arising from the \aandish subset condition.
\end{proof}

\begin{defn}\label{defn:func}
  A \textbf{function} between two sets is an operation $f:B^A$ such that
  \[\begin{array}{rll}
    (x_1\leq x_2) & \types_{x_1,x_2\ocin A}& (f(x_1)\leq f(x_2)).
  \end{array}\]
  The \textbf{function set} is defined by
  \begin{align*}
    (A\to B) &\defeq \ocsetof{f:B^A | \fa x_1^A x_2^A. ((x_1\leq x_2) \imp (f(x_1)\leq f(x_2))) }\\
    (f \leq  g) &\defeq \fa x^A. (f(x)\leq g(x)).
  \end{align*}
\end{defn}

\begin{thm}\label{thm:func}
  In the antithesis translation, an $\L$-function $f:A\to B$ is an $\I$-function that is \textbf{strongly extensional}, i.e.\ $(f(x_1)\neq f(x_2)) \types_{x_1,x_2: A} (x_1\neq x_2)$.
  The inequality on $A\to B$ is
  \( (f\neq g) \logeq \exists x^A. (f(x)\neq g(x))\).\qed
\end{thm}

\begin{eg}\label{eg:linear-prel}
  We have $(x\leq y)\mand P(x) \vdash P(y)$ iff $(x\leq y) \vdash (P(x) \imp P(y))$, and symmetry of $\leq $ then implies $(x\leq y) \vdash (P(x) \imp P(y)) \aand (P(y) \imp P(x))$, i.e.\ $(x\leq y) \types (P(x) \liff P(y))$.
  Therefore, relations on $A$ are the same as functions from $A$ to the set $\Omega$ from \cref{eg:lin-set-omega}.
  (Note that this requires the $\aand$ in~\eqref{eq:omega-eq}.)
  Thus we can define the \textbf{power set} of $A$ to be $\P A \defeq (A\to \Omega)$.
  Its induced equality is
  \[ (U \leq  V) \defeq \fa x^A. ((x\lin U) \liff (x\lin V)). \]
  In the antithesis translation, we have
  \[ (U\ineq V) \defeq \exists x^A. ((x\in U \land x\in \cancel V) \lor (x\in\cancel U \land x\in V)). \]
\end{eg}

\begin{eg}\label{eg:twovar-strext}
  In the antithesis translation, an $\L$-function $f:A\times B\to C$ must be strongly extensional for the disjunctive product inequality, $(f(x_1,y_1)\ineq f(x_2,y_2)) \types (x_1\ineq x_2) \lor (y_1\ineq y_2)$.
  By contrast,
  an $\L$-function $f:A\ltens B\to C$ need only be strongly extensional in each variable separately: $(f(x,y_1) \ineq f(x,y_2)) \types (y_1\ineq y_2)$ and $(f(x_1,y) \ineq f(x_2,y)) \types (x_1\ineq x_2)$.
  Both are useful notions; see \cref{eg:lpo}.
\end{eg}

\begin{eg}
  In particular, functions from $B$ to $\P A = (A\to\Omega)$ classify subsets not of $A\times B$, but of $A\ltens B$.
  In the antithesis translation, an $\L$-subset of $A\ltens B$ is a pair of $\I$-subsets $U,\cancel U\subseteq A\times B$ such that
  \begin{alignat*}{2}
    ((x,y_1)\in U) \land ((x,y_2) \in \cancel U) &\types&\;& (y_1\ineq y_2)\\
    ((x_1,y)\in U) \land ((x_2,y) \in \cancel U) &\types&\;& (x_1\ineq x_2)\\
    \intertext{whereas an $\L$-subset of $A\times B$ satisfies the stronger condition}
    ((x_1,y_1)\in U) \land ((x_2,y_2) \in \cancel U) &\types&\;& (x_1\ineq x_2) \lor (y_1\ineq y_2).
  \end{alignat*}
\end{eg}

The $\L$-relations on an $\L$-set are closed under the additive connectives, as well as linear negation.
This defines the additive operations of set algebra:
\begin{alignat*}{2}
  U \lcap V &= \lsetof{x | (x\lin U) \aand (x\lin V)} &\qquad
  U \lcup V &= \lsetof{x | (x\lin U) \aor (x\lin V)} \\
  \textstyle\lbigcap_i U_i &= \lsetof{x | \fa i. (x\lin U_i) } &\qquad
  \textstyle\lbigcup_i U_i &= \lsetof{x | \ex i. (x\lin U_i) } \\
  \cm U &= \lsetof{x | \smash{\ntp{x\lin U}}} &\qquad
  \lempty &= \lsetof{x | \bot}.
\end{alignat*}
Here the index $i$ in $\lbigcap_i$ and $\lbigcup_i$ belongs to some type $I$, while $U$ is a predicate on $I\times A$ that respects the equality of $A$.
In particular, it might be the case that $I$ is itself an $\L$-set and $U$ is a predicate on $I\times A$ or $I\mand A$.

We write $U\lsub V$ to mean $\fa x^A. ((x\lin U)\imp (x\lin V))$; in the antithesis translation this means that $U\subseteq V$ and $\cancel V \subseteq \cancel U$.
Since $\fa$ commutes with $\aand$, we have $(U\leq V) \logeq ((U\lsub V) \aand (V\lsub U))$.
By duality, $U\lnsub V$ means $\ex x^A. ((x\lin U) \mand (x\lnin V))$.

Like linear negation, the complement of $\L$-subsets is involutive ($\cmcm{U} = U$) but not Boolean: $U\lcup \cm U\qeq A$ and $U\lcap \cm U\qeq \lempty$ both assert that $U$ is decidable.

\begin{lem}
  In the antithesis translation, an $\L$-subset is \textbf{nonempty}, i.e.\ $U\lneq \lempty$, if and only if its affirmative part is $\I$-inhabited, i.e.\ $\exists x^A. (x\in U)$.
\end{lem}
\begin{proof}
  The definition of inequality on $\P A$ gives
  \[ \pf{(U\lneq \lempty)}
  \;\defeq\; \exists x. ((x\in U \land x\in A) \lor (x\in \emptyset \land x\in \cancel U))
  \;\logeq\; \exists x. (x\in U). \qedhere \]
\end{proof}

  Multiplicatives and exponentials do not generally preserve subsets, but they do induce operations on subsets by a reflection or coreflection process:
  \begingroup\allowdisplaybreaks
  \begin{align*}
    U \lmcap V &\defeq \lsetof{x:A | \ex y^A. ((x\leq y) \mand (y\lin U) \mand (y\lin V))}\\
    U \lmcup V &\defeq \lsetof{x:A | \fa y^A. ((x\leq y) \imp ((y\lin U) \mor (y\lin V)))}\\
    \ochat U &\defeq \lsetof{x:A | \ex y^A. ((x\leq y) \mand \oc (y\lin U))}\\
    \wnchk U &\defeq \lsetof{x:A |\fa y^A. ((x\leq y) \imp \wn (y\lin U))}.
  \end{align*}
  \endgroup
  The poset of subsets of $A$ thereby becomes semicartesian and $\ast$-autonomous with a Seely comonad.
  In particular, as a replacement for the false equalities $U\lcup \cm U\qeq A$ and $U\lcap \cm U\qeq \lempty$ we have the true ones $U\lmcup \cm U  = A$ and $U \lmcap \cm U = \lempty$,
  and the $\ochat$-coalgebras form a ``Heyting algebra of affirmative subsets''.
  In the antithesis translation, $\ochat U$ is the affirmative part of $U$ with its \textbf{inequality complement}:
  \begin{alignat*}{2}
    (x\in \ochat U) &\defeq (x\in U) & \qquad
    (x\in \cancel{\ochat U}) &\defeq \forall y^A. ((y\in U) \to (x\ineq y)).
  \end{alignat*}
  Thus an ``affirmative subset'' (i.e.\ $U= \ochat U$) is determined by an ordinary $\I$-subset.

\begin{rmk}\label{rmk:set-hyd}
  If $\ASet$ denotes the category of $\L$-sets and functions\footnote{Strictly speaking we should either quotient these functions by pointwise equality or consider $\ASet$ to be some sort of ``e-category'', but we will not delve into these waters.}, we have constructed a pseudofunctor $\cP : \ASet\op\to \cAff$, which is in fact an affine hyperdoctrine --- although, as suggested in \cref{rmk:aff-quant}, we are generally more interested in quantifiers for the projections $A \mand B \to A$ than $A\times B\to A$.
  This affine hyperdoctrine over $\ASet$ seems analogous to the ``tripos-to-topos'' construction~\cite{hjp:tripos} in intuitionistic logic, but it differs in two important ways.

  Firstly, it is unclear whether the relations on an $\L$-set $A$ can be recovered from the category $\ASet$ as any sort of ``subobject''.
  \cref{thm:comp-aff} is discouraging in this regard.
  Secondly, it is unclear whether the equality relation on an $\L$-set $A$ admits any characterization in terms of this hyperdoctrine over $\ASet$: it cannot be the Leibniz--Lawvere equality, since by \cref{thm:eq-aff} the latter is affirmative.

  For these reasons, we will continue to work Bishop-style, with $\L$-sets \emph{defined} to be types equipped with an equality predicate.
  However, it seems possible that this affine hyperdoctrine over $\ASet$ might shed some semantic light on the question of affine type dependency (see \cref{rmk:dtt}).
\end{rmk}

Finally, we note that unique existence and ``anafunctions'' (see \cref{defn:anafun}) also behave sensibly.
  Recall that classically we can express ``there is at most one $x$ with $P(x)$'' either as ``for all $x,y$, if $P(x)$ and $P(y)$, then $x=y$'' or ``there do not exist $x,y$ with $x\neq y$ such that $P(x)$ and $P(y)$''.
  Intuitionistically these are no longer equivalent (unless $\ineq$ is tight), and only the former is ``correct''.
  But linearly they are again equivalent:
  \begin{equation*}
    \fa x y. ((P(x) \mand P(y)) \imp (x\leq y))
    \;\logeq\;
    \ntP{\ex x y. ((x\lneq y) \mand P(x) \mand P(y))}.
  \end{equation*}
  In the antithesis translation, these statements yield the ``correct'' intuitionistic version augmented by a strong uniqueness ``if $x\ineq y$ and $P(x)$, then $\cancel P(y)$''.
  An even stronger sort of uniqueness would arise from the \aandish linear condition
  \[ \fa x y. ((P(x) \aand P(y)) \imp (x\leq y)), \]
  which in the antithesis translation yields ``if $x\ineq y$, then either $\cancel P(x)$ or $\cancel P(y)$''.

\begin{defn}\label{thm:lin-anafun}
  For $\L$-sets $A,B$, an \textbf{anafunction} from $A$ to $B$ is a relation $F$ on $A\ltens B$ that is total and functional, i.e.\ such that
  \[\begin{array}{rll}
    &\types_{x\ocin A}& \ex y^B. F(x,y)\\
    F(x,y_1) \mand F(x,y_2) &\types_{x\ocin A, y_1:B ,y_2\ocin B}&  (y_1\leq y_2).
  \end{array}\]
\end{defn}

\begin{thm}\label{thm:anafun}
  In the antithesis translation, an $\L$-anafunction from $A$ to $B$ corresponds to an $\I$-anafunction that is ``strongly extensional'' in the sense that
  \[ F(x_1,y_1) \land F(x_2,y_2) \land (y_1\ineq y_2) \types_{x_1,x_2: A, y_1,y_2: B} (x_1\ineq x_2). \]
\end{thm}
\begin{proof}
  An $\L$-anafunction consists of two $\I$-relations $F,\cancel F$ on $A\times B$ such that
  \[
  \begin{array}{rll}
    & \types_{x: A, y: B} &\neg (F(x,y) \land \cancel F(x,y))\\
    F(x,y_1) \land \cancel F(x,y_2) & \types_{x: A, y_1:B,y_2: B} & (y_1\ineq y_2)\\
    F(x_1,y) \land \cancel F(x_2,y) &\types_{x_1:A,x_2: A, y: B} & (x_1\ineq x_2)\\
    &\types_{x: A} & \exists y^B. F(x,y)\\
    &\types_{x: A}& \neg \forall y^B. \cancel F(x,y)\\
    F(x,y_1) \land F(x,y_2) &\types_{x: A, y_1:B,y_2: B}&  (y_1\ieq y_2)\\
    F(x,y_1) \land (y_1\ineq y_2) &\types_{x: A, y_1:B,y_2: B} & \cancel F(x,y_2).
  \end{array}
  \]
  The fourth and sixth axioms say that $F$ is an $\I$-anafunction.
  Given this, the second and seventh say $\cancel F(x,y_2) \logeq \exists y_1^B. (F(x,y_1) \land (y_1\ineq y_2))$, which implies the first and fifth, and unravels the third to the claimed strong extensionality property.
\end{proof}

A function is strongly extensional just when its corresponding anafunction is.
Thus, a function comprehension principle is equally sensible affinely as intuitionistic\-ally, and in its absence we can once again work with anafunctions instead.
Moreover, \cref{thm:anafun} implies that function comprehension is preserved by the antithesis construction: if an intuitionistic hyperdoctrine \cP satisfies function comprehension, so does the affine hyperdoctrine $\cP_{\pm}$.

\section{Algebra}
\label{sec:algebra}

Roughly speaking, there are two approaches to intuitionistic constructive algebra.
The first uses apartness only minimally; inequality usually means denial $\neg (x\ieq y)$ and is avoided as much as possible.
For instance, apartness relations are absent from~\cite{johnstone:fields}, and are only rarely used in~\cite{mrr:constr-alg}.
The second approach equips all sets with inequalities (often tight apartnesses),\footnote{Recall that for us, an \emph{inequality relation} is irreflexive and symmetric, an \emph{apartness relation} additionally satisfies $(x\ineq z) \types (x\ineq y) \lor (y\ineq z)$, and is \emph{tight} if $\neg (x\ineq y) \types (x\ieq y)$.
  In~\cite{tvd:constructivism-ii} an ``apartness'' is necessarily tight (otherwise they speak of a ``pre-apartness''), and in~\cite{mrr:constr-alg} it seems that \emph{no} axioms are demanded in general of a relation called ``inequality''.}
and all classical definitions are augmented by ``strong negative'' information such as anti-subgroups and anti-ideals.
This is the tradition of Heyting; see~\cite[Chapter 8]{tvd:constructivism-ii}.

The second approach gives more refined information.
For instance, the 
real numbers are a field in the strong sense that any number \emph{apart} from $0$ is invertible; but without apartness, all we can say is that they are a local ring in which every noninvertible element is zero.
However, carrying apartness relations around is tedious and error-prone, and not every algebraic structure admits a natural apartness:
  \begin{quote}\small
    We could demand that every set come with an inequality, putting inequality on the same footing as equality\dots 
    With such an approach, whenever we construct a set we must put an inequality on it, and we must check that our functions are strongly extensional.
    This is cumbersome and easily forgotten, resulting in incomplete constructions and incorrect proofs.~\cite[p31]{mrr:constr-alg}
  \end{quote}
Moreover, rewriting all of algebra in ``dual'' form looks very unfamiliar to the classical mathematician, and even a constructive mathematician may find it unaesthetic.

The antithesis translation resolves this by automatically handling the ``bookkeeping'' of apartness relations, allowing familiar-looking definitions (written in affine logic) to nevertheless carry the correct constructive meaning (when translated into intuitionistic logic).
It also reveals the above two approaches as ends of a continuum: $\I$-sets with denial inequality are the $\L$-sets with affirmative equality, while $\I$-sets with a (tight) apartness are the $\L$-sets with a (refutative) \aandish equality.
There are also natural examples in between; see \cref{eg:lpo}.

\begin{defn}\label{defn:group}
  A \textbf{group} is an ($\L$-)set $G$ together with an element $e\ocin G$ and functions $m:G\ltens G\to G$ and $i:G\to G$ such that
  \begin{alignat*}{4}
    &\types_{x\ocin G}&\;& m(x,e) \leq x &\qquad
    &\types_{x\ocin G}&\;& m(x,i(x)) \leq e\\
    &\types_{x\ocin G}&\;& m(e,x) \leq x &\qquad
    &\types_{x\ocin G}&\;& m(i(x),x) \leq e\\
    &\types_{x,y,z\ocin G}&\;& m(m(x,y),z) \leq m(x,m(y,z)).
  \end{alignat*}
  A group is \textbf{\aandish} if $m$ is a function on $G\times G$.
\end{defn}

As usual, we write $xy$ and $x^{-1}$ instead of $m(x,y)$ and $i(x)$.

\begin{thm}\label{thm:std-gp}
  In the antithesis translation, an $\L$-group consists of an $\I$-group equipped with an inequality relation such that
  \[\begin{array}{rll}
    x^{-1}\ineq y^{-1} &\types_{x,y: G} & x\ineq y\\
    x u \ineq x v &\types_{x,u,v: G} & u\ineq v\\
    x u \ineq y u &\types_{x,y,u: G} & x\ineq y.
  \end{array}\]
  The extra condition for $G$ to be \aandish is
  \begin{alignat}{2}\label{eq:strong-gp}
    (x u \ineq y v) &\types_{x,y,u,v: G} &\;& (x \ineq y) \lor (u\ineq v)
  \end{alignat}
  which is equivalent to $\ineq$ being an apartness.
  In particular:
  \begin{itemize}
  \item An $\L$-group with affirmative equality is precisely an $\I$-group.
  \item A \aandish $\L$-group with refutative equality
    is precisely a \textbf{group with apartness relation} in the sense of~\cite[8.2.2]{tvd:constructivism-ii}, i.e.\ an $\I$-group with a tight apartness for which the group operations are strongly extensional.
  \item An arbitrary $\L$-group is precisely an $\I$-group with a (symmetric irreflexive) \textbf{translation invariant} inequality as in~\cite[Exercise II.2.5]{mrr:constr-alg}.
    \qed
  \end{itemize}
\end{thm}

The fact that~\eqref{eq:strong-gp} is equivalent to $\ineq$ being an apartness is a standard exercise in constructive algebra.
In fact, it can be proven internally in affine logic that an $\L$-group is \aandish if and only if it has \aandish equality.

\begin{defn}
  A \textbf{subgroup} of a group $G$ is a subset $H\lsub G$ such that
  \begin{alignat*}{2}
    &\types &\;& (e\lin H)\\
    (x\lin H) &\types_{x\ocin G} &\;& (x^{-1} \lin H)\\
    (x\lin H) \mand (y\lin H) &\types_{x,y\ocin G} &\;& (xy \lin H).\\
    \intertext{A subgroup is \textbf{\aandish} if it satisfies the stronger condition}
    (x\lin H) \aand (y\lin H) &\types_{x,y\ocin G} &\;& (xy \lin H).
  \end{alignat*}
\end{defn}





\begin{thm}\label{thm:std-subgp}
  In the antithesis translation, a $\L$-subgroup $H$ of $G$ is:
  \begin{enumerate}
  \item An $\I$-subgroup $H$ of the $\I$-subgroup $G$; together with
  \item An $\I$-subset $\cancel H$ of $G$ satisfying the following axioms:\label{item:stdsubgp2}
    \begin{alignat*}{2}
      (x\in H) \land (y\in \cancel H) &\types_{x,y: G} &\;& (x\ineq y)\\
      (x^{-1}\in \cancel H) &\types_{x\in G} &\;& (x\in \cancel H)\\
      (xy\in \cancel H) \land (x\in H) &\types_{x,y: G} &\;& (y\in \cancel H)\\
      (xy\in \cancel H) \land (y\in H) &\types_{x,y: G} &\;& (x\in \cancel H).
    \end{alignat*}
  \end{enumerate}
  Moreover:
  \begin{itemize}
  \item $H$ is \aandish iff the last two axioms are replaced by the following stronger one:
    \[ (xy\in \cancel H) \types_{x,y: G} (x\in \cancel H) \lor (y\in \cancel H). \]
  \item An affirmative $\L$-subgroup of an affirmative $\L$-group is precisely an $\I$-subgroup of an $\I$-group, together with its logical complement $\cancel H \defeq \setof{x\in G | \neg(x\in H)}$.
  \item If $G$ is refutative and \aandish, then $H$ is refutative
and \aandish if and only if $\cancel H$ is an \textbf{antisubgroup compatible with the apartness} in the sense of~\cite[8.2.4]{tvd:constructivism-ii} together with its logical
complement.
\qed
  \end{itemize}
\end{thm}

\begin{defn}
  An $\L$-subgroup $H$
  is \textbf{normal} if $(x\lin H) \types (y x y^{-1} \lin H)$.
\end{defn}

In the antithesis translation, if $H$ and $G$ are affirmative then normality reduces to ordinary normality, whereas if they are \aandish and refutative it reduces to normality for an antisubgroup~\cite[8.2.7]{tvd:constructivism-ii}.

\begin{thm}
  Let $H$ be a normal subgroup of $G$.  Then $(x \leq_H y) \defeq (x y^{-1} \lin H)$ defines a new equality predicate on the underlying type of $G$, and the resulting set is again a group, denoted $G/H$.
\end{thm}
\begin{proof}
  The closure axioms of a subgroup directly imply the axioms of an equality predicate.
  It remains to show that $m$ and $i$ are functions $G/H \ltens G/H \to G/H$ and $G/H \to G/H$.
  For the first, we have
  \begin{alignat*}{2}
    (x\leq_H y) \mand (w\leq_H z)
    &\logeq (x y^{-1} \lin H)\mand (w z^{-1} \lin H)\\
    &\types (x y^{-1} \lin H)\mand (y w z^{-1} y^{-1} \lin H) &&\qquad \text{(by normality)}\\
    &\types (x y^{-1} y w z^{-1} y^{-1} \lin H)\\
    &\logeq ((x w) (y z)^{-1}\lin H)\\
    &\logeq (x w \leq_H y z).
  \end{alignat*}
  Similarly, for the second we have
  \[ (x\leq_H y) \logeq (x y^{-1} \lin H) \types (y^{-1} x y^{-1} y \lin H) \logeq (y^{-1} x \lin H) \logeq (y^{-1} \leq_H x^{-1}).\qedhere \]
\end{proof}


\begin{eg}\label{eg:lpo}
  Let $G = \mathbf{2}^{\dN}$ be the set of infinite binary sequences, with pointwise addition mod 2, and
  \( H \defeq \lsetof{x\in G | \ex m. \fa n. ((m \lle n) \imp (x_n \leq 0))} \).
  Then $G$ is a \aandish group with refutative equality, while $H$ is a normal subgroup that is neither \aandish, affirmative, nor refutative.
  In the quotient $G/H$ we have
  \begin{align*}
    (x\leq 0) &\logeq \ex m. \fa n. ((m\lle n) \imp (x_n \leq 0))\\
    (x\lneq 0) &\logeq \fa m. \ex n. ((m\lle n) \mand (x_n \leq 1)).
  \end{align*}
  That is, $x\leq 0$ if $x$ is eventually $0$, and $x\lneq 0$ if $x$ is $1$ infinitely often.
  Neither of these is the Heyting negation of the other, so $G/H$ is neither affirmative nor refutative.
  Similarly, $G/H$ is not \aandish, so in the antithesis translation its inequality is not an apartness and its multiplication is not strongly extensional \emph{for the disjunctive product inequality}, though it is for the weaker equality on $G/H\mand G/H$.

  In~\cite[p31]{mrr:constr-alg} this example is used to argue that not all sets should have inequalities.
  From our perspective, it shows instead that not all groups should be required to be \aandish.
\end{eg}

A \textbf{(commutative) ring} is an abelian group $(R,+,0)$ with a multiplication function $\cdot :R\ltens R\to R$ and unit $1:R$ satisfying the usual axioms; it is \textbf{\aandish} if both $+$ and $\cdot$ are defined on $R\times R$.
In the antithesis translation:
\begin{itemize}
\item An affirmative $\L$-ring is an ordinary $\I$-ring.
\item A \aandish refutative $\L$-ring is a \textbf{ring with apartness} as in~\cite[8.3.1]{tvd:constructivism-ii} (except that they also assume $0\ineq 1$). 
\item A general $\L$-ring is an $\I$-ring with an inequality such that $(x\ineq y) \iff (x-y \ineq 0)$ and $(xy \ineq 0) \to (y\ineq 0)$.
\end{itemize}
An \textbf{ideal} is an additive subgroup $J$ with $(x\lin J) \types (xy\lin J)$.
In the antithesis translation, in the affirmative case this is an ordinary $\I$-ideal, while in the \aandish refutative case it is an \textbf{anti-ideal}~\cite[8.3.6]{tvd:constructivism-ii}: an additive antisubgroup $\cancel J$ with $(xy\in\cancel J) \to (x\in \cancel J) \land (y\in\cancel J)$.
The quotient $R/J$ of an $\L$-ring by an ideal is straightforward, and its antithesis translation yields the apartness on the quotient of an apartness ring by the complement of an anti-ideal~\cite[8.3.8]{tvd:constructivism-ii}.

\begin{defn}
  Let $J$ be an ideal of the $\L$-ring $R$ that is \textbf{proper}, i.e.\ $1\lnin J$.
  \begin{itemize}
  \item $J$ is \textbf{\aprime} if $(xy\lin J) \types (x\lin J) \aor (y\lin J)$.
  \item $J$ is \textbf{\mprime} if $(xy\lin J) \types (x\lin J) \mor (y\lin J)$.
  \item $R$ is \textbf{\aint} if $(0)$ is proper and \aprime.
  \item $R$ is \textbf{\mint} if $(0)$ is proper and \mprime.
  \end{itemize}
\end{defn}

If $J$ is proper, then $R/J$ is \aint or \mint exactly when $J$ is \aprime or \mprime, respectively.
In the antithesis translation:
\begin{itemize}
\item A \aprime affirmative $\L$-ideal in an affirmative $\L$-ring is a proper $\I$-ideal such that $(xy\in J) \types (x\in J) \lor (y\in J)$.
  An affirmative $\L$-ring is \aint if $\neg(0\ieq 1)$ and $(xy\ieq 0) \types (x\ieq 0)\lor (y\ieq 0)$;
this is~\cite[axiom I1]{johnstone:fields}.
\item Similarly, an affirmative $\L$-ring is \mint if it satisfies $\neg (0\ieq 1)$ and \cite[axiom I2]{johnstone:fields}: $(xy\ieq 0) \land \neg (x\ieq 0) \to (y\ieq 0)$.
\item A \mprime \aandish refutative $\L$-ideal in a \aandish refutative $\L$-ring is an anti-ideal $\cancel J$ in an $\I$-ring with apartness that is proper ($1\in \cancel J$) and such that $(x\in \cancel J) \land (y\in\cancel J) \types (x y \in \cancel J)$, i.e.\ a \textbf{prime anti-ideal} as in~\cite[8.3.10]{tvd:constructivism-ii}.
\item Finally, an arbitrary $\L$-ring is \mint if and only if $1\ineq 0$ and we have $(x\ineq 0) \land (xy \ieq 0) \to (y\ieq 0)$ and also $(x\ineq 0) \land (y\ineq 0) \to (x y \ineq 0)$.
  Combined with the above characterization of $\L$-rings, this is precisely an integral domain in the sense of~\cite[Exercise II.2.7]{mrr:constr-alg}.
\end{itemize}

\begin{defn}
  Let $J$ be a proper ideal of the $\L$-ring $R$.
  \begin{itemize}
  \item $J$ is \textbf{\amax} if $\types_{x:R} (x\lin J) \aor \ex y. (1-xy \lin J)$.
  \item $J$ is \textbf{\mmax} if $\types_{x:R} (x\lin J) \mor \ex y. (1-xy \lin J)$.
  \item $R$ is a \textbf{\afield} if $(0)$ is proper and \amax.
  \item $R$ is a \textbf{\mfield} if $(0)$ is proper and \mmax.
  \end{itemize}
\end{defn}

We write $\inv(x) \defeq \ex y. (x y \leq 1)$ for ``$x$ is invertible''; this is the second disjunct in (either kind of) maximality for $(0)$.
The quotient $R/J$ is a \afield or \mfield if and only if $J$ is \amax or \mmax, respectively.
In the antithesis translation:
\begin{itemize}
\item An affirmative $\L$-ring is a \afield just when its corresponding $\I$-ring satisfies $\neg (0\ieq 1)$ and $(x\ieq 0) \lor \inv(x)$.
  These are called \textbf{discrete fields} (since they necessarily have decidable equality) or \textbf{geometric fields}~\cite[axiom F1]{johnstone:fields}.
\item A general $\L$-ring is a \mfield just when its corresponding $\I$-ring with inequality satisfies $0\ineq 1$ and $(x\ineq 0) \to \inv(x)$.
  This is precisely a \textbf{field} as in~\cite{mrr:constr-alg} with $\ineq$ irreflexive (in~\cite{mrr:constr-alg} the zero ring is a ``field'' with $0\ineq 0$).
\item A \mfield has \aandish refutative equality just when its $\I$-ring has a tight apartness; these are the \textbf{Heyting fields} of~\cite{mrr:constr-alg} and the \textbf{fields} of~\cite{tvd:constructivism-ii}.
\item \Aandish refutative \mmax $\L$-ideals are the \textbf{minimal anti-ideals} of~\cite{tvd:constructivism-ii}.
\item Finally, the \emph{affirmative} $\L$-rings that are \mfield{}s are the $\I$-rings satisfying $\neg (1\ieq 0)$ and $\neg(x\ieq 0) \to \inv(x)$, which is~\cite[axiom F2]{johnstone:fields}.
\end{itemize}

\begin{rmk}
  The name ``geometric field'' arises because such fields are the models of a geometric theory.
  However, antithesis translations of \mfield{s} are also a geometric theory if we include the inequality $\ineq$ as part of the theory.
  The apartness axiom for $\ineq$ is also geometric; only the tightness axiom $\neg(x\ineq y) \types (x\ieq y)$ fails to be so.

  In fact, writing a classical definition in affine logic and passing across the antithesis translation often (though not always) produces a geometric theory.
  It is a sort of refinement of the ``Morleyization'' (see e.g.~\cite[D1.5.13]{ptj:elephant2}).
\end{rmk}

\section{Order}
\label{sec:posets}

When equality is a defined relation, we can either introduce order and topology as structures on a type which induce an equality, or as structures on a set that might determine the equality by a ``separation'' axiom.
We prefer the former.

\begin{defn}
  A \textbf{preorder} on an $\L$-type $A$ is a predicate $\lle$ on $A\times A$ with
  \begin{mathpar}
    \types_{x\ocin A} (x\lle x)\and
    (x\lle y) \mand (y\lle z) \types_{x,y,z\ocin A} (x\lle z).
  \end{mathpar}
  \begin{itemize}
  \item A preorder is \textbf{\aandish} if $(x\lle y) \aand (y\lle z) \types_{x,y,z\ocin A} (x\lle z)$.
  \item A \textbf{\mtotal order} is a preorder such that $\types_{x,y\ocin A} (x\lle y) \aor (y\lle x)$.
  \item A \textbf{\atotal order} is a preorder such that $\types_{x,y\ocin A} (x\lle y) \mor (y\lle x)$.
  \end{itemize}
\end{defn}

If $A$ has a preorder, then $(x\leq y) \defeq (x\lle y) \aand (y\lle x)$ makes $A$ into a set, and $\lle$ is then a relation defined on $A\ltens A$.
The sets-with-preorder we obtain in this way are exactly the \textbf{partial orders}: sets with a preorder such that $\lle$ is a relation on $A\ltens A$ and is \textbf{$\aand$-antisymmetric}, i.e.\ $(x\lle y) \aand (y\lle x) \types_{x,y\ocin A} (x\leq y)$.

\begin{eg}
  The equality on $\Omega$ from \cref{eg:lin-set-omega} is induced in this way from the natural preorder $(P\lle Q) \defeq(P \imp Q)$.
\end{eg}

In the antithesis translation, an $\L$-partial-order contains two relations $\le$ and $\nle$; but
it is often more suggestive 
to write $x<y$ instead of $y\nle x$.

\begin{thm}\label{thm:preord}
  In the antithesis translation, a partial order on an $\L$-set $A$ consists of two $\I$-relations $\le$ and $<$ such that
\allowdisplaybreaks
  \begin{alignat*}{2}
    &\types_{x: A} &\;&(x\le x)\\
    (x\le y)\land (y\le z) &\types_{x,y,z: A} &\;&(x\le z)\\
    (x\le y) \land (y\le x) &\types_{x,y: A}&\;& (x\ieq y)\\
    (x<y)\land  (y\le z) &\types_{x,y,z: A}&\;& (x<z)\\
    (x\le y) \land (y<z)  &\types_{x,y,z: A}&\;& (x<z)\\
    (x<y) &\types_{x,y: A}&\;& (x\ineq y)\\
    (x\ineq y) &\types_{x,y: A}&\;& ((x<y) \lor (y<x)).
  \end{alignat*}
  That is, $\le$ is an $\I$-partial-order, $<$ is a ``bimodule'' over it, and for $x,y: A$ we have $(x\ineq y) \logeq ((x<y) \lor (y<x))$.
  Moreover, the $\L$-partial-order is\dots
  \begin{itemize}
  \item \dots \aandish if and only if $<$ is cotransitive: $(x<z) \types (x<y) \lor(y<z)$.
  \item \dots \mtotal if and only if $(x<y)\to (x\le y)$ (hence $<$ is transitive).
  \item \dots \atotal if and only if
    $\le$ is total, $(x\le y)\lor (y\le x)$.\qed
  \end{itemize}
\end{thm}

Such ``order pairs'' appear often in constructive mathematics, but the only abstract such definition I know of was in the Lean 2 proof assistant.\footnote{\url{https://github.com/leanprover/lean2/blob/master/library/algebra/order.lean\#L102}; it was removed in Lean 3.
I am indebted to Floris van Doorn for pointing this out.}
Often either $\le$ or $<$ is the other's negation (i.e.\ $\lle$ is affirmative or refutative),
but not always:

\begin{eg}
  Conway's \emph{surreal numbers}~\cite{conway:onag} are defined in classical logic by:
  \begin{itemize}\footnotesize
  \item[-] If $L,R$ are any two sets of numbers, and no member of $L$ is $\ge$ any member of $R$, then there is a number $\{ L | R \}$.
    All numbers are constructed in this way.
  \item[-] $x\ge y$ iff (no $x^R\le y$ and $x\le$ no $y^L$).
  \end{itemize}
  (For $x = \{L|R\}$, $x^L$ and $x^R$ denote typical members of $L$ or $R$ respectively.)
  Leaving aside the problematic inductive nature of this definition, we can write it affinely as
  \begin{align*}
    (y \lle x) &\defeq \ntpp{\ex x^R. (x^R \lle y)} \aand \ntpp{\ex y^L. (x\lle y^L)}\\
    &\logeq \fa x^R. (y \llt x^R ) \;\aand\; \fa y^L. (y^L \llt x)\\
    \intertext{where $\llt$ is its negation}
    (x\llt y) &\defeq \ntp{y\lle x}
    \logeq \ex x^R. (x^R \lle y) \;\aor\; \ex y^L. (x \lle y^L).
  \end{align*}
  In the antithesis translation, this yields a simultaneous inductive definition of $\le$ and $<$, neither of which is the Heyting negation of the other; see~\cite{forsbergfinite} and~\cite[\S11.6]{hottbook}.
  Omitting $R$ yields the \emph{plump ordinals} (see~\cite{taylor:ordinals} and~\cite[Ex.~11.17]{hottbook}).
\end{eg}



Recall that classically, if $\le$ is a total order we can define $x<y$ by $\neg (y\le x)$ or by $(x\le y)\land (x\neq y)$, and recover $x\le y$ as $\neg (y< x)$ or as $(x=y)\lor (x<y)$.
For an ``order pair'' as in \cref{thm:preord} that is \mtotal, the former holds, but the latter generally fails.
However, in affine logic we can say:

\begin{thm}\label{thm:lessthan-or-equal}
  Let $\lle$ be a refutative \mtotal order, and write $(x\llt y) \defeq \ntp{y \lle x}$.
  Then we have
  \begin{align}
    (x\llt y) &\logeq (x\lle y) \mand (x\lneq y)\label{eq:ltoreq1}\\
    (x\lle y) &\logeq (x\leq y) \mor (x\llt y).\label{eq:ltoreq2}
  \end{align}
\end{thm}
\begin{proof}
  For any partial order we have
  \begin{align*}
    (x\lle y)\mand (x\lneq y)
    &\logeq (x\lle y) \mand \ntp{(x\lle y) \aand (y\lle x)}\\
    &\logeq (x\lle y) \mand \big(\ntp{x\lle y} \aor \ntp{y\lle x}\big)\\
    &\logeq \big((x\lle y) \mand \ntp{x\lle y}\big) \aor \big((x\lle y) \mand\ntp{y\lle x}\big)\\
    &\logeq \bot \aor \big((x\lle y) \mand\ntp{y\lle x}\big)\\
    &\logeq (x\lle y) \mand (x\llt y).
  \end{align*}
  This certainly implies $x\llt y$.
  Conversely, \mtotality of $\lle$ means $(x\llt y) \types (x\lle y)$, while refutativity implies $(x\llt y) \types ((x\llt y)\mand (x\llt y))$, so that $(x\llt y)$ implies $(x\lle y) \mand (x\llt y)$.
  This gives~\eqref{eq:ltoreq1}, while~\eqref{eq:ltoreq2} is simply its De Morgan dual.
\end{proof}

Thus, while the constructive $\le$ does not mean ``less than or equal to'', at least in some cases (such as $\dR$; see \cref{sec:reals}) it does mean ``less than \emph{par} equal to'' (or perhaps, as suggested in \cref{rmk:words}, ``less than unless equal to'' or ``less than or else equal to''.)

\begin{rmk}
  In classical and intuitionistic mathematics, preorders can be identified with \emph{thin categories}: categories in which there is at most one arrow with any given domain and codomain.
  The situation is a bit more subtle in affine logic, and depends on choosing a correct definition of ``category''.
  Since in general we do not compare objects of a category for ``equality'', the objects of a category should form only a type rather than a set.
  Similarly, since we only compare arrows for equality if they are known to have the same domain and codomain, rather than a single $\L$-set of arrows we should have a collection of such sets $\hom_{\sC}(x,y)$ indexed by pairs of objects $x,y$.
  (This requires our theory to have dependent types.)

  Now an $\L$-set $A$ in which ``all elements are equal'' carries no more information than the proposition $\ex x^A.\top$, which is \emph{affirmative}.
  Thus, a ``thin category'' consists of a type together with an affirmative binary relation $\ex x^{\hom_\sC(x,y)}.\top$ that is transitive and reflexive, hence coincides with an \emph{affirmative} preorder.
  In general, therefore, whenever preorders are being treated like categories (for instance, when they are equipped with Grothendieck topologies to define sheaf toposes), they should be assumed affirmative.
\end{rmk}

\section{Real analysis}
\label{sec:reals}

Recall from \cref{rmk:nat-strong} that the natural numbers type $\dN$ is a strong set.
In fact it is a strong \atotal order, with order relation defined recursively:
\begin{mathpar}
  (0\lle n) \defeq \top\and
  (n+1\lle 0) \defeq \bot\and
  (n+1 \lle m+1) \defeq (n\lle m).
\end{mathpar}
The \textbf{integers} $\dZ$ are the type $\dN\times\dN$ with
\( ((a,b) \lle (c,d)) \defeq (a+d \lle b+c), \)
and the \textbf{rational numbers} $\dQ$ are $\dZ\times \dN$ with
\( ((x,y) \lle (u,v)) \defeq (x\cdot (v+1) \lle u\cdot (y+1)). \)
These \atotal orders are affirmative, refutative, \aandish, and decidable, as are the induced equalities $(x\leq y) \defeq ((x\lle y) \aand (y\lle x))$.
In the antithesis translation, they yield the usual posets of numbers.

We define addition and multiplication by recursion on $\dN$, and then by the usual formulas on $\dZ$ and $\dQ$, making $\dZ$ a \aint \aandish ring and $\dQ$ a \aandish \afield.

\begin{defn}
  The \textbf{Cauchy real numbers} are the partially ordered $\L$-set
  \begin{align*}
    \dR_c &\defeq \ocsetof{ x:\dQ^\dN | \fa n m. |x_n-x_m| \lle \frac1{n+1}+\frac1{m+1}}\\
    (x\lle y) &\defeq \fa n. \left(x_n \lle y_n + \textstyle\frac{2}{n+1}\right)\\
    (x\leq y) &\defeq ((x\lle y) \aand (y\lle x)) \logeq \fa n. |x_n-y_n| \lle \textstyle\frac{2}{n+1}.
  \end{align*}
\end{defn}

The set $\dR_c$ is a \aandish \mtotal order and a \aandish ring that is a \mfield.

\begin{thm}
  In the antithesis translation, the $\L$-set $\dR_c$ is the usual such $\I$-set with its usual linear order and induced equality and apartness.\qed
\end{thm}

The Dedekind real numbers are a little more surprising.
We first note that, just as in classical logic (but not intuitionistic logic), the notion of ``one-sided cut'' in affine logic doesn't depend on the side, or whether the cuts are open or closed.

\begin{defn}
  Let $L \lsub \dQ$.
  \begin{itemize}
  \item $L$ is a \textbf{lower set} if $\fa r s. (((s\lin L) \mand (r\llt s)) \imp (r\lin L))$.
  \item $L$ is \textbf{upwards-open} if $\fa r. ((r\lin L) \imp \ex s. ((r\llt s) \mand (s\lin L)))$.
  \item $L$ is \textbf{upwards-closed} if $\fa s. ((\fa r. ((r \llt s) \imp (r \lin L))) \imp (s\lin L))$.
  \end{itemize}
  Dually, we have \textbf{upper sets}, \textbf{downwards-open}, and \textbf{downwards-closed}.
\end{defn}

\begin{thm}\label{thm:cuts}
  The following $\L$-sets are isomorphic:
  \begin{align*}
    &\ocsetof{\opl \lin \P\dQ | \opl \text{ is an upwards-open lower set}}\\
    &\ocsetof{\cll \lin \P\dQ | \cll \text{ is an upwards-closed lower set}}\\
    &\ocsetof{\opu \lin \P\dQ | \opu \text{ is a downwards-open upper set}}\\
    &\ocsetof{\clu \lin \P\dQ | \clu \text{ is a downwards-closed upper set}}.
  \end{align*}
\end{thm}
\begin{proof}
  The isomorphisms are
  \begin{alignat*}{2}
    \cll &\defeq \lsetof{s\lin\dQ | \fa r. ((r \llt s) \imp (r \lin \opl))}
    &\qquad \clu &\defeq \cm{\opl}\\
    \opl &\defeq \lsetof{r\lin \dQ | \ex s. ((r\llt s) \mand (s\lin \cll))}
    &\qquad \opu &\defeq \cm{\cll}\\
    \clu &\defeq \lsetof{r\lin\dQ | \fa s. ((r \llt s) \imp (s \lin \opu))}
    &\qquad \cll &\defeq \cm{\opu}\\
    \opu &\defeq \lsetof{s\lin \dQ | \ex r. ((r\llt s) \mand (r\lin \clu))}
    &\qquad \opl &\defeq \cm{\clu}\qedhere
  \end{alignat*}
\end{proof}

\begin{defn}\label{defn:cuts}
  We write $\cC$ for any of the sets in \cref{thm:cuts}, and we call its elements \textbf{cuts}.
\end{defn}

\noindent
We give $\cC$ the partial order induced from containment of \emph{lower} sets.
Thus, if we write $x_\opl,x_\cll,x_\opu,x_\clu$ for the four representations of $x\lin\cC$, we have
\[
\begin{array}{ccccccccc}
  (x\lle y)&\logeq&(x_\opl \lsub y_\opl)&\logeq&(x_\cll \lsub y_\cll)&\logeq&(y_\opu \lsub x_\opu)&\logeq&(y_\clu \lsub x_\clu)\\
  (x\llt y)&\logeq&(y_\opl \lnsub x_\opl)&\logeq&(y_\cll \lnsub x_\cll)&\logeq&(x_\opu \lnsub y_\opu)&\logeq&(x_\clu \lnsub y_\clu).
\end{array}
\]
Using \atotal{ity} of the order on $\dQ$, we can show that this order on $\cC$ is \mtotal.
If we identify $r\lin\dQ$ with the cut $r_\opl \defeq \lsetof{q\lin\dQ | q \llt r}$, then $\dQ$ is fully order-embedded in $\cC$, and moreover for any $x\lin\cC$ and $r\lin\dQ$ we have
\begin{alignat*}{2}
  (r \llt x) &\logeq (r\lin x_\opl)
  &\qquad (x \llt r) &\logeq (r \lin x_\opu)\\
  (r \lle x) &\logeq (r\lin x_\cll)
  &\qquad (x \lle r) &\logeq (r \lin x_\clu).
\end{alignat*}
Thus, we can define a cut $x$ by specifing any one of the relations $(-\llt x)$, $(- \lle x)$, $(x \llt -)$, and $(x\lle -)$ on $\dQ$ which has the appropriate property.
This is usually more congenial than working explicitly with upper or lower subsets of $\dQ$.

In intuitionistic logic, it is common to work with two-sided cuts instead.
But because an $\L$-subset is a \emph{complemented} $\I$-subset, in the antithesis translation our one-sided $\L$-cuts become two-sided $\I$-cuts.

\begin{thm}\label{thm:std-ivl}
  In the antithesis translation, $\cC$ corresponds to the set of pairs $(L,U)$ of $\I$-subsets of $\dQ$ such that $L$ is an upwards-open lower set, $U$ is a downwards-open upper set, and $L<U$.
  Its induced order is
  \begin{align*}
    ((L_1,U_1) \le (L_2,U_2)) &\logeq ((L_1\subseteq L_2) \land (U_2 \subseteq U_1)).\\
    ((L_1,U_1) < (L_2,U_2)) &\logeq \exists r. ((r\in L_2) \land (r\in U_1)).
  \end{align*}
\end{thm}
\begin{proof}
  By \cref{thm:rel}, an element of $\P\dQ$ is a disjoint pair $(L,\cancel L)$ of subsets of $\dQ$.
  To say that it is an $\L$-lower-set means that $L$ is an $\I$-lower-set and $\cancel L$ is an $\I$-upper-set.
  Given this, disjointness is equivalent to $L < \cancel L$.
  And to say that it is $\L$-upwards-open means that $L$ is $\I$-upwards-open and $\cancel L$ is $\I$-downwards-closed.
  Finally, the bijection between open and closed \emph{upper} cuts (or, dually, lower ones) is also true intuitionistically:
  \begin{equation*}
    U \defeq \setof{s | \exists r. ((r<s) \land (r\in \cancel L)) }\hspace{1cm}
    \cancel L \defeq \setof{r | \forall s. ((r<s) \to (s\in U)) }.\qedhere
  \end{equation*}
\end{proof}

The $\I$-set of pairs $(L,U)$ in \cref{thm:std-ivl} is also called the set of \textbf{(rational) cuts}~\cite{richman:cuts}, or sometimes the \textbf{interval domain}.
It is distinct from $\dR$ even classically, containing additionally all closed intervals $[a,b]$ for $-\infty\le a\le b\le \infty$.

\begin{defn}
  The \textbf{Dedekind real numbers} $\dR_d$ are the $\L$-set of $x\ocin\cC$ with
  \begin{itemize}
  \item \textbf{boundedness}: $\ex q^\dQ.(q\llt x) \mand \ex q^\dQ. (x \llt q)$
  \item \textbf{\alocated{}ness}: $\fa r^\dQ s^\dQ. ((r\llt s) \imp ((r\llt x) \aor (x \llt s)))$.
  \end{itemize}
\end{defn}

\emph{All} cuts are ``\mlocated{}'', $(r\llt s) \imp ((r\llt x) \mor (x \llt s))$, by $\mor$-excluded-middle.
In the antithesis translation, $\dR_d$ is the usual set of Dedekind reals. 

\begin{rmk}
  As we did for $\dN$ in \cref{rmk:nat-strong}, we can prove entirely in affine logic that the real numbers form a \aandish set.
  Suppose $(x\lle y)\aand (y\lle z)$; to show $x\lle z$ we assume $r:\dQ$ with $r\llt x$ and must prove $r\llt z$.
  Now there is a $s:\dQ$ with $r\llt s\llt x$, and since $y$ is \alocated we have $(r\llt y)\aor (y\llt s)$.
  Doing a case split on this, if $r\llt y$ we use $y\lle z$ to conclude $r\llt z$, while if $y\llt s$ we use $x\lle y$ to conclude $x\llt s$, a contradiction.
\end{rmk}


Now, there are at least two natural ways to define addition on $\cC$:
\begin{align*}
  (x \umplus y \llt q) &\defeq \ex r^\dQ s^\dQ. ((q\leq r+s)\mand (x \llt r) \mand (y \llt s))\\
  (q \llt x \lmplus y) &\defeq \ex u^\dQ v^\dQ. ((q\leq u+v)\mand (u\llt x) \mand (v\llt y)).
\end{align*}

If $x,y\ocin\dR_d$, one can prove that $x\lmplus y \leq x\umplus y$, and in the antithesis translation they give the usual addition on Dedekind reals:
\begin{align}
  (x+y < q) &\defeq \exists rs. ((q=r+s) \land (x<r) \land (y<s))\label{eq:uiplus}\\
  (q < x+y) &\defeq \exists rs. ((q=r+s) \land (r<x) \land (s<y))\label{eq:liplus}.
\end{align}
However,
for \emph{cuts}, $\umplus$ and $\lmplus$ are distinct.
In the antithesis translation, with $+$ for $\I$-cuts defined using~\eqref{eq:uiplus}--\eqref{eq:liplus}, we have $(x\umplus y < q) \logeq (x+y < q)$, but $q < x\umplus y$ is weaker than $q < x+y$.
One place where this matters is in defining metric spaces.

\begin{defn}\label{defn:cut-metric}
  A \textbf{cut-metric} on an $\L$-type $X$ is an operation $d:\cC^{X\times X}$ with
  \begin{alignat}{2}
    &\types_{x,y\ocin X} &\;& 0 \lle d(x,y) \notag\\
    &\types_{x\ocin X} &\;& d(x,x)\leq 0 \notag\\
    &\types_{x,y,z\ocin X} &\;& d(x,z) \lle d(x,y) \umplus d(y,z).\notag
  \end{alignat}
\end{defn}

For any cut-metric, $(x\lle y) \defeq (d(x,y) \leq 0)$ defines a preorder.
If $d$ is symmetric, this is already an equality making $X$ a set; otherwise we can symmetrize it as in \cref{sec:posets}.
(We can also symmetrize $d$ directly with $d'(x,y) \defeq \sup(d(x,y),d(y,x))$.)
If $X$ is already a set and $d$ a function, the usual metric separation condition $(d(x,y)\leq 0) \types (x\leq y)$ makes its equality coincide with that obtained in this way.

In particular, if $d(x,y)\ocin \dR_d$ for all $x,y$, then the antithesis translation of $X$ is an $\I$-quasi-metric space, and an $\I$-metric space if we impose symmetry.

Now suppose $X$ is a cut-metric space and we have $a\lin X$ and $B\lsub X$.
As observed intuitionistically in~\cite{richman:cuts}, $\cC$ is a complete lattice (which $\dR_d$ is not, constructively\footnote{The ``strongly monotonic cuts'' or ``MacNeille reals'' are also a complete lattice, but their meets and joins involve double-negation, making them less useful than those of $\cC$; see~\cite[\S3]{richman:cuts}.}); thus we can define the \textbf{distance from $a$ to $B$} as an infimum:
\begin{align*}
  d(a,B) &\defeq \textstyle\inf_{b\lin B} d(a,b).
\end{align*}
Rather than defining infima in $\L$-posets in general, we simply make this explicit:
\begin{align*}
  (d(a,B) \llt q) &\defeq \ex b^X. ((b\lin B) \mand (d(a,b) \llt q))\\
  (q\lle d(a,B)) &\logeq \fa b^X. ((b\lin B) \imp (q \lle d(a,b)))\\
  (q\llt d(a,B)) &\logeq \ex r. ((q\llt r) \mand \fa b^X. ((b\lin B) \imp (r \lle d(a,b)))).
\end{align*}
Even if each $d(a,b)$ is a Dedekind real, $d(a,B)$ may not be.
But the observation of Richman~\cite{richman:cuts} is that if we treat $d(a,B)$ as a cut, then its inequality relations to rational (hence also real) numbers are exactly what we would expect of such a ``distance''.
In the antithesis translation, these become:
\begin{align*}
  (d(a,B) < q) &\logeq \exists b^X. ((b\in B) \land (d(a,b) < q))\\
  (q \le d(a,B)) &\logeq \forall b^X. (((b\in B) \to (q \le d(a,b))) \land ((d(a,b) < q) \to (b\in \cancel B))).
\end{align*}
If $B$ is affirmative, $q \le d(a,B)$ becomes Richman's $\forall b^B. (q \le d(a,b))$.
We also have:
\begin{align*}
  \MoveEqLeft (d(a,B) \lle d(a,B'))\\
  &\logeq \fa q. ((d(a,B') \llt q) \imp (d(a,B) \llt q))\\
  &\logeq \fa q. ((\ex \smash{b'}^X. ((b'\lin B') \mand (d(a,b') \llt q))) \imp \ex b^X. ((b\lin B) \mand (d(a,b) \llt q)))\\
  &\logeq \fa q. \fa \smash{b'}^X. (((b'\lin B') \mand (d(a,b') \llt q)) \imp \ex b^X. ((b\lin B) \mand (d(a,b) \llt q))).\\
  \intertext{At least in the Dedekind-real case, we can then write $\ep \defeq q - d(a,b')$ to get}
  &\mathrel{\phantom{\logeq}} \fa \ep. \fa \smash{b'}^X. ((b'\lin B') \imp \ex b^X. ((b\lin B) \mand (d(a,b) \llt d(a,b') + \ep)))
\end{align*}
whose antithesis translation, when $B,B'$ are affirmative, reduces to Richman's: 
\begin{equation*}
  (d(a,B) \le d(a,B')) \logeq
  \forall \ep. \forall \smash{b'}^X. ((b'\in B') \to \exists b^X. ((b\in B) \land (d(a,b) < d(a,b') + \ep))).
\end{equation*}
Still following~\cite{richman:cuts}, we can define the \textbf{(directed) Hausdorff distance} between two subsets $A,B\lsub X$ as:
\begin{align*}
  d(A,B) &\defeq \textstyle \sup_{a\lin A} d(a,B) = \sup_{a\lin A} \inf_{b\lin B}d(a,b) \\
  (d(A,B) \llt q) &\logeq \ex q'. ((q' \llt q) \mand \fa a^X. ((a \lin A) \imp \ex b^X. ((b\lin B) \mand (d(a,b) \llt q'))))
\end{align*}
However, unlike Richman, we can show:

\begin{thm}\label{thm:hausdorff}
  The Hausdorff distance is a cut-metric on $\P X$.
\end{thm}
\begin{proof}
  The proof of the triangle inequality is essentially the same as that of its ``upper portion'' in~\cite[\S6]{richman:cuts}.
  We must show if $d(A,B)\umplus d(B,C) \llt q$ then $d(A,C)\llt q$.
  By definition of $\umplus$, we have $\ex r s. ((q\leq r+s)\mand (d(A,B) \llt r) \mand (d(B,C) \llt s))$.
  Now $d(A,B) \llt r$ and $d(B,C) \llt s$ yield $r'\llt r$ and $s'\llt s$ such that
  \begin{align*}
    & \fa a^X. ((a \lin A) \imp \ex b^X. ((b\lin B) \mand (d(a,b) \llt r')))\\
    & \fa b^X. ((b \lin B) \imp \ex c^X. ((c\lin C) \mand (d(b,c) \llt s'))).
  \end{align*}
  Thus, for any $a\lin A$ we get $b\lin B$ with $d(a,b)\llt r'$, then from $b$ we get $c\lin C$ with $d(b,c)\llt s'$.
  Hence $d(a,c) \lle d(a,b)\umplus d(b,c) \llt r'+s' \lle q$, so that $d(A,C)\llt q$.
\end{proof}

In~\cite[\S6]{richman:cuts} Richman notes that the cut-valued Hausdorff distance fails the triangle inequality if addition of cuts is defined by~\eqref{eq:uiplus}--\eqref{eq:liplus}.
He concludes that one should forget the ``lower cut'' part of the Hausdorff distance.
We conclude instead that the relevant addition of cuts is $\umplus$, not~\eqref{eq:uiplus}--\eqref{eq:liplus}.
Indeed,~\eqref{eq:uiplus}--\eqref{eq:liplus} are suspect right away, as it is not even clear whether they can be obtained \emph{simultaneously} as the antithesis translation of \emph{any} single definition of addition for $\L$-cuts.

The other example in~\cite[\S6]{richman:cuts} where cuts seem to have problems can also be resolved with affine logic.
Suppose (intuitionistically) $A$ is an abelian group and $p$ a prime with $\bigcap_n p^n A = 0$, i.e.\ $(\forall n. \exists c. (a \ieq p^n c)) \types_{a\in A} (a\ieq 0)$.
Define
\[|a| \defeq \inf \setof{p^{-n} | a \in p^n A} =
\inf \setof{p^{-n} | \exists c. (a \ieq p^n c)}. \]
In classical mathematics, this defines an ``ultranorm'', i.e.\
\(
  |a+b| \le \sup(|a|,|b|).
\)
Intuitionistically, if we interpret $|a|$ as a cut and $\sup$ as the binary supremum (the union of lower parts and intersection of upper parts), then the upper part of $|a+b| \le \sup(|a|,|b|)$ holds but the lower part can fail.

Our solution is to replace this ``additive'' binary supremum with a multiplicative one.
Returning to affine logic, for cuts $x,y$ we define
\begin{align*}
  (\msup (x,y) \llt q) &\defeq (x \llt q) \mand (y \llt q).
\end{align*}
Now let $A$ be an abelian $\L$-group and $(\fa n. \ex c^A. (a \leq p^n c)) \types_{a\in A} (a\leq 0)$, and define
\[|a| \defeq \inf \lsetof{x | \ex n. ((x \leq p^{-n}) \mand \ex c. (a \leq p^n c))}. \]
I claim that then we have
\(
  |a+b| \lle \msup(|a|,|b|).\notag
\)
This is again just like the ``upper part'' proof from~\cite{richman:cuts}: if $|a|\lle p^{-n}$ and $|b|\lle p^{-n}$, then $a \leq p^n c$ and $b \leq p^n d$ for some $c,d$, so that $a+b \leq p^n(c+d)$ and hence $|a+b|\lle p^{-n}$.
So in both cases, it is not that cuts are inadequate, but that the operations on cuts sometimes need to use multiplicative connectives rather than additive ones.

\begin{rmk}
  Another problematic area of analysis for constructive mathematics is the theory of measure spaces.
  Already in~\cite{bb:constr-analysis} complemented subsets were used as the domain for a constructive measure, and~\cite{nlab:cheng-spaces} formulates an abstract notion of measurable space based on a Chu construction like ours.
  This suggests that affine logic would also be a natural context for constructive measure theory.
\end{rmk}

\begin{rmk}
  We end this section with an example where proof-relevance matters.
  An $\I$-sequence of real (or rational) numbers is \textbf{Cauchy} if
  \begin{equation}
    \forall \ep. \exists k. \forall nm. (n>k\land m>k \to |x_n-x_m|\le \ep),\label{eq:cauchy}
  \end{equation}
  and \textbf{diverges}~\cite[\S2.3]{bb:constr-analysis} if
  \begin{equation}
    \exists \ep. \forall k. \exists nm. (n>k \land m>k \land |x_n-x_m|>\ep).\label{eq:div}
  \end{equation}
  These are formal De Morgan duals, so if we define Cauchy-ness of an $\L$-sequence by
  \begin{equation}
    \fa \ep. \ex k. \fa nm. (n>k \mand m>k \imp |x_n-x_m|\lle\ep),\label{eq:lcauchy}
  \end{equation}
  then its linear negation is divergence.

  However, in the absence of countable choice, it is often better to consider a Cauchy sequence as coming with a \emph{function} $K_\ep$, and Bishop presumably understands a divergent sequence to come with functions $N_k,M_k$.
  But if we write out the assertions of such functions by hand, the corresponding formulas
  \begin{gather}
    \exists K. \forall \ep. \forall nm. (n>K_\ep \land m>K_\ep \to |x_n-x_m|\le\ep)\label{eq:skcauchy} \\
    \exists \ep. \exists NM. \forall k. (N_k>k \land M_k > k \land |x_{N_k}-x_{M_k}|>\ep).\label{eq:skdiv}
  \end{gather}
  are no longer De Morgan duals.
  G\"{o}del's ``Dialectica'' interpretation%
~\cite{godel:dialectica,depaiva:dialectica-like,paiva:dialectica-chu,hofstra:dialectica} automatically does this sort of ``Skolemization'', so that~\eqref{eq:cauchy} and~\eqref{eq:div} would be interpreted as~\eqref{eq:skcauchy} and~\eqref{eq:skdiv} respectively; but
  this doesn't solve the problem that the two pairs are not each other's negations.

  Instead, we can write~\eqref{eq:cauchy} and~\eqref{eq:div} using the propositions-as-types interpretation into dependent type theory.
  This gives
  \begin{gather}
    \textstyle \prod_\ep \sum_k \prod_{n,m} (n>k\land m>k \to |x_n-x_m|\le \ep) \label{eq:patcauchy}\\
    \textstyle\sum_\ep \prod_k \sum_{n,m} (n>k \land m>k \land |x_n-x_m|>\ep)\label{eq:patdiv}
  \end{gather}
  which include the Skolem functions automatically, due to the ``type-theoretic axiom of choice'' $\prod_{x:A} \sum_{y:B} C(x,y) \cong \sum_{f:A\to B} \prod_{x:A} C(x,f(x))$.
  Moreover,~\eqref{eq:patcauchy} and~\eqref{eq:patdiv} are still De Morgan duals with respect to $\Sigma$ and $\Pi$.
  Therefore, we can obtain them from the antithesis translation of~\eqref{eq:lcauchy} applied to the propositions-as-types hyperdoctrine mentioned in \cref{eg:flim-hd}.
\end{rmk}

\section{Topology}
\label{sec:topology}

Finally, we consider point-set topologies.
There are many classically-equivalent ways to define a topology; first we consider neighborhood relations.

Note that the preorder $(U\lsub V) \defeq \fa x^A. ((x\lin U) \imp (x\lin V))$ on $\Omega^A$ makes sense even if $A$ is only a type, making $\Omega^A$ into a set.

\begin{defn}\label{defn:topology}
  A \textbf{topology} on a type $A$ is a predicate $\lint$ on $A\times \Omega^A$ with
  \[
  \begin{array}{rlll}
    (x\lint U) &\types & (x\lin U) & \text{(reflexivity)}\\
    (x\lint U) \mand (U\lsub V) & \types& (x\lint V) & \text{(isotony)}\\
    &\types & (x\lint A) & \text{(nullary additivity)}\\
    (x\lint U) \mand (x\lint V) &\types & (x\lint U \lcap V) &\text{(binary additivity)}\\
    (x\lint U) & \types& (x\lint \lsetof{ y | y \lint U}) & \text{(transitivity)}
  \end{array}
  \]
\end{defn}

Isotony implies each $(x\lint-)$ is a relation on the set $\Omega^A$.
We define a preorder on $A$ by $(x\lle y) \defeq \fa U. ((y\lint U) \imp (x\lint U))$, making $A$ into a set as well, such that $\lint$ is a relation on $A \ltens \Omega^A$.
Moreover, an arbitrary predicate $U:\Omega^A$ is contained in a smallest relation $\hat{U} \defeq \lsetof{x\in A | \ex y^A. ((x\leq y) \mand U(y))}$, and by definition of equality on $A$ we have $(x\lint U) \liff (x\lint \hat{U})$.
Thus, $\lint$ is determined by its behavior on subsets, so we may consider it to be a relation on $A\ltens \P A$.
If we instead assume $A$ is given as a set and $\lint$ as a relation on $A\ltens \P A$, then to ensure that the equality coincides with the one constructed above we must impose the $T_0$ axiom
\[ \fa U. ((x\lint U) \liff (y\lint U)) \types (x\leq y). \]

Now a relation on $A\ltens \P A$ is equivalently a function $\intr : \P A \to \P A$, and \cref{defn:topology} translates into an affine version of an ``interior operator'':
\begin{enumerate}[label=(\roman*)]
\item $\intr(U) \lsub U$
\item $(U\lsub V) \imp (\intr(U) \lsub \intr (V))$
\item $\intr(A) \leq A$
\item $\intr (U) \lmcap \intr (V) \lsub \intr(U\lcap V)$\label{item:lintr-meet}
\item $\intr(\intr(U)) \leq \intr(U)$
\end{enumerate}
plus the following form of the $T_0$ axiom:
\begin{enumerate}[label=(\roman*),start=6]
\item $(\fa U. ((x\lin \intr(U)) \liff (y\lin \intr(U)))) \imp (x\leq y)$.
\end{enumerate}
The most interesting of these axioms is~\ref{item:lintr-meet}, which is a version of the classical
\begin{equation}
  \intr(U)\cap \intr(V) \subseteq \intr(U\cap V)\label{eq:cintr-meet}
\end{equation}
that uses an \emph{additive} intersection on the right but a \emph{multiplicative} one on the left.
This may look more surprising than the (equivalent) binary additivity axiom for $\lint$ in \cref{defn:topology}, since in the latter $\mand$ is a logical connective while $\lcap$ is a set operation.
However, the following example suggests that this odd-looking mixture of additive and multiplicative intersections is exactly right:

\begin{eg}\label{eg:metric-top}
  Any cut-metric space (\cref{defn:cut-metric}) has a topology defined by:
  \[ (x \lint U) \defeq
  \ex \ep^{\dQ_{>0}}. \fa y^X. ((d(x,y) \llt \ep) \imp (y\lin U)).
  \]
  To prove binary additivity, from $(x\lint U)\mand (x\lint V)$ we can get $\ep_U$ and $\ep_V$, and then choose $\ep \defeq \min(\ep_U,\ep_V)$ to prove $x\lint (U\lcap V)$.
  Note that here we need to use \emph{both} hypotheses at once, so they must be combined with $\mand$ rather than $\aand$.
  We then have to show that given $y$ with $d(x,y) \llt \ep$ we have $y\lin U\lcap V$, i.e.\ that $(y\lin U) \aand (y\lin V)$.
  For $y\lin U$ we use $d(x,y) \llt \ep \lle \ep_U$ and the hypothesis from $x\lint U$, and dually.
  Note that here we need to use the same hypothesis $d(x,y) \llt \ep$ in proving both subgoals $y\lin U$ and $y\lin V$, so they must be combined with $\aand$ rather than $\mand$.
\end{eg}

Axiom~\ref{item:lintr-meet} is further clarified by writing it in terms of $\cl(U) \defeq \cmp{\intr (\cm U)}$:
\begin{equation}
  \label{eq:lcl-join}
  \cl(U\lcup V) \lsub \cl (U) \lmcup \cl (V)
\end{equation}
Since $\cmp{-}$ is involutive, in linear logic $\cl$ and $\intr$ contain the same data.
But intuitionistically, ``closure operators'' do not respect unions: a point may lie in the closure of $U\cup V$ without our being able to decide which of $U$ or $V$ it lies in the closure of.
Our~\eqref{eq:lcl-join} remedies this by taking one of the unions to be multiplicative.

On the other hand, classically \emph{and} intuitionistically the converse of~\eqref{eq:cintr-meet} always holds, so~\eqref{eq:cintr-meet} is equivalent to closure of the fixed points of $\intr$ (the open sets) under binary intersections.
It is harder to express~\ref{item:lintr-meet} using ``open $\L$-subsets''.

We now move on to the antithesis translation of an $\L$-topology.
This is rather complicated, since not only does $\lint$ give rise to two relations, each $\L$-subset $U$ is actually \emph{two} (disjoint) $\I$-subsets.
We start with some familiar special cases.

\begin{thm}
  Under the antithesis translation, an $\L$-topology such that
  \[
  \begin{array}{rll}
    (x\lint U)& \types& \oc(x\lint U) \mand (x \lint \oc U)
  \end{array}
  \]
  corresponds exactly to a $T_0$ $\I$-topology on a type $A$.
  If we write $x \iint U$ for the $\I$-relation ``$x$ is in the interior of $U$'', then the induced inequality on $A$ is
  \[ (x\ineq y) \defeq \exists U.((x\iint U) \land \neg(y\iint U)) \lor \exists U.(\neg(x\iint U) \land(y\iint U)). \]
\end{thm}
\begin{proof}
  The assumption implies that $\lint$ is determined by a single $\I$-relation $\iint$ between points of $A$ and $\I$-subsets of $A$.
  The axioms on $\lint$ then translate to the usual definition of an $\I$-topology in terms of a neighborhood relation.
  Finally, our definition of equality in a topological space corresponds to the $T_0$ axiom.
\end{proof}

If $U$ is an $\I$-subset of an $\I$-set $A$ with inequality $\ineq$, we write
\[(x\notin U) \defeq \forall y^A. ((y\in U) \to (x\ineq y)).\]
This is the same as saying that $x$ belongs to the inequality complement of $U$, i.e.\ $x\lnin \ochat U$ in the antithesis translation.

\begin{thm}
  Under the antithesis translation, an $\L$-topology such that\footnote{A simpler attempt at~\eqref{eq:lapart2} would be $(x\lint \wn U) \types (x\lint U)$, but that is inconsistent with reflexivity at least in the antithesis translation, since $\wn(\zero,\zero) = (\one,\zero)$.}
  \begin{alignat}{2}
    (x\lint U)& \types&\;& \oc(x\lint U)\label{eq:lapart1}\\
    (x\lint \wn U) \mand \oc(x\lin U)& \types&\;& (x \lint U)\label{eq:lapart2}
  \end{alignat}
  corresponds to a \textbf{point-set pre-apartness space} satisfying the \textbf{reverse Kolmogorov property} in the sense of~\cite[p20]{bv:apartness},\footnote{\cite{bv:apartness} writes $x\notin K$ to mean $\neg(x\in K)$; our $x\notin K$ is written there as $x\in \mathord{\sim}K$.} i.e.\ an $\I$-set with an inequality $\ineq$ and a relation $\apart$ between points and $\I$-subsets such that
  \begin{alignat}{2}
    (x \apart K)& \types &\;& (x\notin K)\label{eq:iapart2}\\
    (x\apart K) \land (L\subseteq K) &\types &\;& (x\apart L)\label{eq:iapart3a}\\
    &\types&\;& (x\apart \emptyset)\label{eq:iapart1}\\
    (x\apart K) \land (x\apart L) &\types&\;& (x\apart K\cup L)\label{eq:iapart3b}\\
    \forall x.((x\apart K) \to (x\notin L)) &\types &\;& \forall x.((x\apart K) \to (x\apart L))\label{eq:iapart4}\\
    (x\apart K) \land \neg(y\apart K) &\types &\;& (x\ineq y),\label{eq:iapartT0rev}\\
    \intertext{and which also satisfies the additional ``forwards Kolmogorov property'' that}
    (x\ineq y) &\types &\;& (x\apart\{y\}) \lor (y\apart \{x\}).\label{eq:iapartT0fwd}
  \end{alignat}
\end{thm}
\begin{proof}
  Since $U\lsub \wn U$, and $x\lint U$ is affirmative by~\eqref{eq:lapart1}, the converse of~\eqref{eq:lapart2} also holds.
  Thus $\lint$ is determined by its behavior on refutative $\L$-subsets, hence by one $\I$-relation between points and $\I$-subsets.
  We define $(x \apart K) \defeq (x \lint (\neg K,K))$.
  But note that $(\neg K,K)$ is not an $\L$-subset, and as noted above $\lint$ is determined by its behavior on $\L$-subsets; thus $x\apart K$ is also equivalent to $x \lint \wnchk(\neg K,K) = (\setof{y | y\notin K},K)$.
  In particular, reflexivity of $\lint$ implies~\eqref{eq:iapart2}.

  Statements~\eqref{eq:iapart3a}--\eqref{eq:iapart3b} are straightforward translations of $\L$-isotony and additivity.
  The direct translation of $\L$-transitivity is
  \(
    (x\apart K) \types (x \apart \setof{y | \neg (y\apart K)}),
    \)
  which is equivalent to~\eqref{eq:iapart4} and~\eqref{eq:iapartT0rev} together.
  Our definition of equality yields
  \[(x\ineq y) \logeq \exists K.((x\apart K) \land \neg (y\apart K)) \lor \exists K.(\neg(x\apart K) \land (y\apart K)) . \]
  However, if $x\apart K$ and $\neg (y\apart K)$, then $y\in \setof{z|\neg(z\apart K)}$, whence $x\apart \{y\}$; whereas conversely if $x\apart \{y\}$ then we can take $K\defeq \{y\}$.
  Thus, we have
  \[(x\ineq y) \logeq (x\apart\{y\}) \lor (y\apart \{x\}).\]
  The right-to-left implication follows from~\eqref{eq:iapart2}, while the left-to-right is~\eqref{eq:iapartT0fwd}.
\end{proof}

Thus, $\I$-topologies and apartnesses are special cases of $\L$-topologies.
But in some sense these restrictions on $\L$-topologies miss the point, because virtually no \emph{naturally defined} $\L$-topologies satisfy them!
In the antithesis translation, a general $\L$-topology consists of two relations $\iint$ and $\niint$ between points and complemented subsets (\cref{thm:rel}); and even for Dedekind-metric spaces neither is the Heyting negation of the other, and both parts of a complemented subset are used.

\begin{eg}\label{eg:ltop-met}
  Recall that in \cref{eg:metric-top} we showed that any $\L$-cut-metric space has an underlying $\L$-topology.
  In the antithesis translation, this topology becomes:
  \begin{align*}
    (x \iint (U,\cancel U)) &\defeq \exists \ep^{\dQ_{>0}}. \forall y^X. (((d(x,y)<\ep) \to (y\in U)) \land ((y\in \cancel U) \to (d(x,y)\ge \ep)))\\
    (x \niint (U,\cancel U)) &\defeq \forall \ep^{\dQ_{>0}}. \exists y^X. ((d(x,y)<\ep) \land (y\in \cancel U)).
  \end{align*}
  This is degenerate only in that the relation $\niint$ only depends on $\cancel U$, not on $U$.
  But since both conjuncts in $(x \iint (U,\cancel U))$ remain true under shrinking $\ep$, we can distribute the quantifiers and take a minimum of the two $\ep$'s to write
\begin{multline*}
(x \iint (U,\cancel U)) \logeq (\exists \ep^{\dQ_{>0}}. \forall y^X. ((d(x,y)<\ep) \to (y\in U)))\\ \land (\exists \ep^{\dQ_{>0}}. \forall y^X.((y\in \cancel U) \to (d(x,y)\ge \ep)))
\end{multline*}
where the first conjunct depends only on $U$ and the second only on $\cancel U$.
Thus, we may think of $x \iint (U,\cancel U)$ as ``$x$ is in the interior of $U$ and is apart from $\cancel U$''.
\end{eg}

\begin{figure}
  \centering
\fbox{\scalebox{0.7}{$
\begin{array}{rll}
  &\types& \neg ((x \iint (U,\cancel U)) \land (x \niint (U,\cancel U)))\\
  (x\iint (U,\cancel U)) &\types& x\in U \\ 
  x\in \cancel U &\types& (x\niint (U,\cancel U)) \\
  (x\iint (U,\cancel U)) \land (U\subseteq V) \land (\cancel V \subseteq \cancel U) &\types & (x \iint (V,\cancel V)) \\
  (x \niint (V,\cancel V)) \land (U\subseteq V) \land (\cancel V \subseteq \cancel U) &\types & (x\niint (U,\cancel U)) \\
  (x\iint (U,\cancel U)) \land (x \niint (V,\cancel V)) &\types& \exists y. (y \in U \cap \cancel V) \\
                         &\types & (x \iint (A,\emptyset)) \\
                         &\types & \neg (x \niint (A,\emptyset)) \\
  (x\iint (U,\cancel U)) \land (x \iint (V,\cancel V)) &\types& (x \iint (U\cap V, \cancel U \cup \cancel V)) \\
  (x\iint (U,\cancel U)) \land (x \niint (U\cap V, \cancel U \cup \cancel V)) &\types& (x \niint (V,\cancel V)) \\
  (x \iint (U,\cancel U)) &\types& (x \iint (\setof{y | y \iint (U,\cancel U)}, \setof{y | y \niint (U,\cancel U)})) \\
  (x \niint (\setof{y | y \iint (U,\cancel U)}, \setof{y | y \niint (U,\cancel U)})) &\types& (x \niint (U,\cancel U)) 
\end{array}
$}}
\caption{The antithesis translation of an $\L$-topology}
\label{fig:sitop}
\end{figure}

In the general case,
we can write the axioms of an $\L$-topology in terms of $\iint$ and $\niint$, as in \cref{fig:sitop}.
But they are not very familiar, because we are used to spaces that are degenerate in the manner of \cref{eg:ltop-met}: with $\niint$ depending only on $\cancel U$, and $\iint$ the conjunction of two properties depending on $U$ and $\cancel U$ respectively.
This suggests the following definition.

\begin{defn}\label{defn:unified-top}
  A \textbf{unified topology} on an $\I$-type $A$ consists of three predicates $\iint,\apart,\close$ on $A\times \Omega^A$ such that:
  \allowdisplaybreaks
  \begin{itemize}
  \item $\iint$ is a topology in the usual sense:
    \begin{alignat}{2}
      (x\iint U) &\types&\;& (x\in U) \notag\\
      (x\iint U) \land (U\subseteq V) &\types&\;& (x\iint V)\notag\\
      &\types&\;& (x\iint A) \notag\\
      (x\iint U) \land (x\iint V) &\types&\;& (x\iint U\cap V)\notag\\
      (x \iint U) &\types&\;& (x\iint \setof{y | y \iint U}) \notag
    \end{alignat}
  \item $\apart$ satisfies the following apartness axioms:
    \begin{alignat}{2}
      (x\apart K) &\types&\;& \neg (x\in K) \tag{$*$}\\
      (x\apart K) \land (L\subseteq K) &\types&\;& (x\apart L)\notag\\
      &\types&\;& (x\apart \emptyset) \notag\\
      (x\apart K) \land (x\apart L) &\types&\;& (x\apart K\cup L)\notag\\
      (x \apart K) &\types&\;& (x \apart \setof{y | y \close K})\label{eq:aptcl}
    \end{alignat}
  \item $\close$ satisfies the following ``closure space'' axioms:
    \begin{alignat}{2}
      (x\in K) &\types&\;& (x\close K)\notag\\
      (x\close K) \land (K\subseteq L) &\types&\;& (x\close L) \notag\\
      &\types&\;& \neg (x\close \emptyset) \tag{$*$}\\
      (x\close (K \cup L)) \land (x \apart K) &\types&\;& (x\close L) \label{eq:clapt}\\
      (x \close \setof{y | y \close K}) &\types&\;& (x \close K)\notag
    \end{alignat}
  \item The following compatibility condition holds:
    \begin{alignat}{2}
      (x\iint U) \land (x\close K) &\types&\;& \exists y. (y\in U\cap K). \label{eq:intcl}
    \end{alignat}
  \end{itemize}
\end{defn}
In the presence of the other axioms, either of the axioms ($*$) implies the other.
Note that transitivity for $\apart$~\eqref{eq:aptcl} involves $\close$, while binary additivity for $\close$~\eqref{eq:clapt} (in constructively sensible form derived from $\mor$) involves $\apart$.

\begin{thm}
  Given a unified topology, if we define
  \begin{mathpar}
    (x \iint (U,\cancel U)) \defeq (x \iint U) \land (x \apart \cancel U)\and
    (x \niint (U,\cancel U)) \defeq (x \close \cancel U)
  \end{mathpar}
  then we obtain an $\L$-topology (in the antithesis translation) as in \cref{fig:sitop}.\qed
\end{thm}

\noindent
Not every $\L$-topology has this form, but those coming from cut-metrics do, with
\begin{align*}
  (x \iint U) &\defeq \exists \ep^{\dQ_{>0}}. \forall y^X. ((d(x,y)<\ep) \to (y\in U))\\
  (x \apart K) &\defeq \exists \ep^{\dQ_{>0}}. \forall y^X.((y\in K) \to (d(x,y)\ge \ep))\\
  (x \close K) &\defeq \forall \ep^{\dQ_{>0}}. \exists y^X. ((d(x,y)<\ep) \land (y\in K)).
\end{align*}

\begin{eg}
  Recall the Hausdorff cut-metric on $\P X$ from \cref{thm:hausdorff}:
  \[ (d(A,B) \llt q) \logeq \ex q'. ((q' \llt q) \mand \fa a^X. ((a \lin A) \imp \ex b^X. ((b\lin B) \mand (d(a,b) \llt q')))). \]
  In the antithesis translation, if $A,B$ are affirmative, then:
  \begin{itemize}
  \item $d(A,B) < q$ means that there is a $q'<q$ such that for any point $a\in A$, there exists a point $b\in B$ with $d(a,b)<q'$.
  \item $q\le d(A,B)$ means for any $q'<q$, there is a point $a\in A$ such that every point $b\in B$ has $q' \le d(a,b)$.
  \end{itemize}
  Thus, in this case:
  \begin{itemize}
  \item $A\iint \cU$ means there is an $\ep>0$ such that $\cU$ contains all subsets $B$ for which there is an $\ep'<\ep$ such that every point of $A$ is $\ep$-close to some point of $B$.
  \item $A \apart \cK$ means there is an $\ep>0$ such that for every $B\in\cK$ and $\ep'<\ep$ there is a point of $A$ that is at least $\ep'$-far from every point of $B$.
  \item $A \close \cK$ means for any $\ep>0$ there is a $B\in \cK$ and an $\ep'<\ep$ such that every point of $A$ is $\ep'$-close to a point of $B$.
  \end{itemize}
\end{eg}

Thus, the antithesis translation suggests that rather than taking one of neighborhoods, apartness, or nearness as primary, it is more natural to have all structures in parallel.
Of course, \cref{defn:unified-top} is rather unwieldy;
but \cref{defn:topology} is quite simple, suggesting it may be easier to just stay in affine logic.
In the next section we consider this possibility more seriously.

\section{Towards affine constructive mathematics}
\label{sec:lcm}

So far, we have viewed affine logic as a tool for producing definitions and theorems in intuitionistic logic, through the antithesis translation.
However, there are other reasons one might care about the ``affine constructive mathematics'' we have started developing in this paper.
One is that it admits other interesting models.

\begin{eg}
  Linear logicians are familiar with many $\ast$-autonomous categories, such as coherence spaces and phase spaces.
  As in \cref{eg:cplt-hyd}, any complete semicartesian $\ast$-autonomous category with Seely comonad yields an affine hyperdoctrine over $\bSet$.
  I expect there are ``realizability linear triposes'' coming from linear combinatory algebras~\cite{ahs:goi-lca,al:linrealiz}.
  In addition, Dialectica constructions~\cite{depaiva:dialectica-like} also act on fibrations~\cite{hofstra:dialectica}.
  However, many of these models are not semicartesian, hence move beyond affine logic to general linear logic.
\end{eg}

\begin{eg}\label{eg:bool-top}
  Any Boolean algebra is semicartesian and $\ast$-autonomous, with $\mand\logeq\aand$, $\mor\logeq\aor$, and $\oc P\defeq P$.
  Thus, linear logic also specializes directly to classical logic. 

  More generally, on a Boolean algebra we can take any meet-preserving comonad to be $\oc$, such as the interior operator of a topology acting on a powerset.
  Thus, any classical topological space $X$ gives rise to an affine tripos whose propositions are subsets of $X$, with the affirmative and refutative ones being open and closed respectively.
  This relates to the ``modal'' view of sheaves from~\cite{ak:top-fomodall,akk:topos-homodal}.
\end{eg}

\begin{eg}\label{eg:exotic}
  \textbf{\luk{}ukasiewicz logic} is a semicartesian $\ast$-autonomous structure on the unit interval $[0,1]$, with $\top=1$, $\bot=0$, and
  \begin{alignat*}{3}
    P\aand Q &= \min(P,Q) &\qquad
    \fa x. P(x) &= \textstyle\inf_x P(x) &\qquad
    P\mand Q &= \max(0,P+Q-1) \\
    P\aor Q &= \max(P,Q) &\qquad
    \ex x. P(x) &= \textstyle\sup_x P(x) &\qquad
    P\mor Q &= \min(1,P+Q)\\
    \nt P &= 1-P &\qquad &&
    P\imp Q &= \min(1,1-P+Q).
  \end{alignat*}
  It also admits a Seely comonad defined by 
  $\oc 1 = 1$ and $\oc P = 0$ for $P<1$.
  An $\L$-set in this model is precisely a \emph{metric space} with all distances $\le 1$.
  (The distance $d(x,y)$ is actually the \emph{inequality} $x\lneq y$.)
  It is strong iff it is an ultrametric space, and affirmative iff it is discrete. 
  Functions are nonexpansive maps, and anafunctions (\cref{thm:lin-anafun}) are nonexpansive maps between metric completions.
  The $\L$-set $\Omega=[0,1]$ has its usual metric $|x-y|$, and the function set $A\to B$ has the supremum metric.
  For a fixed affirmative $\L$-set $A$, the $\L$-subsets of $A$ are \emph{fuzzy sets} with universe $A$, with their usual induced metric.
  Finally, (closed upper) $\L$-cuts $x\ocin\cC$ are non-decreasing right-continuous functions $\dR\to\dR$; hence bounded $\L$-cuts are cumulative distribution functions of random variables, with Dedekind $\L$-reals corresponding to constant random variables.
\end{eg}

However, what about the philosophical constructivist, in the tradition of Bishop, say?
I believe that one can also motivate affine constructive mathematics on purely philosophical grounds;
what follows is one attempt.

We begin by agreeing with Brouwer's critique of excluded middle, ``$P$ or not $P$'', as a source of non-constructivity.
However, the classical mathematican's belief in this law is not contentless; one may say that the constructivist and the classicist are using the word ``or'' to mean different things.
The constructivist using intuitionistic logic expresses the classical mathematician's ``or'' as $\neg\neg (P\lor Q)$; but the classical mathematician may rebel against the implication that she is unconsciously inserting double negations everywhere.
A more even-handed approach is to stipulate \emph{both} kinds of ``or'' on an equal footing: the constructivist's $P\aor Q$ says that we know which of $P$ or $Q$ holds; while the classicist's $P\mor Q$ says\dots something else.

Before addressing exactly what it says, we consider negation.
Intuitionistically, $\neg P$ means that any proof of $P$ would lead to an absurdity.
But after this definition, one immediately observes that it is not very useful and should be avoided.
So why did we bother defining negation in that way?
A more useful notion of ``negation'' is the \emph{polar opposite} of a statement, i.e.\ the most natural and emphatic way to disprove it.
The opposite of ``every $x$ satisfies $P(x)$'', in this sense, is ``there is an $x$ that fails $P(x)$'': a respectable \emph{constructive} disproof of a universal claim should provide a counterexample.
Similarly, the opposite of ``$P$ and $Q$'' is ``either $P$ fails or $Q$ fails'', and so on.
This negation is involutive, with strict De Morgan duality for quantifiers, conjunctions, and disjunctions.

The most natural way that ``if $P$ then $Q$'' can fail is if $P$ is true and $Q$ is false.
But the opposite of ``$P$ and not $Q$'' is ``$Q$ or not $P$'', so the involutivity of negation means that the latter should be equivalent to ``if $P$ then $Q$''.
In particular, the tautology ``if $P$ then $P$'' is equivalent to ``$P$ or not $P$'', i.e.\ excluded middle.
Thus, the ``or'' appearing here must be the classical one $\mor$.
That is, ``if $P$ then $Q$'' (which we may as well start writing as $P\imp Q$) is equivalent to $\nt P \mor Q$.

This tells us what $P\mor Q$ means: it means $\nt P \imp Q$, i.e.\ if $P$ fails then $Q$ must be true.
But any sort of disjunction is symmetric, so $P\mor Q$ should also be equivalent to $\nt Q \imp P$.
Thus, contraposition must hold: $P\imp Q$ is equivalent to $\nt Q \imp\nt P$.
This, in turn, implies that we can do proofs by contradiction.

Proof by contradiction is generally considered non-constructive.
For instance, a constructive proof of ``there exists an $x$ such that $P(x)$'' ought to specify $x$, whereas proof by contradition seems to subvert this.
But does it really?
If we try to prove ``there exists an $x$ such that $P(x)$'' by contradiction, we would begin by assuming ``for all $x$, not $P(x)$''\dots and we can only \emph{use} that assumption by specifying an $x$!

Non-constructivity only enters if we use that assumption \emph{more than once}, giving different values of $x$, and derive a contradiction without determining which value of $x$ satisfies $P$.
Thus, we can remain constructive in the presence of proof by contradiction by imposing an ``affinity'' restriction that each hypothesis can be used at most once.\footnote{Or a ``linearity'' restriction that it must be used \emph{exactly} once; but this is harder to justify philosophically, since affinity is sufficient to ensure constructivity.}
This is essentially the content of Girard's comment:
\begin{quote}\small
  \dots take a proof of the existence or the disjunction property; we use the fact that the last rule used is an introduction, which we cannot do classically because of a possible contraction.
  Therefore, in the\dots intuitionistic case, $\types$ serves to mark a place where contraction\dots is forbidden\dots.
  Once we have recognized that the constructive features of intuitionistic logic come from the dumping of structural rules on a specific place in the sequents, we are ready to face the consequences of this remark: the limitation should be generalized to other rooms, i.e.\ weakening and contraction disappear.~\cite[p4]{girard:ll}
\end{quote}

We now let $\aand$ and $\mand$ be the De Morgan duals of $\aor$ and $\mor$, and calculate
\begin{equation*}
  (P \mand Q) \imp R
  \;\logeq\; ({\nt P \mor \nt Q}) \mor R
  \;\logeq\; \nt P \mor (\nt Q \mor R)
  \;\logeq\; P \imp (Q \imp R).
\end{equation*}
Thus, to maintain the ``deduction theorem'' that we prove $Q\imp R$ by proving $R$ with $Q$ as an extra hypothesis, we must implicitly combine hypotheses with $\mand$.

The behavior of $\aand$ can be deduced by duality: a hypothesis $P\aand Q$ may as well be used by contradiction, requiring us to show $\ntp{P\aand Q} \logeq \nt P \aor \nt Q$; and since this is the constructive ``or'' it requires us to either show $\nt P$ or $\nt Q$.
Thus, to use a hypothesis $P\aand Q$ we must either use $P$ or $Q$, but not both.
Note the utter reversal of the historical origin of the linear connectives:
\begin{quote}\small
  The most hidden of all linear connectives is \emph{par} [$\mor$], which came to light purely formally as the De Morgan dual of [$\mand$] and which can be seen as the effective part of a classical disjunction.~\cite[p5]{girard:ll}.
\end{quote}

Finally, the linearity/affinity restriction is sometimes too onerous.
For instance, the axioms of a group must be used many times in the proof of any theorem in group theory.
Since we are here regarding the affinity restriction as simply a \emph{syntactic discipline} to which we subject ourselves in order to maintain constructivity, we may allow ourselves to ignore it in certain cases as long as we keep track of where this happens and prevent ourselves in some other way from introducing nonconstructivity in those cases.
This is the purpose of the modality $\oc$: it marks hypotheses, like the axioms of a group, that we allow ourselves to use more than once.
The price we pay is that when \emph{checking} an axiom of the form $\oc P$, we cannot use proof by contradiction (or more precisely, if we try to do so, the hypothesis we get to contradict is not $\nt P$ but the weaker $\wn(\nt P) \defeq \ntp{\oc P}$).
But this is rarely bothersome: when was the last time you saw someone prove that something is a group by assuming that it isn't and deriving a contradiction?
(See also \cref{rmk:affirm-axioms}.)

Whether or not the reader finds the foregoing discussion convincing, I believe it proves that it is \emph{possible} to argue for affine logic, rather than intuitionistic logic, on philosophical constructivist grounds.
Ultimately, of course, the proof of the pudding is in the eating: whether affine constructive mathematics can stand on its own depends on how much useful mathematics can be developed purely in affine logic.
In this paper we have only scratched the surface by exploring a few basic definitions, with the antithesis translation as a guide for their correctness.

\bibliographystyle{alpha}
\bibliography{all}

\end{document}

%% file: decls.tex
\usepackage{amsmath,amssymb,stmaryrd,mathrsfs}

\def\definetac{\newif\iftac}    
\ifx\tactrue\undefined
  \definetac
  \ifx\state\undefined\tacfalse\else\tactrue\fi
\fi

\def\definebeamer{\newif\ifbeamer}
\ifx\beamertrue\undefined
  \definebeamer
  \ifx\uncover\undefined\beamerfalse\else\beamertrue\fi
\fi

\def\definecref{\newif\ifcref}
\ifx\creftrue\undefined
  \definecref
  \creffalse
\fi

\iftac\else\usepackage{amsthm}\fi
\usepackage{tikz,tikz-cd}
\usetikzlibrary{arrows}
\ifbeamer\else
  \usepackage{enumitem}
  \usepackage{xcolor}
  \definecolor{darkgreen}{rgb}{0,0.45,0} 
  \iftac\else\usepackage[pagebackref,colorlinks,citecolor=darkgreen,linkcolor=darkgreen]{hyperref}\fi
  \renewcommand*{\backref}[1]{}
  \renewcommand*{\backrefalt}[4]{({%
      \ifcase #1 not cited.%
            \or cited on p.~#2%
            \else cited on pp.~#2%
      \fi%
    })}
\fi
\usepackage{mathtools}          
\usepackage{graphics}           
\usepackage{ifmtarg}            
\usepackage{microtype}
\usepackage{braket}             
\let\setof\Set
\usepackage{url}                
\usepackage{xspace}             
\ifcref\usepackage[sort]{cleveref}\usepackage{aliascnt}\fi
\usepackage[status=draft,author=]{fixme}
\fxusetheme{color}
\usepackage{mathpartir}

\makeatletter
\let\ea\expandafter

\def\mdef#1#2{\ea\ea\ea\gdef\ea\ea\noexpand#1\ea{\ea\ensuremath\ea{#2}\xspace}}
\def\alwaysmath#1{\ea\ea\ea\global\ea\ea\ea\let\ea\ea\csname your@#1\endcsname\csname #1\endcsname
  \ea\def\csname #1\endcsname{\ensuremath{\csname your@#1\endcsname}\xspace}}

\DeclareRobustCommand\widecheck[1]{{\mathpalette\@widecheck{#1}}}
\def\@widecheck#1#2{%
    \setbox\z@\hbox{\m@th$#1#2$}%
    \setbox\tw@\hbox{\m@th$#1%
       \widehat{%
          \vrule\@width\z@\@height\ht\z@
          \vrule\@height\z@\@width\wd\z@}$}%
    \dp\tw@-\ht\z@
    \@tempdima\ht\z@ \advance\@tempdima2\ht\tw@ \divide\@tempdima\thr@@
    \setbox\tw@\hbox{%
       \raise\@tempdima\hbox{\scalebox{1}[-1]{\lower\@tempdima\box
\tw@}}}%
    {\ooalign{\box\tw@ \cr \box\z@}}}

\newcount\foreachcount

\def\foreachletter#1#2#3{\foreachcount=#1
  \ea\loop\ea\ea\ea#3\@alph\foreachcount
  \advance\foreachcount by 1
  \ifnum\foreachcount<#2\repeat}

\def\foreachLetter#1#2#3{\foreachcount=#1
  \ea\loop\ea\ea\ea#3\@Alph\foreachcount
  \advance\foreachcount by 1
  \ifnum\foreachcount<#2\repeat}

\def\definescr#1{\ea\gdef\csname s#1\endcsname{\ensuremath{\mathscr{#1}}\xspace}}
\foreachLetter{1}{27}{\definescr}
\def\definecal#1{\ea\gdef\csname c#1\endcsname{\ensuremath{\mathcal{#1}}\xspace}}
\foreachLetter{1}{27}{\definecal}
\def\definebold#1{\ea\gdef\csname b#1\endcsname{\ensuremath{\mathbf{#1}}\xspace}}
\foreachLetter{1}{27}{\definebold}
\def\definebb#1{\ea\gdef\csname d#1\endcsname{\ensuremath{\mathbb{#1}}\xspace}}
\foreachLetter{1}{27}{\definebb}
\def\definefrak#1{\ea\gdef\csname f#1\endcsname{\ensuremath{\mathfrak{#1}}\xspace}}
\foreachletter{1}{9}{\definefrak}
\foreachletter{10}{27}{\definefrak}
\foreachLetter{1}{27}{\definefrak}
\def\definesf#1{\ea\gdef\csname i#1\endcsname{\ensuremath{\mathsf{#1}}\xspace}}
\foreachletter{1}{6}{\definesf}
\foreachletter{7}{14}{\definesf}
\foreachletter{15}{20}{\definesf}
\foreachletter{21}{27}{\definesf}
\foreachLetter{1}{27}{\definesf}
\def\definebar#1{\ea\gdef\csname #1bar\endcsname{\ensuremath{\overline{#1}}\xspace}}
\foreachLetter{1}{27}{\definebar}
\foreachletter{1}{8}{\definebar} 
\foreachletter{9}{15}{\definebar} 
\foreachletter{16}{27}{\definebar}
\def\definetil#1{\ea\gdef\csname #1til\endcsname{\ensuremath{\widetilde{#1}}\xspace}}
\foreachLetter{1}{27}{\definetil}
\foreachletter{1}{27}{\definetil}
\def\definehat#1{\ea\gdef\csname #1hat\endcsname{\ensuremath{\widehat{#1}}\xspace}}
\foreachLetter{1}{27}{\definehat}
\foreachletter{1}{27}{\definehat}
\def\definechk#1{\ea\gdef\csname #1chk\endcsname{\ensuremath{\widecheck{#1}}\xspace}}
\foreachLetter{1}{27}{\definechk}
\foreachletter{1}{27}{\definechk}
\def\defineul#1{\ea\gdef\csname u#1\endcsname{\ensuremath{\underline{#1}}\xspace}}
\foreachLetter{1}{27}{\defineul}
\foreachletter{1}{27}{\defineul}

\def\autofmt@n#1\autofmt@end{\mathrm{#1}}
\def\autofmt@b#1\autofmt@end{\mathbf{#1}}
\def\autofmt@d#1#2\autofmt@end{\mathbb{#1}\mathsf{#2}}
\def\autofmt@c#1#2\autofmt@end{\mathcal{#1}\mathit{#2}}
\def\autofmt@s#1#2\autofmt@end{\mathscr{#1}\mathit{#2}}
\def\autofmt@f#1\autofmt@end{\mathsf{#1}}
\def\autofmt@k#1\autofmt@end{\mathfrak{#1}}
\def\autofmt@u#1\autofmt@end{\underline{\smash{\mathsf{#1}}}}
\def\autofmt@U#1\autofmt@end{\underline{\underline{\smash{\mathsf{#1}}}}}
\def\autofmt@h#1\autofmt@end{\widehat{#1}}
\def\autofmt@r#1\autofmt@end{\overline{#1}}
\def\autofmt@t#1\autofmt@end{\widetilde{#1}}
\def\autofmt@k#1\autofmt@end{\check{#1}}

\def\auto@drop#1{}
\def\autodef#1{\ea\ea\ea\@autodef\ea\ea\ea#1\ea\auto@drop\string#1\autodef@end}
\def\@autodef#1#2#3\autodef@end{%
  \ea\def\ea#1\ea{\ea\ensuremath\ea{\csname autofmt@#2\endcsname#3\autofmt@end}\xspace}}
\def\autodefs@end{blarg!}
\def\autodefs#1{\@autodefs#1\autodefs@end}
\def\@autodefs#1{\ifx#1\autodefs@end%
  \def\autodefs@next{}%
  \else%
  \def\autodefs@next{\autodef#1\@autodefs}%
  \fi\autodefs@next}

\DeclareFontFamily{U}{mathb}{\hyphenchar\font45}
\DeclareFontShape{U}{mathb}{m}{n}{
      <5> <6> <7> <8> <9> <10> gen * mathb
      <10.95> mathb10 <12> <14.4> <17.28> <20.74> <24.88> mathb12
      }{}
\DeclareSymbolFont{mathb}{U}{mathb}{m}{n}
\DeclareFontSubstitution{U}{mathb}{m}{n}
\DeclareMathSymbol{\boxleft}      {2}{mathb}{"68}
\DeclareMathSymbol{\boxright}     {2}{mathb}{"69}
\DeclareMathSymbol{\boxtop}       {2}{mathb}{"6A}
\DeclareMathSymbol{\boxbot}       {2}{mathb}{"6B}

\DeclareSymbolFont{bbold}{U}{bbold}{m}{n}
\DeclareSymbolFontAlphabet{\mathbbb}{bbold}

\newcommand{\done}{\ensuremath{\mathbbb{1}}\xspace}

\mdef\delbar{\overline{\partial}}

\newcommand{\inv}{^{-1}}

\mdef\hf{\textstyle\frac12 }
\mdef\thrd{\textstyle\frac13 }
\mdef\qtr{\textstyle\frac14 }

\newcommand{\op}{^{\mathrm{op}}}

\mdef\Id{\mathrm{Id}}
\mdef\id{\mathrm{id}}
\alwaysmath{ell}
\alwaysmath{infty}

\alwaysmath{odot}
\def\frc#1/#2.{\frac{#1}{#2}}   
\mdef\ten{\mathrel{\otimes}}

\mdef\sqten{\mathrel{\boxtimes}}

\DeclareFontFamily{U}{mathb}{\hyphenchar\font45}
\DeclareFontShape{U}{mathb}{m}{n}{
      <5> <6> <7> <8> <9> <10> gen * mathb
      <10.95> mathb10 <12> <14.4> <17.28> <20.74> <24.88> mathb12
      }{}
\DeclareSymbolFont{mathb}{U}{mathb}{m}{n}
\DeclareFontSubstitution{U}{mathb}{m}{n}
\DeclareMathSymbol{\dotplus}       {2}{mathb}{"00}
\DeclareMathSymbol{\dotdiv}        {2}{mathb}{"01}
\DeclareMathSymbol{\dottimes}      {2}{mathb}{"02}
\DeclareMathSymbol{\divdot}        {2}{mathb}{"03}
\DeclareMathSymbol{\udot}          {2}{mathb}{"04}
\DeclareMathSymbol{\square}        {2}{mathb}{"05}
\DeclareMathSymbol{\Asterisk}      {2}{mathb}{"06}
\DeclareMathSymbol{\bigast}        {1}{mathb}{"06}
\DeclareMathSymbol{\coAsterisk}    {2}{mathb}{"07}
\DeclareMathSymbol{\bigcoast}      {1}{mathb}{"07}
\DeclareMathSymbol{\circplus}      {2}{mathb}{"08}
\DeclareMathSymbol{\pluscirc}      {2}{mathb}{"09}
\DeclareMathSymbol{\convolution}   {2}{mathb}{"0A}
\DeclareMathSymbol{\divideontimes} {2}{mathb}{"0B}
\DeclareMathSymbol{\blackdiamond}  {2}{mathb}{"0C}
\DeclareMathSymbol{\sqbullet}      {2}{mathb}{"0D}
\DeclareMathSymbol{\bigstar}       {2}{mathb}{"0E}
\DeclareMathSymbol{\bigvarstar}    {2}{mathb}{"0F}
\DeclareMathSymbol{\corresponds}   {3}{mathb}{"1D}
\DeclareMathSymbol{\updownarrows}          {3}{mathb}{"D6}
\DeclareMathSymbol{\downuparrows}          {3}{mathb}{"D7}
\DeclareMathSymbol{\Lsh}                   {3}{mathb}{"E8}
\DeclareMathSymbol{\Rsh}                   {3}{mathb}{"E9}
\DeclareMathSymbol{\dlsh}                  {3}{mathb}{"EA}
\DeclareMathSymbol{\drsh}                  {3}{mathb}{"EB}
\DeclareMathSymbol{\looparrowdownleft}     {3}{mathb}{"EE}
\DeclareMathSymbol{\looparrowdownright}    {3}{mathb}{"EF}
\DeclareMathSymbol{\curvearrowleftright}   {3}{mathb}{"F2}
\DeclareMathSymbol{\curvearrowbotleft}     {3}{mathb}{"F3}
\DeclareMathSymbol{\curvearrowbotright}    {3}{mathb}{"F4}
\DeclareMathSymbol{\curvearrowbotleftright}{3}{mathb}{"F5}
\DeclareMathSymbol{\leftsquigarrow}        {3}{mathb}{"F8}
\DeclareMathSymbol{\rightsquigarrow}       {3}{mathb}{"F9}
\DeclareMathSymbol{\leftrightsquigarrow}   {3}{mathb}{"FA}
\DeclareMathSymbol{\lefttorightarrow}      {3}{mathb}{"FC}
\DeclareMathSymbol{\righttoleftarrow}      {3}{mathb}{"FD}
\DeclareMathSymbol{\uptodownarrow}         {3}{mathb}{"FE}
\DeclareMathSymbol{\downtouparrow}         {3}{mathb}{"FF}
\DeclareMathSymbol{\varhash}       {0}{mathb}{"23}
\DeclareMathSymbol{\nsqsubseteq}    {3}{mathb}{"86}
\DeclareMathSymbol{\nsqsupseteq}    {3}{mathb}{"87}
\DeclareMathSymbol{\sqSubset}       {3}{mathb}{"94}
\DeclareMathSymbol{\sqSupset}       {3}{mathb}{"95}
\DeclareMathSymbol{\nsqSubset}      {3}{mathb}{"96}
\DeclareMathSymbol{\nsqSupset}      {3}{mathb}{"97}

\let\imp\Rightarrow

\mdef\we{\overset{\sim}{\longrightarrow}}
\mdef\leftwe{\overset{\sim}{\longleftarrow}}

\let\xto\xrightarrow

\def\rightarrowtailfill@{\arrowfill@{\Yright\joinrel\relbar}\relbar\rightarrow}
\newcommand\xrightarrowtail[2][]{\ext@arrow 0055{\rightarrowtailfill@}{#1}{#2}}

\def\twoheadrightarrowfill@{\arrowfill@{\relbar\joinrel\relbar}\relbar\twoheadrightarrow}
\newcommand\xtwoheadrightarrow[2][]{\ext@arrow 0055{\twoheadrightarrowfill@}{#1}{#2}}

\def\slashedarrowfill@#1#2#3#4#5{%
  $\m@th\thickmuskip0mu\medmuskip\thickmuskip\thinmuskip\thickmuskip
   \relax#5#1\mkern-7mu%
   \cleaders\hbox{$#5\mkern-2mu#2\mkern-2mu$}\hfill
   \mathclap{#3}\mathclap{#2}%
   \cleaders\hbox{$#5\mkern-2mu#2\mkern-2mu$}\hfill
   \mkern-7mu#4$%
}
\def\rightslashedarrowfill@{%
  \slashedarrowfill@\relbar\relbar\mapstochar\rightarrow}
\newcommand\xslashedrightarrow[2][]{%
  \ext@arrow 0055{\rightslashedarrowfill@}{#1}{#2}}
\mdef\hto{\xslashedrightarrow{}}
\mdef\htoo{\xslashedrightarrow{\quad}}

\def\toiso{\xto{\smash{\raisebox{-.5mm}{$\scriptstyle\sim$}}}}

\def\jd#1{\@jd#1\ej}
\def\@jd#1|-#2\ej{\@@jd#1,,\;\vdash\;\left(#2\right)}
\def\@@jd#1,{\@ifmtarg{#1}{\let\next=\relax}{\left(#1\right)\let\next=\@@@jd}\next}
\def\@@@jd#1,{\@ifmtarg{#1}{\let\next=\relax}{,\,\left(#1\right)\let\next=\@@@jd}\next}
\def\jdm#1{\@jdm#1\ej}
\def\@jdm#1|-#2\ej{\@@jd#1,,\\\vdash\;\left(#2\right)}
\def\cm{,}

\long\def\my@drawfill#1#2;{%
\@skipfalse
\fill[#1,draw=none] #2;
\@skiptrue
\draw[#1,fill=none] #2;
}
\newif\if@skip
\newcommand{\skipit}[1]{\if@skip\else#1\fi}
\newcommand{\drawfill}[1][]{\my@drawfill{#1}}

\newcounter{nodemaker}
\def\renumberthms#1{}

\newif\ifhyperref
\@ifpackageloaded{hyperref}{\hyperreftrue}{\hyperreffalse}
\iftac
  \let\your@state\state
  \def\state#1{\my@state#1}
  \def\my@state#1.{\gdef\currthmtype{#1}\your@state{#1.}}
  \let\your@staterm\staterm
  \def\staterm#1{\my@staterm#1}
  \def\my@staterm#1.{\gdef\currthmtype{#1}\your@staterm{#1.}}
  \let\@defthm\newtheorem
  \def\switchtotheoremrm{\let\@defthm\newtheoremrm}
  \def\defthm#1#2#3{\@defthm{#1}{#2}} 
  \let\your@section\section
  \def\section{\gdef\currthmtype{section}\your@section}
  \def\currthmtype{}
  \ifhyperref
    \def\autoref#1{\ref*{label@name@#1}~\ref{#1}}
  \else
    \def\autoref#1{\ref{label@name@#1}~\ref{#1}}
  \fi
  \AtBeginDocument{%
    \let\old@label\label%
    \def\label#1{%
      {\let\your@currentlabel\@currentlabel%
        \edef\@currentlabel{\currthmtype}%
        \old@label{label@name@#1}}%
      \old@label{#1}}
  }
  \let\cref\autoref
\else\ifcref
  \def\defthm#1#2#3{%
    \newaliascnt{#1}{thm}
    \newtheorem{#1}[#1]{#2}
    \aliascntresetthe{#1}
    \crefname{#1}{#2}{#3}
    \ea\def\ea\renumberthms\ea##\ea1\ea{\renumberthms{##1}\numberwithin{#1}{##1}}
  }
\else
  \ifhyperref
    \def\defthm#1#2#3{
      \newtheorem{#1}{#2}[section]%
      \expandafter\def\csname #1autorefname\endcsname{#2}%
      \expandafter\let\csname c@#1\endcsname\c@thm}
  \else
    \def\defthm#1#2#3{\newtheorem{#1}[thm]{#2}} 
    \ifx\SK@label\undefined\let\SK@label\label\fi
    \let\old@label\label
    \let\your@thm\@thm
    \def\@thm#1#2#3{\gdef\currthmtype{#3}\your@thm{#1}{#2}{#3}}
    \def\currthmtype{}
    \def\label#1{{\let\your@currentlabel\@currentlabel\def\@currentlabel%
        {\currthmtype~\your@currentlabel}%
        \SK@label{#1@}}\old@label{#1}}
    \def\autoref#1{\ref{#1@}}
  \fi
  \let\cref\autoref
\fi\fi

\newtheorem{thm}{Theorem}[section]
\ifcref
  \crefname{thm}{Theorem}{Theorems}
\else
  
\fi
\defthm{cor}{Corollary}{Corollaries}
\defthm{prop}{Proposition}{Propositions}
\defthm{lem}{Lemma}{Lemmas}
\defthm{sch}{Scholium}{Scholia}
\defthm{assume}{Assumption}{Assumptions}
\defthm{claim}{Claim}{Claims}
\defthm{conj}{Conjecture}{Conjectures}
\defthm{hyp}{Hypothesis}{Hypotheses}
\defthm{ax}{Axiom}{Axioms}
\iftac\switchtotheoremrm\else\theoremstyle{definition}\fi
\defthm{defn}{Definition}{Definitions}
\defthm{redefn}{Re-Definition}{Re-Definitions}
\defthm{defnI}{$\I$-Definition}{$\I$-Definitions}
\defthm{defnL}{$\L$-Definition}{$\L$-Definitions}
\defthm{notn}{Notation}{Notations}
\iftac\switchtotheoremrm\else\theoremstyle{remark}\fi
\defthm{rmk}{Remark}{Remarks}
\defthm{eg}{Example}{Examples}
\defthm{egs}{Examples}{Examples}
\defthm{ex}{Exercise}{Exercises}
\defthm{ceg}{Counterexample}{Counterexamples}

\ifcref
  \crefformat{section}{\S#2#1#3}
  \Crefformat{section}{Section~#2#1#3}
  \crefrangeformat{section}{\S\S#3#1#4--#5#2#6}
  \Crefrangeformat{section}{Sections~#3#1#4--#5#2#6}
  \crefmultiformat{section}{\S\S#2#1#3}{ and~#2#1#3}{, #2#1#3}{ and~#2#1#3}
  \Crefmultiformat{section}{Sections~#2#1#3}{ and~#2#1#3}{, #2#1#3}{ and~#2#1#3}
  \crefrangemultiformat{section}{\S\S#3#1#4--#5#2#6}{ and~#3#1#4--#5#2#6}{, #3#1#4--#5#2#6}{ and~#3#1#4--#5#2#6}
  \Crefrangemultiformat{section}{Sections~#3#1#4--#5#2#6}{ and~#3#1#4--#5#2#6}{, #3#1#4--#5#2#6}{ and~#3#1#4--#5#2#6}
  \crefformat{appendix}{Appendix~#2#1#3}
  \Crefformat{appendix}{Appendix~#2#1#3}
  \crefrangeformat{appendix}{Appendices~#3#1#4--#5#2#6}
  \Crefrangeformat{appendix}{Appendices~#3#1#4--#5#2#6}
  \crefmultiformat{appendix}{Appendices~#2#1#3}{ and~#2#1#3}{, #2#1#3}{ and~#2#1#3}
  \Crefmultiformat{appendix}{Appendices~#2#1#3}{ and~#2#1#3}{, #2#1#3}{ and~#2#1#3}
  \crefrangemultiformat{appendix}{Appendices~#3#1#4--#5#2#6}{ and~#3#1#4--#5#2#6}{, #3#1#4--#5#2#6}{ and~#3#1#4--#5#2#6}
  \Crefrangemultiformat{appendix}{Appendices~#3#1#4--#5#2#6}{ and~#3#1#4--#5#2#6}{, #3#1#4--#5#2#6}{ and~#3#1#4--#5#2#6}
  \crefformat{subappendix}{\S#2#1#3}
  \Crefformat{subappendix}{Section~#2#1#3}
  \crefrangeformat{subappendix}{\S\S#3#1#4--#5#2#6}
  \Crefrangeformat{subappendix}{Sections~#3#1#4--#5#2#6}
  \crefmultiformat{subappendix}{\S\S#2#1#3}{ and~#2#1#3}{, #2#1#3}{ and~#2#1#3}
  \Crefmultiformat{subappendix}{Sections~#2#1#3}{ and~#2#1#3}{, #2#1#3}{ and~#2#1#3}
  \crefrangemultiformat{subappendix}{\S\S#3#1#4--#5#2#6}{ and~#3#1#4--#5#2#6}{, #3#1#4--#5#2#6}{ and~#3#1#4--#5#2#6}
  \Crefrangemultiformat{subappendix}{Sections~#3#1#4--#5#2#6}{ and~#3#1#4--#5#2#6}{, #3#1#4--#5#2#6}{ and~#3#1#4--#5#2#6}
  \crefname{part}{Part}{Parts}
  \crefname{figure}{Figure}{Figures}
\fi

\iftac
  \let\qed\endproof
  \let\your@endproof\endproof
  \def\my@endproof{\your@endproof}
  \def\endproof{\my@endproof\gdef\my@endproof{\your@endproof}}
  \def\qedhere{\tag*{\endproofbox}\gdef\my@endproof{\relax}}
\fi

\iftac
  \def\pr@@f[#1]{\subsubsection*{\sc #1.}}
\fi

\def\thmqedhere{\expandafter\csname\csname @currenvir\endcsname @qed\endcsname}

\ifbeamer\else

\fi

\ifbeamer\else
  \setitemize[1]{leftmargin=2em}
  \setenumerate[1]{leftmargin=*}
\fi

\iftac
  \let\c@equation\c@subsection
\else
  \let\c@equation\c@thm
\fi
\numberwithin{equation}{section}

\ifcref\else
  \@ifpackageloaded{mathtools}{\mathtoolsset{showonlyrefs,showmanualtags}}{}
\fi

\alwaysmath{alpha}
\alwaysmath{beta}
\alwaysmath{gamma}
\alwaysmath{Gamma}
\alwaysmath{delta}
\alwaysmath{Delta}
\alwaysmath{epsilon}
\mdef\ep{\varepsilon}
\alwaysmath{zeta}
\alwaysmath{eta}
\alwaysmath{theta}
\alwaysmath{Theta}
\alwaysmath{iota}
\alwaysmath{kappa}
\alwaysmath{lambda}
\alwaysmath{Lambda}
\alwaysmath{mu}
\alwaysmath{nu}
\alwaysmath{xi}
\alwaysmath{pi}
\alwaysmath{rho}
\alwaysmath{sigma}
\alwaysmath{Sigma}
\alwaysmath{tau}
\alwaysmath{upsilon}
\alwaysmath{Upsilon}
\alwaysmath{phi}
\alwaysmath{Pi}
\alwaysmath{Phi}
\mdef\ph{\varphi}
\alwaysmath{chi}
\alwaysmath{psi}
\alwaysmath{Psi}
\alwaysmath{omega}
\alwaysmath{Omega}

\let\Om\Omega

\makeatother
